\documentclass{amsart}
\usepackage{amssymb,xspace,cmap, bbm, stmaryrd}

\newtheorem*{thma}{Theorem~A}
\newtheorem*{thmb}{Theorem~B}
\newtheorem*{thmc}{Theorem~C}
\newtheorem*{thmd}{Theorem~D}

\newtheorem{thm}{Theorem}[section]
\newtheorem{cor}[thm]{Corollary}
\newtheorem{conj}[thm]{Conjecture}
\newtheorem{prop}[thm]{Proposition}
\newtheorem{fact}[thm]{Fact}
\newtheorem{lemma}[thm]{Lemma}
\newtheorem{claim}{Claim}[thm]
\newtheorem{subclaim}{Subclaim}[claim]

\theoremstyle{definition}
\newtheorem{defn}[thm]{Definition}
\newtheorem{q}[thm]{Question}

\theoremstyle{remark}
\newtheorem{remark}{Remark}

\newenvironment{cproof}{\paragraph{\emph{Proof.\,}}}{\hfill{$\boxtimes$}\par\vspace{2mm}}
\newenvironment{scproof}{\paragraph{\emph{Proof.\,}}}{\hfill{$\boxminus$}\par\vspace{2mm}}

\newcommand*\axiomfont[1]{\textsf{\textup{#1}}\xspace}

\newcommand\zfc{\axiomfont{ZFC}}

\newcommand\sq{\sqsubseteq}
\newcommand\s{\subseteq}
\newcommand\conc{{^\smallfrown}}
\newcommand\br{\blacktriangleright}

\renewcommand{\restriction}{\mathbin\upharpoonright}
\renewcommand\mid{\mathrel{|}\allowbreak}
\newcommand\Mid{\mathrel{}\middle|\mathrel{}}

\newcommand\diagonal{\bigtriangleup}

\DeclareMathOperator{\spec}{Cspec}

\DeclareMathOperator{\crit}{crit}

\DeclareMathOperator{\id}{id}

\DeclareMathOperator{\ind}{ind}
\DeclareMathOperator{\reg}{Reg}

\DeclareMathOperator{\card}{Card}
\DeclareMathOperator{\cf}{cf}
\DeclareMathOperator{\cl}{cl}
\DeclareMathOperator{\Tr}{Tr}
\DeclareMathOperator{\tr}{tr}
\DeclareMathOperator{\im}{Im}
\DeclareMathOperator{\otp}{otp}
\DeclareMathOperator{\dom}{dom}
\DeclareMathOperator{\add}{Add}
\DeclareMathOperator{\acc}{acc}
\DeclareMathOperator{\nacc}{nacc}

\DeclareMathOperator{\U}{U}
\DeclareMathOperator{\pr}{Pr}
\DeclareMathOperator{\ssup}{ssup}

\author{Chris Lambie-Hanson}
\address{Department of Mathematics and Applied Mathematics, Virginia Commonwealth University,
Richmond, VA 23284, USA}
\urladdr{http://people.vcu.edu/~cblambiehanso}

\author{Assaf Rinot}
\address{Department of Mathematics, Bar-Ilan University, Ramat-Gan 5290002, Israel.}
\urladdr{http://www.assafrinot.com}

\thanks{The second author is partially supported by the European Research Council (grant agreement ERC-2018-StG 802756) and by the Israel Science Foundation (grant agreement 2066/18)}
\begin{document}
\title[Knaster and friends II]{Knaster and friends II: The C-sequence number}
\begin{abstract}
  Motivated by a characterization of weakly compact cardinals due to Todorcevic,
  we introduce a new cardinal characteristic,
  the $C$-sequence number, which can be seen as a measure of the compactness
  of a regular uncountable cardinal. We prove a number of $\zfc$ and
  independence results about the $C$-sequence number and its relationship with
  large cardinals, stationary reflection, and square principles. We then
  introduce and study the more general $C$-sequence spectrum and uncover some tight
  connections between the $C$-sequence spectrum and the strong coloring
  principle $\U(\ldots)$, introduced in Part~I of this series.
\end{abstract}
\date{\today}

\maketitle
\tableofcontents
\newpage

\section{Introduction}

A common theme in modern set theory, running through the study of large cardinals,
combinatorial set theory, and inner model theory, is the investigation into the
compactness properties of uncountable cardinals and the extent to which large
cardinal properties can hold at ``small" cardinals.
Two prominent compactness properties of an uncountable cardinal $\kappa$,
each of which is equivalent to weak compactness, are:
\begin{enumerate}
  \item[(P1)] the partition relation $\kappa \rightarrow (\kappa)^n_\theta$ holds for all $n < \omega$ and $\theta < \kappa$;
  \item[(P2)] every $C$-sequence $\langle C_\beta \mid \beta < \kappa \rangle$ is
  trivial (see Theorem \ref{todorcevic_thm} for a precise statement).
\end{enumerate}
Another compactness property which holds at every weakly compact cardinal $\kappa$ but, by \cite{MR3620068},
does not characterize weak compactness is:
\begin{enumerate}
  \item[(P3)] the class of $\kappa$-Knaster posets is closed under $\nu$-support
  products for every $\nu < \kappa$.
\end{enumerate}

This series of papers, in which the present work forms Part~II, is devoted
to the investigation into the ways in which graded families of strong negations of
the above properties (P1)--(P3) can serve to measure the incompactness of regular
uncountable cardinals, and to the exploration of the network of implications and independences
that exist among these families of incompactness properties. Part~I of the series
\cite{paper34} is concerned with strong negations of properties (P1) and (P3),
and their connection. Let us recall some of the relevant definitions and results
from \cite{paper34} here. In what follows and throughout the paper, $\kappa$
denotes a regular uncountable cardinal, and $\chi,\theta,$ and $\mu$ denote cardinals $\le\kappa$.

\begin{defn}[\cite{paper34}]
  $\U(\kappa, \mu, \theta, \chi)$ asserts the existence of a coloring $c:[\kappa]^2
  \rightarrow \theta$ such that for every $\chi' < \chi$, every family
  $\mathcal{A} \s [\kappa]^{\chi'}$ consisting of $\kappa$-many pairwise disjoint
  sets, and every $i < \theta$, there exists $\mathcal{B} \in [\mathcal{A}]^\mu$
  such that $\min(c[a \times b]) > i$ for all $(a, b) \in [\mathcal{B}]^2$.\footnote{We refer the reader to the `Notation
and conventions' subsection at the end of this introduction for explanation of
any nonstandard notation, and in particular
for our conventions regarding elements of $[\mathcal{B}]^2$ for a set $\mathcal{B}$.}
\end{defn}

\begin{fact}[\cite{paper34}]\label{knasterlemma}
  Suppose that $\chi, \theta$ are regular and that $\kappa$ is $({<}\chi)$-inaccessible.
  For every coloring $c:[\kappa]^2\rightarrow\theta$ witnessing $\U(\kappa,\mu,\theta,\chi)$,
  there exists a corresponding poset $\mathbb P$ such that
  \begin{enumerate}
    \item $\mathbb P$ is well-met and $\chi$-directed closed with greatest lower bounds;
    \item if $\mu=2$, then $\mathbb P^\tau$ is $\kappa$-cc for all $\tau<\min(\{\chi, \theta\})$;
    \item if $\mu=\kappa$, then $\mathbb P^\tau$ has precaliber $\kappa$ for all $\tau<\min(\{\chi, \theta\})$;
    \item $\mathbb P^\theta$ is not $\kappa$-cc.
  \end{enumerate}
\end{fact}

Much of Part~I is devoted to analyzing situations in which
$\U(\ldots)$ necessarily holds and, moreover, is witnessed
by \emph{closed} colorings. As a corollary, we obtained:
\begin{fact}[\cite{paper34}] If the class of $\kappa$-Knaster posets is closed under $\omega$-powers,
then $\kappa$ is inaccessible and every stationary subset of $\kappa$ reflects at an inaccessible cardinal.
\end{fact}
 In this paper, which can be read largely independently of \cite{paper34},
we take our investigations of $\U(\ldots)$ in a slightly different direction,
by bringing property~(P2) into the picture.
We introduce and study a cardinal characteristic for regular, uncountable cardinal $\kappa$,
which we denote by $\chi(\kappa)$ and call the \emph{$C$-sequence number} of $\kappa$.
This new cardinal characteristic is connected with our coloring principle in the sense that
it serves as a natural candidate for the fourth parameter of $\U(\kappa,\kappa,\theta,\chi)$,
especially, in the case $\theta=\omega$.
Yet, as time passes by, it becomes evident that it is of interest in its own right.

To motivate our definition, let us recall Todorcevic's characterization of weakly compact cardinals.
\begin{defn}\label{defcsequence}
  A \emph{$C$-sequence} over $\kappa$ is a sequence $\langle C_\beta \mid \beta
  < \kappa \rangle$ such that, for all $\beta < \kappa$, $C_\beta$ is a closed subset of $\beta$ with $\sup(C_\beta)=\sup(\beta)$.
\end{defn}

\begin{thm}[Todorcevic, {\cite[Theorem~1.8]{MR908147}}] \label{todorcevic_thm}
  For every strongly inaccessible cardinal $\kappa$, the following are equivalent.
  \begin{enumerate}
    \item $\kappa$ is weakly compact.
    \item For every $C$-sequence $\langle C_\beta\mid \beta<\kappa\rangle$ there
      exist $\Delta\in[\kappa]^\kappa$ and $b:\kappa\rightarrow\kappa$ such that
      $\Delta\cap\alpha=C_{b(\alpha)}\cap\alpha$ for every $\alpha<\kappa$.
  \end{enumerate}
\end{thm}

We are now ready for the main definition of the paper, which, in light of the preceding theorem,
suggests a way of measuring how far
an inaccessible cardinal $\kappa$ is from being weakly compact. As we
shall see, though, it is of interest for successors of singular cardinals as well.

\begin{defn}[The $C$-sequence number of $\kappa$] \label{c_seq_num_def}
  If $\kappa$ is weakly compact, then let $\chi(\kappa):=0$. Otherwise, let
  $\chi(\kappa)$ denote the least (finite or infinite) cardinal $\chi\le\kappa$
  such that, for every $C$-sequence $\langle C_\beta\mid\beta<\kappa\rangle$,
  there exist $\Delta\in[\kappa]^\kappa$ and $b:\kappa\rightarrow[\kappa]^{\chi}$
  with $\Delta\cap\alpha\s\bigcup_{\beta\in b(\alpha)}C_\beta$
  for every $\alpha<\kappa$.
\end{defn}

With this definition in hand, it is also natural to consider the \emph{$C$-sequence spectrum} of a
cardinal $\kappa$, which can provide additional information.

\begin{defn}[The $C$-sequence spectrum of $\kappa$]\hfill
  \begin{enumerate}
    \item For every $C$-sequence $\vec{C} = \langle C_\beta \mid \beta < \kappa \rangle$,
      $\chi(\vec{C})$ is the least cardinal $\chi \leq \kappa$ such that there exist
      $\Delta \in [\kappa]^\kappa$ and $b:\kappa\rightarrow[\kappa]^\chi$ with
      $\Delta\cap\alpha\s\bigcup_{\beta\in b(\alpha)}C_\beta$
      for every $\alpha<\kappa$.
    \item $\spec(\kappa) := \{\chi(\vec{C}) \mid \vec{C}$ is a $C$-sequence over $\kappa\} \setminus \omega$.
  \end{enumerate}
\end{defn}

In this paper, we present a number of both $\zfc$ results and consistency results regarding the
$C$-sequence number and $C$-sequence spectrum, in particular exploring
how these concepts interact
with large cardinal notions and square principles. Among these results are the
following.

\begin{thma}
\begin{enumerate}
\item If the class of $\kappa$-Knaster posets is closed under $\omega$-powers, then $\chi(\kappa)\le1$.
\item If $\chi(\kappa)\le 1$, then $\kappa$ is greatly Mahlo.
\item If $\chi(\kappa)>1$, then $\min(\spec(\kappa))=\omega$ and $\max(\spec(\kappa))=\chi(\kappa)$.
\item Every finite family of stationary subsets of $E^\kappa_{>\chi(\kappa)}$ reflects simultaneously.
\end{enumerate}
\end{thma}
\begin{proof} (1) follows from Fact~\ref{knasterlemma} and Lemma~\ref{cnontrivial}.
(2) follows from Lemma~\ref{lemma29}(3).
(3) follows from Lemma~\ref{prop53}(3), Corollary~\ref{cor519} and Theorem~\ref{realized}.
(4) is Lemma~\ref{prop53}(4).
\end{proof}
\begin{thmb} Any of the following implies that $\reg(\kappa)\s\spec(\kappa)$:
\begin{enumerate}
\item $\square(\kappa,{<}\omega)$ holds;
\item $\kappa$ is a successor of a regular cardinal;
\item $\kappa$ is an inaccessible cardinal which is not Mahlo.
\end{enumerate}
\end{thmb}
\begin{proof} This is Corollary~\ref{cor521}.
\end{proof}

\begin{thmc}
\begin{enumerate}
\item If $\chi(\kappa)=0$, then for every $\chi\in\{1\}\cup\reg(\kappa+1)$, there is a  $({<}\kappa)$-distributive forcing extension in which $\chi(\kappa)=\chi$ and $\reg(\chi)\s \spec(\kappa)$.
\item It is consistent that $\chi(\aleph_{\omega+1})=\aleph_\omega$.
Assuming the consistency of a supercompact cardinal, it is also consistent that $\chi(\aleph_{\omega+1})=\omega$.
\item If $\lambda$ is a singular limit of strongly compact cardinals,
then $\chi(\lambda^+)=\cf(\lambda)$ and $\reg(\cf(\lambda))\s\spec(\lambda^+)$.
\item If $\lambda$ is a singular limit of supercompact cardinals, then, for every $\chi\in\reg(\lambda)\setminus\cf(\lambda)$, there is a ${<}\lambda$-distributive forcing extension in which $\chi(\lambda^+)=\chi$ and $\reg(\chi)\s\spec(\lambda^+)$.
\end{enumerate}
\end{thmc}
\begin{proof} (1) follows from Corollary~\ref{cor27},  Theorem~\ref{kunenmodel} and Theorem~\ref{thm51}.
(2) follows from Corollary~\ref{cor27} and Theorem~\ref{downtoalephw}.
(3) follows from Theorem~\ref{t55} and Theorem~\ref{thm524}(1).
(4) follows from Theorem~\ref{thm511} and Corollary~\ref{cor515}.
\end{proof}

We then go back to the theme of property~(P1) and uncover an unexpected connection between
the $C$-sequence spectrum and the third (and fourth) parameter of the principle $\U(\ldots)$.

\begin{thmd}
  For every $\theta\in\reg(\kappa)$, the following are equivalent:
\begin{enumerate}
\item $\theta\in\spec(\kappa)$;
\item There is a closed witness to $\U(\kappa,\kappa,\theta,\theta)$.
\end{enumerate}
 \end{thmd}
\begin{proof} This follows from Corollary~\ref{connection}. \end{proof}

\subsection{Organization of this paper}
 In Section~2, we introduce the $C$-sequence number
and prove some basic results regarding it, in particular settling its behavior at successors of regular cardinals.
In Section~3, we present a number of consistency
results concerning the $C$-sequence numbers, the principle $\U(\ldots)$, and related matters,
at both inaccessible cardinals and successors of singular cardinals. These consistency results
will, among other things, indicate that certain results both from this paper and from
\cite{paper34} are sharp. In Section~4, we study the $C$-sequence spectrum, which
provides more information about a cardinal $\kappa$ than the $C$-sequence number alone.
In Section~5, we investigate connections between the $C$-sequence spectrum and $\U(\ldots)$,
in particular proving the theorems that yield the corollaries listed above.
In Section~6, we present some open questions and closing remarks.

\subsection{Notation and conventions}
  When constructing a $C$-sequence $\langle C_\beta \mid \beta < \kappa \rangle$, we
  automatically let $C_{\beta + 1} := \{\beta\}$ for all $\beta < \kappa$ unless
  we explicitly note otherwise.

We say that $\kappa$ is \emph{$({<}\chi)$-inaccessible} iff, for all $\nu<\chi$ and
$\lambda<\kappa$, $\lambda^{\nu}<\kappa$.
We denote by  $H_\Upsilon$
the collection of all sets of hereditary cardinality less $\Upsilon$,
where $\Upsilon$ is a regular cardinal
sufficiently large to satisfy that all objects of interest are in $H_\Upsilon$.

$\reg$ denotes the class of infinite regular
cardinals, and $\reg(\kappa)$ denotes $\reg \cap \kappa$. It will frequently
be convenient for us to refer to the cardinal $\sup(\reg(\kappa))$, where
$\kappa$ is an uncountable cardinal. Note that if $\kappa$ is a limit cardinal,
then $\sup(\reg(\kappa)) = \kappa$, and if $\kappa$ is a successor cardinal, then
$\sup(\reg(\kappa))$ is the immediate predecessor of $\kappa$. $E^\kappa_\chi$
denotes the set $\{\alpha < \kappa \mid \cf(\alpha) = \chi\}$, and
$E^\kappa_{\geq \chi}$, $E^\kappa_{>\chi}$, $E^\kappa_{\neq\chi}$, etc.\ are defined analogously.

For a set of ordinals $a$, we write $\ssup(a) := \sup\{\alpha + 1 \mid
\alpha \in a\}$, $\acc^+(a) := \{\alpha < \ssup(a) \mid \sup(a \cap \alpha) = \alpha > 0\}$,
$\acc(a) := a \cap \acc^+(a)$, $\nacc(a) := a \setminus \acc(a)$,
and $\cl(a):= a\cup\acc^+(a)$.
For sets of ordinals $a$ and $b$, we write $a < b$ if, for all $\alpha \in a$
and all $\beta \in b$, we have $\alpha < \beta$.
For a set of ordinals $a$ and an ordinal $\beta$, we write
$a < \beta$ instead of $a < \{\beta\}$ and $\beta < a$ instead of $\{\beta\} < a$.
For a set of ordinals $A$, $\Tr(A)$ denotes the set $\{ \beta \in E^{\ssup(A)}_{>\omega}
\mid A \cap \beta \text{ is stationary in } \beta\}$.

For any set $\mathcal A$, we write
$[\mathcal A]^\chi:=\{ \mathcal B\s\mathcal A\mid |\mathcal B|=\chi\}$ and
$[\mathcal A]^{<\chi}:=\{\mathcal B\s\mathcal A\mid |\mathcal B|<\chi\}$.
In particular, $[\mathcal{A}]^2$ consists of all unordered pairs from $\mathcal{A}$.
In some scenarios, we will also be interested in ordered pairs from $\mathcal{A}$.
In particular, if $\mathcal{A}$ is either an ordinal or a collection of sets
of ordinals, then we will abuse notation and write $(a,b) \in [\mathcal{A}]^2$
to mean $\{a,b\} \in [\mathcal{A}]^2$ and $a < b$.

\section{The $C$-sequence number}
In this section, we initiate our study of the $C$-sequence number.
Our first lemma will be useful in our later analysis and asserts that sets $\Delta$ and functions $b$ as in Definition~\ref{c_seq_num_def}
can always be chosen to have certain nice properties. In what follows, let us call a function $b:\kappa \rightarrow
[\kappa]^\chi$ \emph{progressive} if $\min(b(\alpha))\ge\alpha$ for all $\alpha < \kappa$.

\begin{lemma} \label{lemma53}
  Suppose that
  \begin{itemize}
  \item $\langle C_\beta\mid\beta<\kappa\rangle$ is a $C$-sequence;
  \item $\chi< \sup(\reg(\kappa))$ is a cardinal;
  \item $A,\Delta',\Gamma\in[\kappa]^\kappa$ are sets;
  \item $b':A\rightarrow[\Gamma]^{\le\chi}$ is a function satisfying $\Delta'\cap\alpha\s\bigcup_{\beta\in b'(\alpha)}C_\beta$ for all $\alpha\in A$.
  \end{itemize}
  Then the following two statements hold.
  \begin{enumerate}
    \item For every stationary $\Sigma\s E^\kappa_{>\chi}$, there exist $\Delta_0,\Delta_1\in[\kappa]^\kappa$,
      and a progressive function $b:\kappa \rightarrow [\Gamma]^{\le\chi}$
      such that
      \begin{itemize}
      \item $\Delta_0\s\Delta'$;
      \item $\acc^+(\Delta_1)\cap E^\kappa_{>\chi}\s\Delta_1$;
      \item $\Delta_0\cap \alpha \s \bigcup_{\beta \in b(\alpha)} C_\beta$ for all $\alpha < \kappa$;
      \item $\Delta_1\cap \alpha \s \bigcup_{\beta \in b(\alpha)}\acc(C_\beta)$ for all $\alpha < \kappa$;
      \item $\{\alpha\in\Sigma\mid \forall \beta\in b(\alpha)[\sup(C_\beta\cap\alpha)=\alpha]\}$ is stationary.
      \end{itemize}
    \item If $\chi$ is a positive integer, then there exist $\Delta\in[\Delta']^\kappa$
      and a function $b:\kappa \rightarrow \Gamma$
      such that $\Delta \cap \alpha\s C_{b(\alpha)}$ for all  $\alpha < \kappa$.
  \end{enumerate}
\end{lemma}

\begin{proof}
  It is easy to extend $\dom(b')$ to the whole of $\kappa$ by sending each $\alpha<\kappa$ to $b'(\min(A\setminus\alpha))$. Thus, we shall assume that $A=\kappa$.

  (1)   Let $\Sigma\s E^\kappa_{>\chi}$ be stationary. For all $\alpha\in\Sigma$, we have $\cf(\alpha)>\chi\ge|b'(\alpha)|$,
  and hence we may define a regressive function $f:\Sigma\rightarrow\kappa$ by letting, for all $\alpha\in\Sigma$,
  $$f(\alpha):=\sup(\{ \sup(C_\beta\cap\alpha)\mid \beta\in b'(\alpha), ~ \sup(C_\beta\cap\alpha)<\alpha\}).$$

  By Fodor's Lemma, let us fix an  $\epsilon < \kappa$ for which $T:=f^{-1}\{\epsilon\}$ is stationary.
  For all $\alpha<\kappa$, set $\alpha':=\min(T\setminus\alpha)$.
  Then, consider the progressive function $b:\kappa\rightarrow[\Gamma]^{\le\chi}$,
  defined by letting, for all $\alpha < \kappa$,
  $$b(\alpha) := \{ \beta\in b'(\alpha')\mid \sup(C_\beta\cap\alpha')=\alpha'\}.$$
  As $\alpha=\alpha'$ for all $\alpha\in T$, we have that $\{\alpha\in\Sigma\mid \forall \beta\in b(\alpha)[\sup(C_\beta\cap\alpha)=\alpha]\}$ covers the stationary set $T$.

  To see that $\Delta_0:=\Delta'\setminus(\epsilon+1)$ and $\Delta_1:=
  E^\kappa_{>\chi}\cap \acc^+(\Delta'\setminus \epsilon)$ are as sought,
  fix an arbitrary $\alpha<\kappa$.

  $\br$ Let $\delta\in\Delta_0\cap\alpha$ be arbitrary. As $\delta\in \Delta'\cap\alpha'$, we may find some $\beta\in b'(\alpha')$ such that $\delta\in C_\beta$.
  In particular, $\sup(C_\beta\cap\alpha')\ge\delta>\epsilon=f(\alpha')$, meaning that $\sup(C_\beta\cap\alpha')=\alpha'$, so $\beta\in b(\alpha)$.

  $\br$ Let $\delta\in\Delta_1\cap\alpha$ be arbitrary.
  Fix $d\s \Delta'\setminus\epsilon$ with $\ssup(d)=\delta$.
  As $d\s \Delta'\cap\alpha'$, we have $d\s \bigcup_{\beta \in b'(\alpha')}C_\beta$.
  But $\cf(\delta)>|b'(\alpha')|$,
  and hence we may find some $\beta\in b'(\alpha')$ such that $\sup(d\cap C_\beta)=\delta$.
  In particular, $\sup(C_\beta\cap\alpha')\ge\delta>\epsilon=f(\alpha')$, meaning that $\sup(C_\beta\cap\alpha')=\alpha'$.
  Consequently, $\beta\in b(\alpha)$ and $\delta\in\acc(C_\beta)$.

\medskip

  (2) By Clause~(1), we may assume that $\Delta'$ is a club and that $b'$ is progressive.
  To avoid trivialities, suppose also that $\chi>1$, and let $n:=\chi-1$.

  Clearly, if $D:=\{\delta\in\Delta'\mid \forall \alpha\in(\delta,\kappa)
  [\delta\in\bigcap_{\beta\in b'(\alpha)}C_\beta]\}$
  is cofinal in $\kappa$,  then we may simply take $\Delta:=D$, and then any
  $b:\kappa \rightarrow \Gamma$ satisfying $b(\alpha) \in b'(\alpha)$ for all $\alpha < \kappa$, will do.
  Thus, suppose that $D$ is bounded below $\kappa$.
  We will find an $\epsilon<\kappa$ and a function $b:\kappa\rightarrow[\Gamma]^n$
  such that $(\Delta'\setminus\epsilon)\cap\alpha\s\bigcup_{\beta\in b(\alpha)}C_\beta$ for all $\alpha<\kappa$.
  The result will then follow by induction.

  For each $\delta\in \Delta'\setminus D$, find $\alpha_\delta \in (\delta, \kappa)$
  and $\beta_\delta \in b'(\alpha_\delta)$ such that $\delta\notin C_{\beta_\delta}$.
  In particular, $\sup(C_{\beta_\delta}\cap\delta)<\delta$.
  Use Fodor's Lemma to find an $\varepsilon < \kappa$
  and a stationary $T \subseteq \Delta'\setminus D$ such that $\sup(C_{\beta_\delta}
  \cap \delta) = \varepsilon$ for all $\delta\in T$.

  Let $\epsilon:=\varepsilon+1$,
  and define $b:\kappa \rightarrow [\Gamma]^n$ by letting, for all $\alpha < \kappa$,
  $$b(\alpha) :=
  b'\left(\alpha_{\min(T\setminus \alpha)}\right) \setminus \{\beta_{\min(T\setminus \alpha)}\}.$$
  To see that $\epsilon$ and $b$ are as sought, fix arbitrary $\alpha<\kappa$ and $\delta\in(\Delta'\setminus\epsilon)\cap\alpha$.
  Set $\delta':=\min(T\setminus\alpha)$. As $\delta<\alpha\le\delta'<\alpha_{\delta'}$,
  we have $\delta\in\Delta\cap\alpha_{\delta'}$, and we may pick some $\beta\in b'(\alpha_{\delta'})$ such that $\delta\in C_\beta$.
  We know that $\sup(C_{\beta_{\delta'}}\cap\delta')=\varepsilon<\epsilon\le\delta<\alpha\le\delta'$, and hence $\delta\notin C_{\beta_{\delta'}}$.
  So $\beta\neq \beta_{\delta'}$, and $\beta\in b(\alpha)$.
\end{proof}

We are now ready to record a few basic facts about $\chi(\kappa)$.

\begin{lemma} \label{prop53}
  The $C$-sequence number satisfies the following properties.
  \begin{enumerate}
    \item $\chi(\kappa) \leq \sup(\reg(\kappa))$.
    \item For every infinite cardinal $\lambda$, we have $\cf(\lambda) \leq
      \chi(\lambda^+) \leq \lambda$. In particular, if $\lambda$ is regular,
      then $\chi(\lambda^+) = \lambda$.
    \item If $\chi(\kappa) > 1$, then $\chi(\kappa) \geq \omega$.
    \item Every finite family of stationary subsets of $E^\kappa_{>\chi(\kappa)}$ reflects simultaneously.
    \item If $V=L$, then $\chi(\kappa)\in\{0,\sup(\reg(\kappa))\}$.
  \end{enumerate}
\end{lemma}
\begin{proof}
  (1) To avoid trivialities, suppose that $\sup(\reg(\kappa))<\kappa$.
  Then $\kappa=\lambda^+$ for $\lambda:=\sup(\reg(\kappa))$.
  To see that $\chi(\kappa)\le\lambda$, let $\vec C=\langle C_\beta\mid\beta<\kappa\rangle$
  be an arbitrary $C$-sequence, and set $\Delta:=\bigcup_{\beta\in\acc(\kappa)}C_\beta$.
  Clearly, $\Delta\in[\kappa]^\kappa$. Now, for every $\alpha<\kappa$, as $|\Delta\cap\alpha|\le\lambda$,
  it is trivial to find $b(\alpha)\in[\kappa]^{\lambda}$ such that
  $\Delta\cap\alpha\s\bigcup_{\beta\in b(\alpha)}C_\beta$.

  (2) By Clause~(1), it suffices to verify that $\cf(\lambda) \leq \chi(\lambda^+)$.
  To this end, let $\kappa:=\lambda^+$, and let $\vec{C} = \langle C_\beta\mid \alpha < \kappa \rangle$
  be a $C$-sequence  such that $\otp(C_\beta)\le\lambda$ for all $\beta < \kappa$.
  Fix $\Delta\in[\kappa]^\kappa$ and $b:\kappa\rightarrow[\kappa]^{\chi(\kappa)}$
  such that $\Delta\cap\alpha\s\bigcup_{\beta\in b(\alpha)}C_\beta$
  for every $\alpha<\kappa$.
  Let $\alpha\in\Delta$ be the unique ordinal to satisfy $\otp(\Delta \cap \alpha)=\lambda^2$,
  so $\otp(\bigcup_{\beta \in b(\alpha)}C_\beta)\ge\lambda^2$.
  As the union of fewer than $\cf(\lambda)$-many sets, each of order type less than $\lambda^2$,
  has order type less than $\lambda^2$, it follows that $\chi(\kappa)=|b(\alpha)|\ge\cf(\lambda)$.

  (3) This follows immediately from Lemma~\ref{lemma53}(2).

  (4) Suppose for sake of contradiction that $n$ is a positive integer and $\langle S_i \mid i< n\rangle$ is a
  sequence of stationary subsets of $E^\kappa_{>\chi(\kappa)}$ that does not reflect simultaneously.
  Let $\langle C_\beta\mid \beta < \kappa \rangle$
  be a $C$-sequence such that, for all $\beta<\kappa$, for some $i(\beta)< n$, $\acc(C_\beta)\cap S_{i(\beta)} = \emptyset$.
  For each $i< n$, let $T_i:=\{ \beta\in\acc(\kappa)\mid i(\beta)=i\}$.

  As $S_0$ is a stationary subset of $E^\kappa_{>\chi(\kappa)}$, we know that $\chi(\kappa)<\sup(\reg(\kappa))$.
  We shall show that, for all $i\le n$, there exist $\Delta_i\in[\kappa]^\kappa$
  and a function $b_i:\kappa\rightarrow[\kappa]^{\chi(\kappa)}$
  such that, for all $\alpha<\kappa$,
  \begin{itemize}
  \item $\Delta_i\cap \alpha \s \bigcup_{\beta \in b_i(\alpha)} C_\beta$;
  \item $b_i(\alpha)\cap\bigcup_{j<i} T_j=\emptyset$;
  \item $\min(b_i(\alpha))>\alpha$.
  \end{itemize}
  As $\bigcup_{j< n}T_j=\acc(\kappa)$, this will mean that $b_n(\alpha) \subseteq
  \nacc(\kappa)$, and hence $|\bigcup_{\beta \in b_n(\alpha)} C_\beta|=\chi(\kappa)$ for all $\alpha<\kappa$,
  contradicting the fact that $(\chi(\kappa))^+<\kappa$.

  The proof is by induction on $i$.

  $\br$ For $i=0$, we do the following.
  As $\chi(\kappa)<\sup(\reg(\kappa))$,
  by the definition of $\chi(\kappa)$ and by Lemma~\ref{lemma53},
  we may find  $\Delta\in[\kappa]^\kappa$ and a progressive function $b:\kappa \rightarrow [\kappa]^{\chi(\kappa)}$
  such that $\Delta \cap \alpha \s \bigcup_{\beta \in b(\alpha)} C_\beta$ for all $\alpha < \kappa$.
  Now, let $\Delta_0:=\Delta$ and define $b_0$ by letting $b_0(\alpha):=b(\alpha+1)$ for all $\alpha<\kappa$.
  Clearly, $\Delta_0$ and $b_0$ are as sought.

  $\br$ Suppose that $i< n$ and that $\Delta_i$, $b_i$ have been defined.
  Define $f:S_{i}\rightarrow\kappa$ by setting, for all $\alpha \in S_i$,
  $$f(\alpha):=\sup\{ \sup(C_\beta\cap\alpha)\mid \beta\in b_i(\alpha)\cap T_{i}\}.$$
  For all $\alpha\in S_{i}$ and $\beta\in b_i(\alpha)\cap T_{i}$,
  we have $\beta>\alpha$ and $\alpha\notin\acc(C_\beta)$, so $\sup(C_\beta\cap\alpha)<\alpha$.
  $f$ is therefore regressive, and we may fix some $\epsilon<\kappa$ for which $A:=f^{-1}\{\epsilon\}$ is stationary.
  Let $\Delta_{i+1}:=\Delta_i\setminus(\epsilon+1)$ and define $b_{i+1}$ by letting, for all $\alpha<\kappa$,
  $b_{i+1}(\alpha):=b_i(\min(A\setminus\alpha))\setminus T_{i}$.

  To see that $\Delta_{i+1}$ and $b_{i+1}$ are as sought,
  fix an arbitrary $\alpha<\kappa$.
  Set $\alpha':=\min(A\setminus\alpha)$, so that $\min(b_{i+1}(\alpha))\ge\min(b_i(\alpha'))>\alpha'\ge\alpha$.
  Finally, let $\tau\in\Delta_{i+1}\cap\alpha$ be arbitrary.
  As $\tau\in\Delta_{i+1}\cap\alpha\s\Delta_i\cap\alpha'$,
  we may pick $\beta\in b_i(\alpha')$ such that $\tau\in C_\beta$.
  Then $\sup(C_\beta\cap\alpha')\ge\tau>\epsilon=f(\alpha')$,
  so $\beta\notin T_i$, and hence $\beta\in b_{i+1}(\alpha)$.

  (5) If $V = L$ and $\kappa$ is a regular uncountable cardinal that is not
  weakly compact, then, for every $\chi\in\reg(\kappa)$, $E^\kappa_\chi$ contains a
  stationary subset that does not reflect. The result now follows from Clause~(4).
\end{proof}

\begin{remark}
\begin{itemize}
\item Clause~(2) is sharp, as witnessed by Theorem~\ref{t55} below.
\item Clause~(4) is sharp, as, by Clause~(2), for every $\kappa$ which is the successor of a regular cardinal,
$E^\kappa_{\chi(\kappa)}$ is stationary and non-reflecting.
It is also sharp in another sense;
by Theorem~\ref{thm51} below (using $\theta:=\omega$),
it is consistent that for some strongly inaccessible cardinal $\kappa$, $\chi(\kappa)=\omega$,
and there is a family $\mathcal S$ consisting of countably many stationary subsets of $E^\kappa_{>\omega}$
such that $\mathcal S$ does not reflect simultaneously.
\end{itemize}
\end{remark}

\begin{cor}
  If $\kappa=\lambda^+$ and $\square_{\lambda,{<}\cf(\lambda)}$ holds, then $\chi(\kappa) = \lambda$.
  In particular, it is consistent for $\chi(\kappa)$ to be a singular cardinal.
\end{cor}
\begin{proof}
  Suppose that $\kappa=\lambda^+$ and $\square_{\lambda,{<}\cf(\lambda)}$
  holds. By Clause~(2) of Lemma~\ref{prop53}, $\cf(\lambda)\le\chi(\kappa)$.
  In particular, we may assume that $\lambda$ is singular.
  By the proof of \cite[Lemma~2.2]{MR2811288}, $\square_{\lambda,{<}\cf(\lambda)}$
  implies that any stationary subset of $\lambda^+$ contains a stationary
  subset that does not reflect.
  Thus, by Clause~(4) of Lemma~\ref{prop53}, $\chi(\kappa)=\sup(\reg(\kappa))=\lambda$.
\end{proof}

\begin{remark}
  The preceding result is sharp in the sense that $\square_{\lambda, \cf(\lambda)}$
  does not imply $\chi(\lambda^+) = \lambda$; this follows from Theorem~\ref{thm513} below.
\end{remark}

\begin{cor}\label{cor27} If $\square(\kappa,{<}\omega)$ holds, then $\chi(\kappa)=\sup(\reg(\kappa))$.
\end{cor}
\begin{proof} By \cite[Theorem~2.8]{hayut_lh}, $\square(\kappa,{<}\omega)$  implies that any stationary set $S\s\kappa$
contains two stationary subsets $S_0$ and $S_1$ such that $\Tr(S_0)\cap\Tr(S_1)=\emptyset$.
Now, appeal to Lemma~\ref{prop53}(4).
\end{proof}

We next argue that, if $\kappa$ is inaccessible and $\chi(\kappa)$ is small, then
$\kappa$ must have a high degree of Mahloness. Before we precisely state and prove
this result, let us recall some facts about canonical functions and the Mahlo hierarchy.

\begin{defn}
  Suppose that $\eta < \kappa^+$. A function $f : \kappa \rightarrow \kappa$ is a \emph{canonical function on $\kappa$ of
  rank $\eta$} if there exists a surjection $e : \kappa \rightarrow \eta$ and a club
  $C \s \kappa$ such that, for all $\alpha \in C$, we have $f(\alpha) = \otp(f[\alpha])$.
  We let $\mathbf{f}^\kappa_\eta$ denote the set of all canonical functions on $\kappa$ of rank
  $\eta$.
\end{defn}

For a function $f:\kappa\rightarrow\kappa$, we let $f+1$ denote the unique function $g:\kappa\rightarrow\kappa$
satisfying $g(\alpha)=f(\alpha)+1$ for all $\alpha<\kappa$.
For two functions $f,g:\kappa \rightarrow \kappa$, we write
$f =^* g$ if there is a club $C \s \kappa$ such that $f(\alpha) = g(\alpha)$ for all $\alpha \in C$.
It is easily verified that, if $f,g:\kappa \rightarrow \kappa$, $\eta < \kappa^+$, and
$f \in \mathbf{f}^\kappa_\eta$, then $g \in \mathbf{f}^\kappa_\eta$ if and only if $g = ^* f$.
We now recall some basic facts about canonical functions. For proofs of these facts, see \cite[\S3]{MR763898} and \cite[\S1]{MR2833150}.

\begin{fact} \label{canonical_facts}
  \begin{enumerate}
    \item For all $\eta < \kappa$, the constant function taking value $\eta$ is in $\mathbf{f}^\kappa_\eta$.
    \item The identity function is in $\mathbf{f}^\kappa_\kappa$.
    \item If $\eta < \kappa^+$ and $f \in \mathbf{f}^\kappa_\eta$, then $f+1 \in \mathbf{f}^\kappa_{\eta + 1}$.
    \item If $\eta < \kappa^+$ and $f \in \mathbf{f}^\kappa_\eta$, then there is a club $C \s \kappa$ such that,
      for all uncountable $\alpha \in C\cap\reg(\kappa)$, $f \restriction \alpha \in
      \mathbf{f}^\alpha_{f(\alpha)}$.
    \item Suppose that $\xi < \eta < \kappa^+$ and $e:\kappa \rightarrow \eta$ is a surjection. Define $f:\kappa \rightarrow \kappa$
      by letting $f(\alpha) := \otp(e[\alpha] \cap \xi)$ for all $\alpha < \kappa$. Then $f \in \mathbf{f}^\kappa_\xi$.
  \end{enumerate}
\end{fact}

Let us now recall the definition of the Mahlo hierarchy. The definition is by recursion.

\begin{defn}
  Assume that, for all regular uncountable $\alpha < \kappa$ and all $\eta \leq \alpha^+$, we have
  already specified what it means for $\alpha$ to be $\eta$-Mahlo. By recursion on $\eta \leq \kappa^+$,
  we say that $\kappa$ is $\eta$-Mahlo if
  \begin{itemize}
    \item ($\eta = 0$) $\kappa$ is inaccessible;
    \item ($\eta = \xi + 1$) the set $\{\alpha < \kappa \mid \alpha \text{ is } f(\alpha)\text{-Mahlo}\}$ is stationary in $\kappa$ for some (any) $f \in \mathbf{f}^\kappa_\xi$;
    \item ($\eta$ limit) $\kappa$ is $\xi$-Mahlo for all $\xi < \eta$.
  \end{itemize}
  $\kappa$ is said to be \emph{greatly Mahlo} if it is $\kappa^+$-Mahlo.
\end{defn}
\begin{remark} It is also possible to define greatly Mahlo cardinals without explicitly mentioning canonical functions; see the discussion following Theorem~3 of \cite{MkSh:367}.
\end{remark}

\begin{prop} \label{diag_refl_prop}
  Suppose that $\kappa$ is inaccessible and, for every sequence $\langle S_i \mid i < \kappa \rangle$ of stationary
  subsets of $\kappa$, there exists an inaccessible $\beta < \kappa$ such that $S_i \cap \beta$ is stationary in
  $\beta$ for all $i < \beta$. Then $\kappa$ is greatly Mahlo.
\end{prop}
\begin{proof}
  We prove by induction on $\eta \leq \kappa^+$ that $\kappa$ is $\eta$-Mahlo. Note that our assumption in fact implies
  that, for every sequence $\langle S_i \mid i < \kappa \rangle$ of stationary subsets of $\kappa$, there are
  \emph{stationarily many} inaccessible $\beta < \kappa$ such that $S_i \cap \beta$ is stationary in
  $\beta$ for all $i < \beta$. In particular, it follows that $\kappa$ is 1-Mahlo. In addition, if $\eta \leq \kappa^+$
  is a limit ordinal and we have shown that $\kappa$ is $\xi$-Mahlo for all $\xi < \eta$, it follows by definition
  that $\kappa$ is $\eta$-Mahlo. There are three remaining cases to consider.

  \textbf{Case 1: $\eta = \xi + 2$ for some $\xi < \kappa^+$ and $\kappa$ is $(\xi + 1)$-Mahlo.} Fix $f \in \mathbf{f}^\kappa_\xi$.
  By Clause~(4) of Fact~\ref{canonical_facts}, there is a club $C \s \kappa$ such that, for every inaccessible $\beta \in C$,
  $f \restriction \beta \in \mathbf{f}^\beta_{f(\beta)}$. Let
  $S := \{\alpha < \kappa \mid \alpha \text{ is } f(\alpha)\text{-Mahlo}\}$. Since $\kappa$ is $(\xi + 1)$-Mahlo,
  $S$ is stationary in $\kappa$. By our hypothesis, there are stationarily many inaccessible $\beta \in C$
  such that $S \cap \beta$ is stationary in $\beta$. It follows that each such $\beta$ is $f(\beta) + 1$-Mahlo.
  By Clause~(3) of Fact~\ref{canonical_facts}, $f + 1 \in \mathbf{f}^\kappa_{\xi + 1}$. It follows that
  $\kappa$ is $(\xi + 2)$-Mahlo.

  \textbf{Case 2: $\eta = \xi + 1$ for some limit ordinal $\xi \in E^{\kappa^+}_{<\kappa}$ and $\kappa$
  is $\xi$-Mahlo.} Let $\langle\xi_i \mid i < \cf(\xi)\rangle$ be an increasing, continuous sequence of ordinals
  converging to $\xi$. Fix a surjection $e:\kappa \rightarrow \xi$, and let $f:\kappa \rightarrow \kappa$
  be defined by letting $f(\alpha) := \otp(e[\alpha])$ for all $\alpha < \kappa$. For $i < \cf(\xi)$,
  let $f_i:\kappa \rightarrow \kappa$ be defined by letting $f_i(\alpha) := \otp(e[\alpha] \cap \xi_i)$
  for all $\alpha < \kappa$. Then $f \in \mathbf{f}^\kappa_\xi$ and, by Clause~(5) of Fact~\ref{canonical_facts},
  $f_i \in \mathbf{f}^\kappa_{\xi_i}$ for all $i < \cf(\xi)$. Note that, for all $\beta < \kappa$,
  $f(\beta) = \sup\{f_i(\beta) \mid i < \cf(\xi)\}$.

  For all $i < \cf(\xi)$, let $S_i := \{\alpha < \kappa \mid \alpha \text{ is }f_i(\alpha)\text{-Mahlo}\}$.
  Since $\kappa$ is $\xi$-Mahlo, it follows that each $S_i$ is stationary in $\kappa$. Using our hypothesis
  and Clause~(4) of Fact~\ref{canonical_facts}, we see that there are stationarily many inaccessible $\beta < \kappa$ such that,
  for all $i < \cf(\xi)$,
  \begin{itemize}
    \item $f_i \restriction \beta \in \mathbf{f}^\beta_{f_i(\beta)}$;
    \item $S_i \cap \beta$ is stationary in $\beta$.
  \end{itemize}
  Each such $\beta$ is therefore $f_i(\beta)$-Mahlo for all $i < \cf(\xi)$, so, since $f(\beta) = \sup\{f_i(\beta) \mid i < \cf(\xi)\}$,
  $\beta$ is in fact $f(\beta)$-Mahlo. It follows that $\kappa$ is $(\xi+1)$-Mahlo.

  \textbf{Case 3: $\eta = \xi + 1$ for some $\xi \in E^{\kappa^+}_\kappa$ and $\kappa$ is $\xi$-Mahlo.}
  Let $\langle\xi_i \mid i < \kappa\rangle$, $e:\kappa \rightarrow \xi$, $f:\kappa \rightarrow \kappa$,
  and $\langle f_i \mid i < \kappa \rangle$ be defined as in Case 2. Again, we have $f(\beta) =
  \sup\{f_i(\beta) \mid i < \kappa\}$ for all $\beta < \kappa$.
  But notice that $D := \{\beta < \kappa \mid e[\beta] \s \xi_\beta\}$ is a club in $\kappa$ and,
  for all $\beta \in D$, we in fact have
  $f(\beta) = \sup\{f_i(\beta) \mid i < \beta\}$.

  For all $i < \kappa$, let $S_i := \{\alpha < \kappa \mid \alpha \text{ is }f_i(\alpha)\text{-Mahlo}\}$.
  Since $\kappa$ is $\xi$-Mahlo, $S_i$ is stationary in $\kappa$ for all $i < \kappa$. Using our hypothesis
  and Clause~(4) of Fact~\ref{canonical_facts} (and taking a diagonal intersection of the clubs obtained
  thence), we see that there are stationarily many inaccessible $\beta \in D$ such that, for all
  $i < \beta$,
  \begin{itemize}
    \item $f_i \restriction \beta \in \mathbf{f}^\beta_{f_i(\beta)}$;
    \item $S_i \cap \beta$ is stationary in $\beta$.
  \end{itemize}
  Each such $\beta$ is therefore $f_i(\beta)$-Mahlo for all $i < \beta$, so, since $f(\beta) = \sup\{f_i(\beta) \mid i < \beta\}$,
  $\beta$ is in fact $f(\beta)$-Mahlo. It follows that $\kappa$ is $(\xi+1)$-Mahlo.
\end{proof}

\begin{lemma}\label{lemma29}
  \begin{enumerate}
    \item If $\kappa$ is inaccessible and $\chi(\kappa) < \kappa$, then $\kappa$ is $\omega$-Mahlo.
    \item If $\chi(\kappa) = 1$, then, for every sequence $\langle S_i \mid i < \kappa \rangle$ of
      stationary subsets of $\kappa$, there exists an inaccessible $\beta < \kappa$ such
      that $S_i \cap \beta$ is stationary in $\beta$ for all $i < \beta$.
    \item If $\chi(\kappa) = 1$, then $\kappa$ is greatly Mahlo.
  \end{enumerate}
\end{lemma}
\begin{proof}
  (1) Suppose that $\kappa$ is inaccessible and $\chi(\kappa) < \kappa$.
  We first show that $\kappa$ is Mahlo (i.e., 1-Mahlo). Suppose not, and let $D$ be a club in
  $\kappa$ consisting of singular cardinals. Let $\vec{C} = \langle C_\beta \mid
  \beta < \kappa \rangle$ be a $C$-sequence such that, for all $\beta\in\acc(\kappa)$,
  \begin{itemize}
    \item $\otp(C_\beta)=\cf(\beta)$;
    \item if  $\beta \in D$, then $\min(C_\beta)=\cf(\beta)$;
    \item if $\beta \notin D$, then $\min(C_\beta)=\sup(D \cap \beta)$.
  \end{itemize}

  Note that for all $\beta\in D$ and nonzero $\alpha\le\beta$,
  we have $\otp(C_\beta\cap\alpha)<\alpha$, since, otherwise, $\otp(C_\beta\cap\alpha)=\alpha$ and $\min(C_\beta)<\alpha$, so that $$\cf(\beta)=\min(C_\beta)<\alpha=\otp(C_\beta\cap\alpha)\le\otp(C_\beta)=\cf(\beta),$$ which gives a contradiction.
  It follows that, for all $\alpha\in D$ and $\beta\in[\alpha,\kappa)$, $\otp(C_\beta \cap \alpha) <\alpha$.

  As $\chi(\kappa)<\sup(\reg(\kappa))$, we appeal to Lemma~\ref{lemma53} and fix $\Delta \in [\kappa]^\kappa$ and a progressive function $b:\kappa \rightarrow
  [\kappa]^{\chi(\kappa)}$ such that $\Delta \cap \alpha \s \bigcup_{\beta \in b(\alpha)}
  C_\beta$ for all $\alpha < \kappa$. For all $\alpha \in D \cap E^\kappa_{>\chi(\kappa)}$,
  let $\epsilon_\alpha := \sup\{\otp(C_\beta \cap \alpha) \mid \beta \in b(\alpha)\}$,
  and note that $\epsilon_\alpha < \alpha$. We can therefore apply Fodor's Lemma
  to find a stationary $S \s D \cap E^\kappa_{>\chi(\kappa)}$ and an $\epsilon < \kappa$
  such that $\epsilon_\alpha = \epsilon$ for all $\alpha \in S$.
  Find a large enough $\alpha \in S$
  such that $|\Delta \cap \alpha| > \max\{|\epsilon|, \chi(\kappa)\}$.
  As $\Delta \cap \alpha \s \bigcup_{\beta \in b(\alpha)}C_\beta \cap \alpha$,
  we get a contradiction to the fact that $|\bigcup_{\beta \in b(\alpha)}C_\beta \cap \alpha|
  \leq \max\{|\epsilon|, \chi(\kappa)\} < |\Delta \cap \alpha|$. Thus, $\kappa$
  is Mahlo.

  To finish, suppose there is a positive integer $n$ such that $\kappa$ is $n$-Mahlo
  but not $(n+1)$-Mahlo. Let $E$ be a club in $\kappa$ such that $E$ contains no
  $n$-Mahlo cardinals, and let $S$ be the set of $(n-1)$-Mahlo cardinals below
  $\kappa$ (recall that ``0-Mahlo" is the same as ``inaccessible"). Then
  $S \cap E \cap E^\kappa_{>\chi(\kappa)}$ is a non-reflecting stationary
  subset of $E^\kappa_{>\chi(\kappa)}$, contradicting Clause~(4) of Lemma~\ref{prop53}.

  (2) Suppose that $\chi(\kappa)=1$. In particular, $\kappa$ is inaccessible by Clause~(2) of Lemma~\ref{prop53} and hence Mahlo
  by Clause~(1) of this lemma. Towards a contradiction, suppose that $\langle S_i\mid i<\kappa\rangle$ is a sequence of stationary subsets of $\kappa$, such that,
  for every inaccessible $\beta<\kappa$, for some $i(\beta)<\beta$, $S_{i(\beta)}\cap\beta$ is non-stationary in $\beta$.
  Without loss of generality, we may assume that $\bigcup_{i<\kappa}S_i\s\acc(\kappa)$.
  It follows that we may fix a $C$-sequence $\vec C=\langle C_\beta\mid\beta<\kappa\rangle$ such that, for all $\beta\in\acc(\kappa)$,
  \begin{itemize}
  \item if $\beta$ is singular, then $\otp(C_\beta)=\min(C_\beta)=\cf(\beta)$;
  \item if $\beta$ is inaccessible, then $ C_\beta\cap S_{i(\beta)}=\emptyset$ and $\min(C_\beta)=i(\beta)$.
  \end{itemize}

  Next, by Lemma~\ref{lemma53}, and since $\chi(\kappa)=1$, we can fix a club $\Delta\s\kappa$ and $b:\kappa\rightarrow\kappa$ such that,
  for all $\alpha<\kappa$, $\Delta\cap\alpha\s C_{b(\alpha)}$.
  Let $\alpha\in\acc(\Delta)$ be an arbitrary inaccessible cardinal, and let $\beta:=b(\alpha)$.
  As $\Delta\cap\alpha\s C_\beta$ and $\alpha\in\acc(\Delta)$, we have $\otp(C_{\beta}\cap\alpha)=\alpha$, so $\min(C_\beta)<\alpha$.
  If $\beta$ is singular, then we would get $\otp(C_\beta)=\min(C_\beta)<\alpha\le\otp(C_\beta)$, which yields a contradiction.
  Therefore, $\beta$ is inaccessible and $i(\beta)=\min(C_\beta)<\alpha$.

  Fix a stationary $A\s\reg(\kappa)$ and $i<\kappa$
  such that $i(b(\alpha))=i$ for all $\alpha\in A$.
  Pick $\delta\in\Delta\cap S_i$, and then pick $\alpha\in A$ above $\delta$.
  Then, $\delta\in\Delta\cap\alpha$, whereas $\delta\notin C_{\beta(\alpha)}$, which is again a contradiction.

  (3) This now follows immediately from Clause~(2) and Proposition~\ref{diag_refl_prop}.
\end{proof}
\begin{remark}
It is not the case that $\chi(\kappa) = 1$
implies that $\kappa$ is \emph{strongly} inaccessible.
Indeed, by Corollary~\ref{chaincondition} below,
after adding any number of Cohen reals to a weakly compact cardinal $\kappa$, $\chi(\kappa)\le1$.
\end{remark}

We next show that, assuming the consistency of large cardinals, it is consistently
true that there is a singular cardinal for which $\chi(\lambda^+) = \cf(\lambda)$.

\begin{thm}\label{t55}
  If $\lambda$ is a singular limit of strongly compact cardinals, then $\chi(\lambda^+) = \cf(\lambda)$.
\end{thm}
\begin{proof} By Lemma~\ref{prop53}(2), $\chi(\lambda^+)\ge\cf(\lambda)$.
  To show that $\chi(\lambda^+) \leq \cf(\lambda)$, fix an arbitrary $C$-sequence
  $\vec{C} = \langle C_\beta \mid \beta < \lambda^+ \rangle$. We will produce a club $\Delta$ in $\lambda^+$
  such that for every $\alpha<\lambda^+$, for some $z\in[\lambda^+]^{\le\cf(\lambda)}$,
   $\Delta \cap \alpha  \subseteq \bigcup_{\beta \in z} C_\beta$.

	\begin{claim}[folklore] Let $\delta < \lambda$ be strongly compact.
	Then there exists a $\delta$-complete uniform ultrafilter $D$ over $\acc(\lambda^+)$ which is
	 moreover weakly normal.\footnote{That is, for every $X\in D$ and every regressive function $g:X\rightarrow\lambda^+$,
    there exists $\gamma<\lambda^+$ such that $\Lambda_{g,\gamma}:=\{\beta\in X\mid g(\beta)<\gamma\}$ is in $D$.}
	\end{claim}
	\begin{cproof} Set $\kappa:=\lambda^+$.
	As $\delta$ is $\kappa$-strongly compact, let us fix a $\delta$-complete uniform ultrafilter $U$ over $\kappa$.
	We follow the proof of \cite[Theorem~10.20]{jech}.
	Define an ordering $<_U$ of ${}^{\kappa}\kappa$ by letting $f<_U g$ iff $\{\alpha<\kappa\mid f(\alpha)<g(\alpha)\}\in U$.
     For each $\gamma<\kappa$, 	let $c_\gamma$ denote the constant function from $\kappa$ to $\{\gamma\}$. It is clear that the identity function $\id:\kappa\rightarrow\kappa$
	satisfies $c_\gamma<_U\id$ for all $\gamma<\kappa$.
	As $U$ is a $\sigma$-complete ultrafilter, $<_U$ is well-founded, and it follows that we may fix a function $f:\kappa\rightarrow\kappa$
	satisfying the following two conditions:
	\begin{itemize}
	\item for all $\gamma<\kappa$, $c_\gamma<_U f$;
	\item for every $f':\kappa\rightarrow\kappa$, if $c_\gamma<_U f'$ for all $\gamma<\kappa$, then $\{ \alpha<\kappa\mid f(\alpha)\le f'(\alpha)\}$ is in $U$.
	\end{itemize}
	For every $\alpha<\kappa$, let $(\beta_\alpha,n_\alpha)\in(\{0\}\cup\acc(\kappa))\times\omega$ be such that $f(\alpha)=\beta_\alpha+n_\alpha$.
	Define $f':\kappa\rightarrow\kappa$ via $f'(\alpha):=\max\{\omega,\beta_\alpha\}$, and note that $c_\gamma<_U f'$ for all $\gamma<\kappa$.
	It follows that $f'$ is $U$-equivalent to $f$, so that we may simply assume that $\im(f)\s\acc(\kappa)$.

	Now, it is easy to see that $D:=\{ X\s\acc(\kappa)\mid f^{-1}[X]\in U\}$ is a $\delta$-complete ultrafilter over $\acc(\kappa)$.
	It is uniform, since otherwise, there exists $\gamma<\kappa$
	with $f^{-1}[\gamma]$ in $U$, contradicting the fact that $c_{\gamma+1}<_U f$.
	Finally, to see that $D$ is weakly normal, fix an arbitrary regressive map $g:X\rightarrow\kappa$ with $X\in D$.
	Put $h:=(g\circ f)\restriction f^{-1}[X]$. As $g$ is a regressive function (defined on a set in $U$), $h<_U f$. So, by the choice of $f$, there exists $\gamma<\kappa$
	for which $Y:=\{ \alpha\in\dom(h)\mid h(\alpha)<c_\gamma(\alpha)\}$ is in $U$.
	In particular, $Z:=f[Y] \cap X$ is in $D$ and, for every $\beta\in Z$, there exists $\alpha\in Y$ such that $\beta=f(\alpha)$,
	so that $g(\beta)=h(\alpha)<c_\gamma(\alpha)=\gamma$. Consequently, $\Lambda_{g,\gamma}:=\{\beta\in X\mid g(\beta)<\gamma\}$ covers $Z$,
	therefore, it belongs to $D$.
	\end{cproof}

   \begin{claim} \label{claim581}
    Suppose that $\delta<\lambda$ is an uncountable cardinal and that
    there exists a uniform $\delta$-complete weakly normal ultrafilter $D$ over $\acc(\lambda^+)$.
    Then there exists $A\in[\lambda^+]^{\lambda^+}$ such that, for every
    $B\in[A]^{<\delta}$, there exists $\beta<\lambda^+$ for which $B\s C_\beta$.
  \end{claim}

  \begin{cproof} As $D$ is $\delta$-complete, it suffices to prove that the following set has size $\lambda^+$:
    $$A:=\{\alpha\in\lambda^+\mid \{\beta\in\acc(\lambda^+)\mid \alpha\in C_\beta\}\in D\}.$$
    To this end, let $\epsilon<\lambda^+$ be arbitrary, and we shall find $\alpha\in A$ above $\epsilon$.

    We will recursively define an increasing sequence $\langle \gamma_n
    \mid n < \omega \rangle$ or ordinals below $\lambda^+$. Let $\gamma_0:=\epsilon$. Now, suppose that $n<\omega$ and that $\gamma_n$ has already been defined.
    Define a regressive function $g_n:\acc(\lambda^+\setminus\gamma_n)\rightarrow\lambda^+$ via $g_n(\beta):=\min(C_\beta\setminus(\gamma_n+1))$.
    Then, pick $\gamma_{n+1}<\lambda^+$ for which $\Lambda_{g_n,\gamma_{n+1}}$ is in $D$.

    As $D$ is $\sigma$-complete, $\Lambda:=\bigcap_{n<\omega}\Lambda_{g_n,\gamma_{n+1}}$ is in $D$.
    Let $\alpha:=\sup_{n<\omega}\gamma_n$.
    For every $\beta\in\Lambda$ and $n<\omega$, $C_\beta\cap(\gamma_n,\gamma_{n+1})$ is nonempty,
    so $\alpha\in\acc(C_\beta)\s C_\beta$. Thus, $\Lambda$ witnesses that $\alpha\in A$.
  \end{cproof}

  Fix a strictly increasing sequence $\langle \lambda_i \mid i < \cf(\lambda) \rangle$
  of strongly compact cardinals that converges to  $\lambda$. By the preceding claim, for each $i<\cf(\lambda)$,
  let us fix $A_i\in[\lambda^+]^{\lambda^+}$ such that for every $B\in[A]^{<\lambda_i}$, for some $\beta<\lambda^+$, $B\s C_\beta$.

  We claim that $\Delta:=\bigcap_{i<\cf(\lambda)}\acc^+(A_i)$ is as sought. To see this, let $\alpha<\lambda^+$ be arbitrary. By increasing $\alpha$,
  we can clearly assume that $\otp(\Delta\cap\alpha)=\alpha$ and $\cf(\alpha)=\omega$.
  Now, using the definition of $\Delta$ and the fact that $|\Delta\cap\alpha| \leq \lambda$, fix $\langle B_i\mid i<\cf(\lambda)\rangle$ such that
  \begin{itemize}
  \item $\Delta\cap\alpha=\bigcup_{i<\cf(\lambda)}\acc^+(B_i)$;
  \item for every $i<\cf(\lambda)$, $B_i\in[A_i]^{<\lambda_i}$ and $\sup(B_i)=\alpha$.
  \end{itemize}
  For each $i<\cf(\lambda)$, pick $\beta_i<\lambda^+$ such that $B_i\s C_{\beta_i}$. As $C_{\beta_i}$ is closed below $\alpha$, we also have $\acc^+(B_i)\s C_{\beta_i}$.
  Altogether, $\Delta\cap\alpha\s\bigcup_{i<\cf(\lambda)}C_{\beta_i}$, as sought.
\end{proof}

\begin{remark}
  We briefly remark upon the relationship between the assertion
  $``\chi(\lambda^+) = \cf(\lambda)"$ and two other prominent compactness principles
  that hold at successors of singular limits of strongly compact cardinals,
  namely the tree property and the existence of scales with stationarily many
  bad points. Our remarks indicate that $``\chi(\lambda^+) = \cf(\lambda)"$
  is somewhat orthogonal to the other two properties.

  We first note that the existence of a scale of length $\aleph_{\omega + 1}$
  with stationarily many bad points entails the existence of a non-reflecting
  stationary subset of $E^{\aleph_{\omega + 1}}_{>\omega}$ and hence, by
  Lemma~\ref{prop53}(4), it in fact implies that $\chi(\aleph_{\omega + 1}) >
  \omega$. Similarly, if $\cf(\lambda) < \kappa < \lambda$ and $\kappa$
  is supercompact, then there is a bad scale of length $\lambda^+$, but this
  situation is certainly compatible with the existence of many non-reflecting
  stationary subsets of $E^{\lambda^+}_{\geq \kappa}$ and hence with
  $\chi(\lambda^+) = \lambda$.

  By Theorem~\ref{thm513} below, $\chi(\lambda^+) = \cf(\lambda)$ is compatible
  with the existence of a (special) $\lambda^+$-Aronszajn tree and hence does not
  imply the tree property at $\lambda^+$. On the other hand, in
  known models for the tree
  property at $\aleph_{\omega + 1}$, there are non-reflecting stationary
  subsets of $E^{\aleph_{\omega + 1}}_\omega$ and hence $\chi(\aleph_{\omega + 1})
  > \omega$, so the tree property or even its strengthenings do not imply
  $\chi(\lambda^+) = \cf(\lambda)$.
\end{remark}

\section{Changing the value of the $C$-sequence number}

In this section, we prove a number of consistency results regarding the $C$-sequence
number, square principles, and the principle $\U(\ldots)$. We will begin with consistency
results at inaccessible cardinals and then proceed to successors of singular cardinals.
Let us first recall the following result from Part~I of this series.

\begin{fact}[\cite{paper34}]\label{lemma610}
  Suppose that $\theta \in \reg(\kappa)$, $\chi < \kappa$, and $c:[\kappa]^2\rightarrow\theta$
  witnesses $\U(\kappa,2,\theta,\chi)$. For every infinite cardinal $\theta'<\chi$,
  \begin{enumerate}
    \item if $\theta'<\theta$, then $\mathcal T(c)$ admits no $\theta'$-ascending path;
    \item if $\cf(\theta')\neq\theta$, then $\mathcal T(c)$ admits no $\theta'$-ascent path.
  \end{enumerate}
\end{fact}

We will also need a certain \emph{indexed square} principle, which we now recall.

\begin{defn}[\cite{narrow_systems}]
  Suppose that $\theta < \kappa$ are infinite, regular cardinals. Then $\square^{\ind}(\kappa, \theta)$
  asserts the existence of a sequence $\langle C_{\alpha, i} \mid \alpha < \kappa, ~ i(\alpha) \leq i < \theta \rangle$
  such that
  \begin{enumerate}
    \item for all $\alpha < \kappa$, we have $i(\alpha) < \theta$;
    \item for all $\alpha \in \acc(\kappa)$, $\langle C_{\alpha, i} \mid i(\alpha) \leq i < \theta \rangle$ is a $\subseteq$-increasing sequence of clubs in $\alpha$;
    \item for all $\alpha < \beta$ in $\acc(\kappa)$ and all $i(\beta) \leq i < \theta$, if $\alpha \in \acc(C_{\beta, i})$, then $C_{\beta, i} \cap \alpha = C_{\alpha, i}$ (and, in particular, $i \geq i(\alpha)$);
    \item for all $\alpha < \beta$ in $\acc(\kappa)$, there is $i(\beta) \leq i < \theta$ such that $\alpha \in \acc(C_{\beta, i})$;
    \item there do not exist a club $D$ in $\kappa$ and an ordinal $i < \theta$ such that, for all $\alpha \in \acc(D)$, $D \cap \alpha = C_{\alpha, i}$.
  \end{enumerate}
  A sequence satisfying these requirements is called a $\square^{\ind}(\kappa, \theta)$-sequence.
\end{defn}

\subsection{Inaccessibles}
It follows from Corollary~\ref{cor27} that it is consistent for an inaccessible cardinal $\kappa$
that the $C$-sequence number $\chi(\kappa)$ is equal to $\kappa$.
In addition, it follows from Lemma~\ref{lemma53}(2) that $\chi(\kappa)$ cannot be equal to an integer greater than $1$.
In this subsection we show, among other things, that for an inaccessible cardinal $\kappa$,
$\chi(\kappa)$ can consistently take any value in $\{0,1\} \cup \reg(\kappa + 1)$.
We begin with a result that follows almost immediately from analysis done by Cox
and L\"{u}cke on a model of Kunen and, among other things, provides a model
with an inaccessible cardinal $\kappa$ for which $\chi(\kappa) = 1$.

\begin{thm}\label{kunenmodel}
  Suppose that $\kappa$ is a weakly compact cardinal.
  Then there is a cofinality-preserving forcing extension in which:
  \begin{enumerate}
    \item $\kappa$ is strongly inaccessible;
    \item $\chi(\kappa)=1$;
    \item there exists a coherent $\kappa$-Souslin tree;
    \item $\pr_1(\kappa,\kappa,\kappa,\omega)$ holds;\footnote{Recall that
    $\pr_1(\kappa, \kappa, \kappa, \omega)$ asserts the existence of a coloring
    $c:[\kappa]^2 \rightarrow \kappa$ such that, for every family
    $\mathcal{A} \subseteq [\kappa]^{<\omega}$ consisting of $\kappa$-many
    pairwise disjoint sets, and for every $i < \kappa$, there is $(a,b)
    \in [\mathcal{A}]^2$ such that $c[a \times b] = \{i\}$.}
    \item $\U(\kappa,\kappa,\omega,\omega)$ fails.
  \end{enumerate}
\end{thm}
\begin{proof}
   The desired forcing extension was first isolated by Kunen in \cite{MR495118}.
   The key feature of the forcing extension is that there exists a coherent $\kappa$-Souslin
   tree $(T,\le)$ such that forcing with $\mathbb T:=(T,\ge)$ resurrects the weak
   compactness of $\kappa$. Clauses (1) and (3) follow immediately, whereas Clause~(4) follows from Clause~(3).
   Cox and L\"{u}cke show \cite[Theorem~1.14]{MR3620068} that in this model,
   for every $\nu < \kappa$ the class of
   $\kappa$-Knaster posets is closed under $\nu$-support products. Clause (5)
   then follows from Fact~\ref{knasterlemma}.

   We now establish Clause (2). Fix a $C$-sequence
   $\vec{C} = \langle C_\alpha \mid \alpha  < \kappa \rangle$ in the forcing extension.
   By forcing with $\mathbb{T}$, we resurrect the weak compactness of
   $\kappa$, so, by Theorem
   \ref{todorcevic_thm}, there is a $\mathbb{T}$-name $\dot{D}$ for a club, every initial
   segment of which is equal to the initial segment of a club in $\vec{C}$. But,
   since $\mathbb{T}$ has the $\kappa$-cc, there is a club $\Delta$ in $\kappa$ such that
   $\Vdash_{\mathbb{T}}``\check{\Delta} \subseteq \dot{D}"$. Then every initial
   segment of $\Delta$ is covered by a club in $\vec{C}$, so it witnesses this
   instance of $\chi(\kappa) = 1$.
\end{proof}

We next show that, if $\kappa$ is weakly compact, then there are mild forcing
extensions in which $\chi(\kappa)$ takes any desired value in $\reg(\kappa)$.

\begin{thm}\label{thm51}
  Suppose that $\kappa$ is a weakly compact cardinal and $\theta\in\reg(\kappa)$.
  Then there is a cofinality-preserving forcing extension in which:
  \begin{enumerate}
    \item $\kappa$ is strongly inaccessible;
    \item $\chi(\kappa) = \theta$;
    \item $\square^{\ind}(\kappa, \theta)$ holds;
    \item every $\kappa$-Aronszajn tree admits a $\theta$-ascent path;
    \item there exists a non-reflecting stationary subset of $E^\kappa_\theta$;
    \item $\U(\kappa,2,\theta',\theta^+)$ fails for all $\theta'\in\reg(\kappa)\setminus\{\theta\}$.
      Furthermore, for every $\theta'\in\reg(\theta)$ and for every function $c:[\kappa]^2 \rightarrow \theta'$, there is
      $\mathcal{A} \subseteq [\kappa]^{\leq \theta}$, consisting of $\kappa$-many pairwise disjoint sets,
      such that $\bigcap\{c[a\times b]\mid (a,b)\in[\mathcal A]^2\}\neq\emptyset$.
    \item for every poset $\mathbb{Q}$, if $\mathbb{Q}^\theta$ has the $\kappa$-cc, then
    $\mathbb{Q}^\tau$ has the $\kappa$-cc for all $\tau < \kappa$.
  \end{enumerate}
\end{thm}
\begin{proof}
  By using a preparatory forcing if necessary, we may assume that the weak compactness
  of $\kappa$ is indestructible under forcing with $\add(\kappa, 1)$ (cf.~\cite{MR495118}).
  Let $\mathbb{P}$ be the standard forcing to add a $\square^{\ind}(\kappa, \theta)$-sequence by closed initial segments,
  and, for all $i < \theta$, let $\dot{\mathbb{T}}_i$ be a $\mathbb{P}$-name for
  the forcing to thread the $i^{\mathrm{th}}$ column of the generically-added
  $\square^{\ind}(\kappa, \theta)$-sequence. (See \cite[\S 3]{hayut_lh} for more information
  on all of these forcing notions.)

  Let $G$ be $\mathbb{P}$-generic over $V$, and, in $V[G]$, let
  $\vec{C} := \langle C_{\alpha, i} \mid \alpha < \kappa, ~ i(\alpha) \leq i < \theta \rangle$
  be the generically-added $\square^{\ind}(\kappa, \theta)$-sequence.
  We shall show that $V[G]$ is our desired model. It is clear that Clauses (1) and (3) hold. We
  next verify Clause~(5).
  \begin{claim} \label{nonreflecting_claim}
    Let $S:=\{\alpha \in E^\kappa_\theta\mid \forall i \in [i(\alpha), \theta)\exists \eta < \alpha[\otp(C_{\alpha, i} \setminus \eta) = \theta]\}$.
    Then $S$ is a non-reflecting stationary subset of $E^\kappa_\theta$.
  \end{claim}
  \begin{cproof}
    To see that $S$ does not reflect, fix an arbitrary $\beta \in E^\kappa_{>\theta}$. Let
    $B$ be the set of $\alpha$ in $C_{\beta, i(\beta)}$ such that $\theta^2$ divides
    $\otp(C_{\beta, i(\beta)} \cap \alpha)$. Then $B$ is club in $\beta$ and $B \cap S = \emptyset$,
    so $S$ does not reflect at $\beta$.

    We next show that $S$ is stationary. To this end, let $\dot{S} \in V$ be a canonical
    $\mathbb{P}$-name for $S$, let $\dot{D} \in V$ be a $\mathbb{P}$-name for a club in
    $\kappa$, and let $p = \langle C^p_{\alpha, i} \mid \alpha \leq \gamma^p,
    ~ i(\alpha)^p \leq i < \theta \rangle \in \mathbb{P}$ be an arbitrary condition. Working in $V$, we will find $q \leq_{\mathbb{P}} p$
    such that $q \Vdash_{\mathbb{P}}``\dot{D} \cap \dot{S} \neq \emptyset"$. By extending
    $p$ if necessary, we may assume that $i(\gamma^p)^p = 0$.

    We will recursively construct a decreasing sequence $\langle p_\eta \mid \eta < \theta \rangle$ of conditions
    in $\mathbb{P}$ together with an increasing sequence of ordinals
    $\langle \xi_\eta \mid \eta < \theta \rangle$ satisfying the following
    requirements:
    \begin{enumerate}
      \item $p_0 = p$;
      \item for all $\eta < \theta$, we have $\gamma^{p_\eta} < \xi_\eta <
        \gamma^{p_{\eta+1}}$ and $p_{\eta + 1} \Vdash_{\mathbb{P}}``\check{\xi}_\eta
        \in \dot{D}"$;
      \item for all $\eta < \theta$, $i(\gamma^{p_\eta})^{p_\eta} = 0$
        and $\{\gamma^{p_\zeta} \mid \zeta < \eta\} \subseteq \acc(C^{p_\eta}_{\gamma^{p_\eta}, 0})$;
      \item for all $i < \eta < \theta$, we have $\acc(C^{p_\eta}_{\gamma^{p_\eta}, i}
        \setminus \gamma^{p_i}) = \{\gamma^{p_j} \mid i < j < \eta\}$.
    \end{enumerate}

    The construction proceeds as follows. To start, let $p_0 := p$. If $\eta < \theta$
    and we have constructed $\langle p_\zeta \mid \zeta \leq \eta \rangle$
    and $\langle \xi_\zeta \mid \zeta < \eta \rangle$ satisfying the above requirements,
    then construct $\xi_\eta$ and $p_{\eta + 1}$ as follows. First, find $p_\eta^* \leq_{\mathbb{P}}
    p_\eta$ and an ordinal $\xi_\eta > \gamma^{p_\eta}$ such that
    $p_\eta^* \Vdash_{\mathbb{P}}``\check{\xi}_\eta \in \dot{D}"$.
    Without loss of generality, we may assume that $\gamma^{p^*_\eta} \geq \xi_\eta$.
    We now define $p_{\eta + 1} \leq_{\mathbb{P}} p^*_\eta$ with $\gamma^{p_{\eta + 1}}
    = \gamma^{p^*_\eta} + \omega$. To aid in the definition, let $i^*$ denote the least
    $i \in [\eta, \theta)$ such that $\gamma^{p_\eta} \in \acc\left(C^{p^*_\eta}_{\gamma^{p^*_\eta}, i}\right)$.
    To define $p_{\eta + 1}$, it suffices to define $C^{p_{\eta + 1}}_{\gamma^{p_{\eta + 1}}, i}$ for
    all $i < \theta$. If $i < i^*$, then let
    \[
      C^{p_{\eta + 1}}_{\gamma^{p_{\eta + 1}}, i}
      := C^{p_\eta}_{\gamma^{p_\eta}, i} \cup \{\gamma^{p_\eta}\} \cup \{\gamma^{p^*_\eta} + n \mid n < \omega\}.
    \]
    If $i^* \leq i < \theta$, then let
    \[
      C^{p_{\eta + 1}}_{\gamma^{p_{\eta + 1}}, i} := C^{p^*_\eta}_{\gamma^{p^*_\eta}, i} \cup
      \{\gamma^{p^*_\eta} + n \mid n < \omega\}.
    \]
    It is easily verified that $\xi_\eta$ and $p_{\eta + 1}$ are as desired.

    If $\eta \in \acc(\theta)$ and $\langle p_\zeta, \xi_\zeta \mid \zeta < \eta \rangle$ have
    been constructed, then simply define $p_\eta$ such that $\gamma^{p_\eta} :=
    \sup\{\gamma^{p_\zeta} \mid \zeta < \eta\}$ and, for all $i < \theta$,
    $C^{p_\eta}_{\gamma^{p_\eta}, i} := \bigcup_{\zeta < \eta} C^{p_\zeta}_{\gamma^{p_\zeta}, i}$.
    It is again easily verified that $p_\eta$ is as desired.

    Now define a condition $q \in \mathbb{P}$ by letting $\gamma^q := \sup\{\gamma^{p_\eta}
    \mid \eta < \theta\}$, $i(\gamma^q)^q := 0$, and, for all $i < \theta$,
    $C^q_{\gamma^q, i} := \bigcup_{\eta < \theta} C^{p_\eta}_{\gamma^{p_\eta}, i}$.
    Then $q$ is a lower bound for $\langle p_\eta \mid \eta < \theta \rangle$ and hence,
    for all $\eta < \theta$, we have $q \Vdash_{\mathbb{P}}``\check{\xi}_\eta \in \dot{D}"$.
    Since $\dot{D}$ is forced to be a club and $\gamma^q = \sup\{\xi_\eta \mid \eta < \theta\}$,
    it follows that $q \Vdash_{\mathbb{P}}``\check{\gamma}^q \in \dot{D}"$.
    Also, for all $i < \theta$, requirement~(4) in our construction implies that
    $\acc(C^q_{\gamma^q, i} \setminus \gamma^{p_i}) = \{\gamma^{p_j} \mid i < j < \theta\}$.
    In particular, $\otp(C^q_{\gamma^q, i} \setminus \gamma^{p_i}) = \theta$.
    But then $q \Vdash_{\mathbb{P}}``\check{\gamma}^q \in \dot{S}"$, so
    $q \Vdash_{\mathbb{P}}``\dot{D} \cap \dot{S} \neq \emptyset"$, as desired.
  \end{cproof}

  Clause~(4) is precisely Theorem 1.3 of \cite{lh_lucke}. Together with Clause~(1) and Fact~\ref{lemma610},
  this gives a weak form of Clause~(6). We next verify Clause~(6).

\begin{claim}
    Let $c:[\kappa]^2\rightarrow \theta'$ be an arbitrary function with $\theta'\in\reg(\kappa)\setminus\{\theta\}$.
    Then there is $k < \theta'$ and a collection $\mathcal{A} \subseteq [\kappa]^{\leq \theta}$
    of size $\kappa$, consisting of pairwise disjoint sets,
    such that, for all $(a,b)  \in [\mathcal{A}]^2$,
    \begin{itemize}
      \item if $\theta' < \theta$, then $k \in c[a \times b]$;
      \item if $\theta' > \theta$, then $c[a \times b] \cap k \neq \emptyset$.
    \end{itemize}
    In particular, $\U(\kappa, 2, \theta', \theta^+)$ fails.
\end{claim}
\begin{cproof}
  For every $i < \theta$,
  we know that, in $V$, $\mathbb{P} * \dot{\mathbb{T}}_i$ is equivalent to
  $\add(\kappa, 1)$, so, in $V[G]$, forcing with $\mathbb{T}_i$
  resurrects the weak compactness of $\kappa$. Therefore, it is forced by
  $\mathbbm{1}_{\mathbb{T}_i}$ that there is a $k < \theta'$ and an unbounded
  $H \subseteq \kappa$ such that $c``[H]^2 = \{k\}$. Let $\dot{k}_i$
  and $\dot{H}_i$ be $\mathbb{T}_i$-names for such $k$ and $H$.
  Recall that, for any pair $(j,i)\in[\theta]^2$, we have projection maps $\pi_{ji}:\mathbb{T}_j
  \rightarrow \mathbb{T}_i$ and that, for all $j, j' < \theta$, $t \in \mathbb{T}_j$, and $t' \in \mathbb{T}_{j'}$,
  for all sufficiently large $i < \theta$, $\pi_{ji}(t)$ and $\pi_{j'i}(t')$ are $\leq_{\mathbb{T}_i}$-comparable. Fix $t \in \mathbb{T}_0$ such that, for all $i < \theta$,
  $\pi_{0i}(t)$ decides the value of $\dot{k}_i$, say as $k_i$.

  We now recursively construct a sequence $\langle t_\gamma \mid \gamma < \kappa \rangle$
  of conditions from $\mathbb{T}_0$ and a matrix $\langle \alpha_{\gamma, i} \mid \gamma
  < \kappa, ~ i < \theta \rangle$ of ordinals below $\kappa$. To this end, suppose that $\delta < \kappa$
  and we have defined $\langle t_\gamma \mid \gamma < \delta \rangle$ and
  $\langle \alpha_{\gamma, i} \mid \gamma < \delta, ~ i < \theta \rangle$.
  Let $\beta_\delta := \sup\{\alpha_{\gamma, i} \mid \gamma < \delta, ~ i < \theta\}$,
  and find a condition $t_\delta \leq_{\mathbb{T}_0} t$ such that, for all $i < \theta$,
  $\pi_{0,i}(t_\delta)$ decides the value of the ordinal $\min(\dot{H}_i \setminus (\beta_\delta + 1))$.
  Call this value $\alpha_{\delta, i}$.

  For each $\gamma < \kappa$, let $a_\gamma := \{a_{\gamma, i} \mid i < \theta\}$. It is clear that,
  for all $\gamma < \delta < \kappa$, we have $\sup(a_\gamma)\le\beta_\delta < \min(a_\delta)$.
  In particular, $a_\gamma < a_\delta$. Let $\mathcal{A} :=
  \{a_\gamma \mid \gamma < \kappa\}$. If $\theta' < \theta$, then fix $k < \theta'$
  and  $I \in[\theta]^\theta$ such that $k_i = k$ for all $i \in I$.
  If $\theta' > \theta$, then let $k := \sup\{k_i + 1 \mid i < \theta\}$ and $I := \theta$.

  We claim that $\mathcal{A}$ and $k$ are as desired in the statement of the theorem.
  To this end, fix $(\gamma,\delta)\in[\kappa]^2$. For all sufficiently large $i < \theta$,
  we know that $\pi_{0i}(t_\gamma)$ and $\pi_{0i}(t_\delta)$ are $\le_{\mathbb{T}_i}$-comparable.
  We can therefore find $i \in I$ such that $\pi_{0i}(t_\gamma)$ and $\pi_{0i}(t_\delta)$
  are $\le_{\mathbb{T}_i}$-comparable. Without loss of generality, suppose that
  $\pi_{0i}(t_\delta) \leq_{\mathbb{T}_i} \pi_{0i}(t_\gamma)$.
  Then $\pi_{0i}(t_\delta) \Vdash_{\mathbb{T}_i}``\alpha_{\gamma, i}, \alpha_{\delta, i}
  \in \dot{H}_i."$ Since $t_\delta \leq_{\mathbb{T}_0} t$, it follows that
  $c(\{\alpha_{\gamma, i}, \alpha_{\delta, i}\}) = k_i$. If
  $\theta' < \theta$, then $k_i = k$ and, if $\theta' > \theta$, then $k_i < k$.
  In either case, the fact that $\alpha_{\gamma, i} \in a_\gamma$ and
  $\alpha_{\delta, i} \in a_\delta$ leads to the desired conclusion.
\end{cproof}

  We next verify Clause~(2). By Lemma~\ref{prop53}(4), the existence of
  a non-reflecting stationary subset of $E^\kappa_\theta$ implies that $\chi(\kappa) \geq \theta$.
  Thus, it remains to show that $\chi(\kappa) \leq \theta$. To this end, fix an arbitrary $C$-sequence
  $\langle E_\beta \mid \beta < \kappa\rangle$. We must find $\Delta \in [\kappa]^\kappa$
  and a function $b:\kappa \rightarrow [\kappa]^{\leq \theta}$ such that
  $\Delta \cap \alpha \subseteq \bigcup_{\beta \in b(\alpha)} E_\beta$ for all
  $\alpha < \kappa$.

  For each $i < \theta$, forcing over $V[G]$ with $\mathbb{T}_i$ resurrects
  the weak compactness of $\kappa$. In particular, there is a $\mathbb{T}_i$-name $\dot{\Delta}_i$ for a club in $\kappa$ such that
  \[
  \Vdash_{\mathbb{T}_i}``\text{for all } \check{\alpha} <
  \check{\kappa}, \text{ there is } \dot{\beta} < \check{\kappa} \text{ such that }
  \dot{\Delta}_i \cap \check{\alpha} \subseteq E_{\dot{\beta}}".
  \]
  Moreover, since, for any pair
  $(i,j)\in[\theta]^2$, $\pi_{ij}:\mathbb{T}_i \rightarrow \mathbb{T}_j$
  is a projection, we have that, for all $(i, j)\in[\theta]^2$, $\dot{\Delta}_j$ can be
  interpreted as a $\mathbb{T}_i$-name. Replacing each $\dot{\Delta}_i$ with a
  $\mathbb{T}_i$-name for $\bigcap_{i \leq j < \theta} \dot{\Delta}_j$, we may
  assume that $\Vdash_{\mathbb{T}_i}``\dot{\Delta}_i
  \subseteq \dot{\Delta}_j"$ for all $(i,j)\in[\theta]^2$.

  Recursively construct an increasing sequence $\langle \delta_\xi \mid \xi < \kappa \rangle$
  of ordinals below $\kappa$ together with conditions $\langle t_\xi \mid \xi < \kappa \rangle$
  from $\mathbb{T}_0$ such that $t_\xi \Vdash_{\mathbb{T}_0}``\check{\delta}_\xi \in
  \dot{\Delta}_0"$ for all $\xi < \kappa$. Then let $\Delta' := \{\delta_\xi \mid \xi < \kappa\}$.
  It will clearly suffice to find $b':\kappa \rightarrow [\kappa]^{\leq \theta}$
  such that, for all $\xi < \kappa$, we have $\{\delta_\eta \mid \eta < \xi\}
  \subseteq \bigcup_{\beta \in b'(\xi)} E_\beta$.

  To this end, fix $\xi < \kappa$. For all $\eta < \xi$, let $\alpha_\eta < \kappa$
  be such that $t_\eta = C_{\alpha_\eta, 0}$. Find $\gamma \in \acc(\kappa)$
  such that $\gamma > \alpha_\eta$ for all $\eta < \xi$. For all $\eta < \xi$,
  let $j_\eta < \theta$ be least such that $\alpha_\eta \in \acc(C_{\gamma, j_\eta})$.
  For $j \in [i(\gamma), \theta)$, let $X_j := \{\delta_\eta \mid \eta < \xi \text{ and } j_\eta = j\}$.
  Clearly, $\bigcup_{j \in [i(\gamma), \theta)} X_j = \{\delta_\eta \mid \eta < \xi\}$.

  \begin{claim}
    For all $j \in [i(\gamma), \theta)$, $C_{\gamma, j} \Vdash_{\mathbb{T}_j}``\check{X}_j \subseteq \dot{\Delta}_j"$.
  \end{claim}
  \begin{cproof}
    Fix $j \in [i(\gamma), \theta)$ and $\eta < \xi$ such that $\delta_\eta \in X_j$.
    Recall that $\pi_{0j}:\mathbb{T}_0 \rightarrow \mathbb{T}_j$ is a projection,
    $\Vdash_{\mathbb{T}_0}``\dot{\Delta}_0 \subseteq \dot{\Delta}_j,"$ and
    $t_\eta \Vdash_{\mathbb{T}_0}``\check{\delta}_\eta \in \dot{\Delta}_0."$
    Therefore, $C_{\alpha_\eta, j} = \pi_{0j}(t_\eta) \Vdash_{\mathbb{T}_j}``
    \check{\delta}_\eta \in \dot{\Delta}_j"$. Moreover, we have $\alpha_\eta \in
    \acc(C_\gamma, j)$, so $C_{\gamma, j} \leq_{\mathbb{T}_j} C_{\alpha_\eta, j}$, and
    hence $C_{\gamma, j} \Vdash_{\mathbb{T}_j} ``\check{\delta}_\eta \in \dot{\Delta}_j"$.
  \end{cproof}

  It follows that, for all $j \in [i(\gamma), \theta)$, we have $C_{\gamma, j} \Vdash_{\mathbb{T}_j}
  ``\text{for some } \dot{\beta} < \check{\kappa}, ~ \check{X}_j \subseteq \dot{E}_\beta."$
  But then, in $V[G]$, there \emph{is} a $\beta_j < \kappa$ such that $X_j \subseteq E_{\beta_j}$.
  Letting $b'(\xi) = \{\beta_j \mid j \in [i(\gamma), \theta)\}$, it follows that
  $\{\delta_\eta \mid \eta < \xi\} \subseteq \bigcup_{\beta \in b'(\xi)} E_\beta$,
  and we are done.

  We are left with verifying Clause~(7). To this end, fix in $V[G]$ a poset
  $\mathbb{Q}$ such that $\mathbb{Q}^\theta$ has the $\kappa$-cc.

  \begin{claim}
    There is $i < \theta$ and $t \in \mathbb{T}_i$ such that $t \Vdash_{\mathbb{T}_i}``\check{\mathbb{Q}} \text{ has the } \check{\kappa}\text{-cc}"$.
  \end{claim}
  \begin{cproof}
    Otherwise, for all $i < \theta$, there is a $\mathbb{T}_i$-name $\dot{A}_i$
    for an antichain of size $\kappa$ in $\mathbb{Q}$. For a fixed $i < \theta$, let
    $\langle \dot{q}_{\alpha, i} \mid \alpha < \kappa \rangle$ be a sequence
    of $\mathbb{T}_i$-names for an injective enumeration of $\dot{A}_i$. For
    each $\alpha < \kappa$, fix a condition $t_\alpha \in \mathbb{T}_0$ such
    that, for all $i < \theta$, $\pi_{0i}(t_\alpha)$ decides the value of
    $\dot{q}_{\alpha, i}$, say as $q_{\alpha, i}$. Define a condition
    $\bar{q}_\alpha \in \mathbb{Q}^\theta$ by letting $\bar{q}_\alpha(i) =
    q_{\alpha, i}$ for all $i < \theta$.

    We claim that, for all $\alpha < \beta < \kappa$, $\bar{q}_\alpha$ and
    $\bar{q}_\beta$ are incompatible in $\mathbb{Q}^\theta$. To this end,
    fix $\alpha < \beta < \kappa$. There is a sufficiently large $i < \theta$ such that
    $\pi_{0i}(t_\alpha)$ and $\pi_{0i}(t_\beta)$ are comparable in $\mathbb{T}_i$.
    Without loss of generality, assume that $\pi_{0i}(t_\beta) \leq_{\mathbb{T}_i}
    \pi_{0i}(t_\alpha)$. Then $\pi_{0i}(t_\beta)$ forces $q_{\alpha, i}$ and
    $q_{\beta, i}$ to be distinct elements of the antichain $\dot{A}_i$, so
    they are incompatible in $\mathbb{Q}$. It follows that $\bar{q}_\alpha$
    and $\bar{q}_\beta$ are incompatible. But then $\{\bar{q}_\alpha \mid
    \alpha < \kappa\}$ is an antichain of size $\kappa$ in $\mathbb{Q}^\theta$,
    contradicting the assumption that $\mathbb{Q}^\theta$ has the $\kappa$-cc.
  \end{cproof}

  Fix $i < \theta$ and $t \in \mathbb{T}_i$ as given by the claim, and let $H$
  be $\mathbb{T}_i$-generic over $V[G]$ with $t \in H$. Then $\mathbb{Q}$ has the
  $\kappa$-cc and $\kappa$ is weakly compact in $V[G*H]$. It follows that
  $\mathbb{Q}^\tau$ has the $\kappa$-cc for all $\tau < \kappa$ in $V[G*H]$.
  But $(\mathbb{Q}^\tau)^{V[G*H]} = (\mathbb{Q}^\tau)^{V[G]}$ for all $\tau < \kappa$, and the
  $\kappa$-cc is downward absolute, so $\mathbb{Q}^\tau$ has the $\kappa$-cc
  for all $\tau < \kappa$ in $V[G]$.
\end{proof}

\begin{remark}
  By Corollary~\ref{non_reflecting_cor} below, the fact that there is a non-reflecting
  stationary set in the model from Theorem~\ref{thm51} implies that we in fact have
  $\reg(\theta^+) \subseteq \spec(\kappa)$. It is an interesting question whether
  or not $\card(\theta^+) \subseteq \spec(\kappa)$ necessarily holds in this model.

  It will follow from results in Part~III of this series that, in the model from
  Theorem~\ref{thm51}, $\U(\kappa, \kappa, \theta, \chi)$ holds for all $\chi < \kappa$,
  indicating a way in which Clause (6) of the theorem is sharp.
\end{remark}

\subsection{Successors of singulars}
We now turn our attention to successors of singular cardinals. We first prove
an analogue of Theorem~\ref{thm51}, indicating that, if $\lambda$ is a singular
limit of sufficiently large cardinals, then there are mild forcing extensions in
which $\chi(\lambda^+)$ takes any prescribed value in $\reg(\lambda^+) \setminus
\cf(\lambda)$.

\begin{thm} \label{thm511}
  Suppose that $\lambda$ is a singular limit of strongly compact cardinals that
  are indestructible under $\lambda^+$-directed closed set forcings, and let
  $\theta \in \reg(\lambda^+) \setminus \cf(\lambda)$. Then there is a
  cofinality-preserving forcing extension in which
  \begin{enumerate}
    \item $\chi(\lambda^+) = \theta$;
    \item $\square^{\ind}(\lambda^+, \theta)$ holds;
    \item every $\lambda^+$-Aronszajn tree has a $\theta$-ascent path;
    \item $\U(\lambda^+,2,\theta',\theta^+)$ fails for all $\theta' \in \reg(\lambda^+) \setminus \{\cf(\lambda), \theta\}$;
    \item there exists a non-reflecting stationary subset of $E^{\lambda^+}_\theta$.
  \end{enumerate}
\end{thm}
\begin{proof}
  The proof is very similar to that of Theorem~\ref{thm51}, so many details will
  be suppressed. Let $\mathbb{P}$ be the standard
  forcing to add a $\square^{\ind}(\lambda^+, \theta)$-sequence and, for all
  $i < \theta$, let $\dot{\mathbb{T}}_i$ be a $\mathbb{P}$-name for the forcing
  to thread the $i^{\mathrm{th}}$ column of the generically-added
  $\square^{\ind}(\kappa, \theta)$-sequence (again, see \cite[\S 3]{hayut_lh} for details).
  Note that, for all $i < \theta$, $\mathbb{P} * \dot{\mathbb{T}}_i$ has a
  dense $\lambda^+$-directed
  closed subset and hence $\lambda$ is still a singular limit of strongly compact cardinals in
  $V^{\mathbb{P} * \dot{\mathbb{T}}_i}$.

  Let $G$ be $\mathbb{P}$-generic over $V$. In $V[G]$, let
  $\vec{C} := \langle C_{\alpha, i} \mid \alpha < \lambda^+, ~ i(\alpha)
  \leq i < \theta \rangle$ be the generically-added $\square^{\ind}(\kappa,
  \theta)$-sequence and, for every $i < \theta$, let $\mathbb{T}_i$ be the interpretation
  of $\dot{\mathbb{T}}_i$ in $V[G]$. We claim that $V[G]$ is the desired model. Clearly,
  $\square^{\ind}(\lambda^+, \theta)$ holds in $V[G]$, and the proof
  of the existence of a non-reflecting stationary subset of $E^{\lambda^+}_\theta$
  is exactly as in the proof of Claim~\ref{nonreflecting_claim}. Also, as in
  the proof of Clause~(4) of Theorem~\ref{thm51}, Clause~(3) follows immediately from
  \cite[Theorem 1.3]{lh_lucke}.

  We now verify Clause~(4). Working in $V[G]$, fix
  $\theta' \in \reg(\lambda^+)\setminus \{\cf(\lambda), \theta\}$ and a coloring
  $c:[\lambda^+]^2 \rightarrow\theta'$.
  We will find a family $\mathcal{A} \s [\lambda^+]^{\le \theta}$ consisting of
  $\lambda^+$-many pairwise disjoint sets and an ordinal
  $k < \theta'$ such that, for all $(a,b) \in [\mathcal{A}]^2$, we have $\min(c[a \times b]) \leq k$.
  For all $i < \theta$, forcing with $\mathbb{T}_i$ resurrects the fact
  that $\lambda$ is a singular limit of strongly compact cardinals, and so by \cite[Theorem~2.14]{paper34},
  $\Vdash_{\mathbb{T}_i}``\U(\lambda^+, 2, \theta', \cf(\lambda)^+)$ fails".
  In particular, for each $i < \theta$, we can fix a $\mathbb{T}_i$-name
  $\dot{\mathcal{A}}_i$ for a subset of $[\lambda^+]^{\leq \cf(\lambda)}$
  consisting of $\lambda^+$-many pairwise disjoint sets and a $\mathbb{T}_i$-name
  $\dot{k}_i$ for an ordinal below $\theta'$ such that
  $$
    \Vdash_{\mathbb{T}_i}``\text{for all } (a,b) \in [\dot{\mathcal{A}}_i]^2,
    \text{ we have } \min(\check{c}[a \times b]) \leq \dot{k}_i".
  $$

  Recall that, for any pair $(j, i)\in[\theta]^2$, we have a
  projection map $\pi_{ji}:\mathbb{T}_j \rightarrow \mathbb{T}_i$. Fix
  $t \in \mathbb{T}_0$ such that, for all $i < \theta$, $\pi_{0i}(t)$ decides
  the value of $\dot{k}_i$, say as $k_i$.

  We now recursively construct a sequence $\langle t_\gamma \mid \gamma < \lambda^+
  \rangle$ of condition from $\mathbb{T}_0$ and a matrix $\langle a_{\gamma, i}
  \mid \gamma < \lambda^+, ~ i < \theta \rangle$ of elements of
  $[\lambda^+]^{\leq \cf(\lambda)}$. To this end, suppose that $\delta < \lambda^+$
  and we have defined $\langle t_\gamma \mid \gamma < \delta \rangle$ and
  $\langle a_{\gamma, i} \mid \gamma < \delta, ~ i < \theta \rangle$.
  Let $\beta_\delta := \sup(\bigcup \{a_{\gamma, i} \mid \gamma < \delta, ~ i < \theta\})$.
  Now use the fact that each $\dot{\mathcal{A}}_i$ is forced to consist of
  $\lambda^+$-many pairwise disjoint elements of $[\lambda^+]^{\leq \cf(\lambda)}$
  and the fact that each $\mathbb{T}_i$ is ${<}\lambda^+$-distributive and hence
  does not add any new elements to $[\lambda^+]^{\leq \cf(\lambda)}$ to find
  a condition $t_\delta \in \mathbb{T}_0$ and a sequence $\langle a_{\delta, i}
  \mid i < \theta \rangle$ of elements of $[\lambda^+]^{\leq \cf(\lambda)}$ such
  that, for all $i < \theta$,
  \begin{itemize}
    \item $\beta_\delta < a_{\delta, i}$;
    \item $\pi_{0i}(t_i) \Vdash_{\mathbb{T}_i}``\check{a}_{\delta, i} \in
      \dot{\mathcal{A}_i}"$.
  \end{itemize}

  For each $\gamma < \lambda^+$, let $a_\gamma := \bigcup_{i < \theta} a_{\gamma, i}$,
  and note that $a_\gamma \in [\lambda^+]^{\leq \theta}$. Note also that, for all
  $(\gamma, \delta) \in [\lambda^+]^2$, we have $a_\gamma < a_\delta$.
  Since $\theta' \neq \theta$, we can fix a $k < \theta'$ and an unbounded
  $I \subseteq \theta$ such that, for all $i \in I$, we have $k_i \leq k$.
  We claim that $k$ and $\mathcal{A} := \{a_\gamma \mid \gamma < \lambda^+\}$
  are as desired. To verify this, fix $(\gamma, \delta) \in [\lambda^+]^2$.
  For all sufficiently large $i < \theta$, we know that $\pi_{0i}(t_\gamma)$
  and $\pi_{0i}(t_\delta)$ are $\leq_{\mathbb{T}_i}$-comparable. We can therefore
  find $i \in I$ such that $\pi_{0i}(t_\gamma)$ and $\pi_{0i}(t_\delta)$ are
  $\leq_{\mathbb{T}_i}$-comparable. Without loss of generality, suppose that
  $\pi_{0i}(t_\delta) \leq_{\mathbb{T}_i} \pi_{0i}(t_\gamma)$. Then
  $\pi_{0i}(t_\delta) \Vdash_{\mathbb{T}_i}``\check{a}_{\gamma, i},
  \check{a}_{\delta, i} \in \dot{\mathcal{A}_i}"$. Since $t_\delta \leq_{\mathbb{T}_0}
  t$, it follows that $\min(c[a_{\gamma, i} \times a_{\delta, i}]) \leq k_i \leq k$.
  Since $a_{\gamma, i} \subseteq a_\gamma$ and $a_{\delta, i} \subseteq a_\gamma$,
  the conclusion is immediate.

  We finally turn to Clause~(1). To prove that
  $\chi(\lambda^+) \leq \theta$, note that by Theorem \ref{t55}, for all $i < \theta$, we have
  $\Vdash_{\mathbb{T}_i}``\chi(\lambda^+) = \cf(\lambda)"$. Now proceed as in
  the proof of Theorem~\ref{thm51}, making adjustments analogous to those
  made between the proofs of Clause~(6) of Theorem~\ref{thm51} and Clause~(4) of this theorem.
\end{proof}

The following theorem is proven in a similar manner. The theorem and proof use an indexed
square principle known as $\square^{\ind}_{\lambda, \cf(\lambda)}$. As the notation suggests,
it is a strengthening of both $\square_{\lambda, \cf(\lambda)}$ and $\square^{\ind}(\lambda^+, \cf(\lambda))$. For information about
this square principle and the related standard forcing notions that are used in the following
proof, see \cite[\S 9]{cfm_squares}.

\begin{thm} \label{thm513}
  Suppose that $\lambda$ is a singular limit of strongly compact cardinals $\nu$,
  each of which is indestructible under $\nu$-directed closed set forcings. Then
  there is a cofinality-preserving forcing extension in which
  \begin{enumerate}
    \item $\square^{\ind}_{\lambda, \cf(\lambda)}$ holds;
    \item $\chi(\lambda^+) = \cf(\lambda)$;
    \item $\U(\lambda^+, \lambda^+, \theta, \cf(\lambda)^+)$ fails for all $\theta \in \reg(\lambda)\setminus \{\cf(\lambda)\}$.
  \end{enumerate}
\end{thm}
\begin{proof}
  Fix an increasing sequence $\langle \lambda_i
  \mid i < \cf(\lambda) \rangle$ of indestructibly strongly compact cardinals, with
  $\lambda_0 > \cf(\lambda)$, that converges to $\lambda$. Let $\mathbb{P}$ be the standard forcing to add a
  $\square^{\ind}_{\lambda, \cf(\lambda)}$-sequence
  such that, for all $i < \cf(\lambda)$, all clubs in the $i^{\mathrm{th}}$ column of the
  square sequence have order type less than $\lambda_i$. For all
  $i < \cf(\lambda)$, let $\dot{\mathbb{T}}_i$ be a $\mathbb{P}$-name for the
  forcing to thread the $i^{\mathrm{th}}$ column of the generically-added
  $\square^{\ind}_{\lambda, \cf(\lambda)}$-sequence. Note that, for all
  $i < \cf(\lambda)$, $\mathbb{P} * \dot{\mathbb{T}}_i$ has
  a dense $\lambda_i$-directed closed subset and thus preserves the strong compactness
  of $\lambda_j$ for all $j \leq i$.

  Let $G$ be $\mathbb{P}$-generic over $V$. In $V[G]$, let $\vec{C} :=
  \langle C_{\alpha, i} \mid \alpha < \lambda^+, ~ i(\alpha) \leq i < \cf(\lambda)
  \rangle$ be the generically-added $\square^{\ind}_{\lambda, \cf(\lambda)}$-sequence,
  and, for all $i < \theta$, let $\mathbb{T}_i$ be the interpretation of
  $\dot{\mathbb{T}}_i$ in $V[G]$. We claim that $V[G]$ is the desired model.
  Clearly, $\square^{\ind}_{\lambda, \cf(\lambda)}$ holds in $V[G]$. The proof
  of the failure of $\U(\lambda^+, \lambda^+, \theta, \cf(\lambda)^+)$
  for all $\theta \in \reg(\lambda) \setminus \{\cf(\lambda)\}$ is similar to the
  proof of the analogous fact in Theorem~\ref{thm511}, making adjustments similar
  to those made below in the proof that $\chi(\lambda^+) = \cf(\lambda)$, so
  we leave it to the reader.

  We end by proving that $\chi(\lambda^+) = \cf(\lambda)$. By Clause~(2) of
  Lemma~\ref{prop53}, we have $\chi(\lambda^+) \geq \cf(\lambda)$. To
  prove the other inequality, work in $V[G]$ and fix a $C$-sequence
  $\langle E_\beta \mid \beta < \lambda^+ \rangle$. We will find
  $\Delta \in [\lambda^+]^{\lambda^+}$ and $b:\lambda^+ \rightarrow [\lambda^+]^{\leq \cf(\lambda)}$
  such that $\Delta \cap \alpha \subseteq \bigcup_{\beta \in b(\alpha)} E_\beta$
  for all $\alpha < \lambda^+$.

  Fix $i < \cf(\lambda)$. Since forcing with $\mathbb{T}_i$ resurrects the
  strong compactness of $\lambda_i$, the proof of Claim~\ref{claim581} shows
  that, in the extension by $\mathbb{T}_i$, there is an unbounded subset
  $A_i \subseteq (\lambda^+)^{V[G]}$ such that, for all $B \in [A_i]^{<\lambda_i}$,
  there is $\beta < (\lambda^+)^{V[G]}$ such that $B \subseteq E_\beta$. Let
  $\dot{A}_i$ be a $\mathbb{T}_i$-name for $A_i$.

  As usual, for $(j, i)\in[\cf(\lambda)]^2$, we have a projection map $\pi_{ji}:
  \mathbb{T}_j \rightarrow \mathbb{T}_i$ given by $C_{\gamma, j} \mapsto
  C_{\gamma, i}$. Using the fact that $\mathbb{T}_0$ is ${<}\cf(\lambda)^+$-distributive,
  recursively construct a sequence $\langle t_\delta \mid \delta < \lambda^+ \rangle$
  of conditions from $\mathbb{T}_0$ and a matrix $\langle \alpha_{\delta, i} \mid
  \delta < \lambda^+, ~ i < \cf(\lambda) \rangle$ such that:
  \begin{itemize}
    \item for all $(\delta, \delta') \in [\lambda^+]^2$ and all $i, i' < \cf(\lambda)$,
      we have $\alpha_{\delta, i} < \alpha_{\delta', i'}$;
    \item for all $\delta < \lambda^+$ and all $i < \cf(\lambda)$, we have
      $\pi_{0i}(t_\delta) \Vdash_{\mathbb{T}_i}``\alpha_{\delta, i} \in \dot{A}_i"$.
  \end{itemize}
  For all $i < \cf(\lambda)$, let $\Delta_i := \{\alpha_{\delta, i} \mid \delta
  < \lambda^+\}$, and then let $\Delta := \bigcap_{i < \cf(\lambda)} \acc(\Delta_i)
  \cap E^{\lambda^+}_{> \cf(\lambda)}$. Clearly, $\Delta \in [\lambda^+]^{\lambda^+}$.

  To finish, we need to find, for every $\alpha < \lambda^+$, a set
  $b(\alpha) \in [\lambda^+]^{\cf(\lambda)}$ such that $\Delta \cap \alpha
  \subseteq \bigcup_{\beta \in b(\alpha)} E_\beta$. To this end, fix
  $\alpha < \lambda^+$ and let $\delta^* < \lambda^+$ be the least limit ordinal such that
  $\alpha \leq \alpha_{\delta^*, 0}$. For all $\delta < \delta^*$, let
  $\gamma_\delta$ be the unique element of $\acc(\lambda^+)$ such that
  $t_\delta = C_{\gamma_\delta, 0}$. Choose $\gamma \in \acc(\lambda^+)$ such
  that $\gamma_\delta < \gamma$ for all $\delta < \delta^*$. For all
  $i < \cf(\lambda)$, let $B_i := \{\alpha_{\delta, i} \mid \delta < \delta^* \
  \&\ \gamma_\delta \in \acc(C_{\gamma, i})\}$. Clearly, $|B_i| < \lambda_i$
  for all $i < \cf(\lambda)$.

  \begin{claim}
    $\Delta \cap \alpha \subseteq \bigcup_{i < \cf(\lambda)} \acc^+(B_i)$.
  \end{claim}
  \begin{cproof}
    Fix $\eta \in \Delta \cap \alpha$. By the construction of $\Delta$, there is $\delta < \delta^*$
    such that $\cf(\delta) = \cf(\eta)$ and, for all $i < \cf(\lambda)$, $\eta =
    \sup\{\alpha_{\epsilon, i} \mid \epsilon < \delta\}$. Since $\cf(\delta) >
    \cf(\lambda)$, there is $i < \cf(\lambda)$ such that
    \begin{itemize}
      \item $\sup\{\epsilon < \delta \mid \gamma_\epsilon \in \acc(C_{\gamma, i})\} =
        \delta$; and
      \item there is $\varepsilon \in (\delta, \delta^*)$ such that
        $\gamma_\varepsilon \in \acc(C_{\gamma, i})$.
    \end{itemize}
    Then $\eta \in \acc^+(B_i)$.
  \end{cproof}

  For each $i < \cf(\lambda)$ and each $\delta < \delta^*$ such that $\alpha_{\delta, i}
  \in B_i$, we have $C_{\gamma, i} \leq_{\mathbb{T}_i} \pi_{0i}(t_\delta)$. In
  particular, $C_{\gamma, i} \Vdash_{\mathbb{T}_i}``\check{B}_i \in [\dot{A}_i]^{<\check{\lambda}_i}"$.
  It follows that there is $\beta_i$ such that $B_i \subseteq E_{\beta_i}$. Since
  $E_{\beta_i}$ is closed, we in fact have $\acc^+(B_i) \subseteq E_{\beta_i}$.
  But then, letting $b(\alpha) := \{\beta_i \mid i < \cf(\lambda)\}$, we have
  $\Delta \cap \alpha \subseteq \bigcup_{\beta \in b(\alpha)}E_\beta$.
\end{proof}

We now turn to using Prikry-type forcings to obtain models with singular cardinals
$\lambda$ for which $\chi(\lambda^+) = \cf(\lambda)$. We focus on the case
in which $\cf(\lambda) = \aleph_0$. These methods allow us to obtain results
from weaker large cardinal assumptions than those of the previous two results
and will also allow us to bring these results down to smaller singular cardinals, such as $\aleph_\omega$.
We refer the reader to \cite{gitik_handbook} for more information about
the Prikry-type forcings used in this section.
Since the results at $\aleph_\omega$ will require additional technical arguments
that may obscure the main ideas, we begin by presenting results about singular
cardinals that remain limits of large cardinals.
For our first such result, we need the following large cardinal notion. It was
introduced by Neeman and Steel in \cite{neeman_steel}, where it goes by the
name ``$\Pi^2_1$-subcompact". We use the alternative name introduced by Hayut
and Unger in \cite{hayut_unger}.

\begin{defn}
  A cardinal $\lambda$ is \emph{$\lambda^{+}$-$\Pi^1_1$-subcompact} if, for every
  set $A \s H(\lambda^+)$ and every $\Pi^1_1$ statement $\Phi$ such that
  $(H(\lambda^+), \in, A) \models \Phi$, there are $\rho < \lambda$, $B \s H(\rho^+)$,
  and an elementary embedding $j:(H(\rho^+), \in, B) \rightarrow (H(\lambda^+), \in, A)$
  such that $\crit(j) = \rho$ and $(H(\rho^+), \in, B) \models \Phi$.
\end{defn}
It is easily proven (see \cite[Lemma~36]{hayut_unger}) that, if $\lambda$ is $\lambda^+$-$\Pi^1_1$-subcompact, then $\lambda$ is measurable.

\begin{thm}\label{prikryforcing}
  Suppose that $\lambda$ is $\lambda^+$-$\Pi^1_1$-subcompact and $U$ is a normal
  measure over $\lambda$. Then, in the extension by the Prikry forcing defined
  from $U$, we have $\chi(\lambda^+) = \aleph_0$.
\end{thm}
\begin{proof}
  For the duration of this proof, let $[\lambda]^{<\omega}$ denote the set of
  finite increasing sequences from $\lambda$. Let $\mathbb{P}$ be the Prikry forcing
  defined from $U$. Fix, in $V$, a $C$-sequence
  $\langle C_\beta \mid \beta < \lambda^+ \rangle$ such that $\otp(C_\beta) =
  \cf(\beta)$ for all $\beta \in \acc(\lambda^+)$. Let $\dot{D}$ be a $\mathbb{P}$-name
  for a $C$-sequence over $\lambda^+$, and for $\beta < \lambda^+$, let $\dot{D}_\beta$
  be a $\mathbb{P}$-name forced to be the $\beta^{\mathrm{th}}$ entry of $\dot{D}$.
  We will in fact show that there is $\Delta \in [\lambda^+]^{\lambda^+}$ in $V$
  which is forced to witness the instance of $\chi(\lambda^+) = \aleph_0$ for
  $\dot{D}$ in $V^{\mathbb{P}}$.

  Let $\Gamma := \{\beta < \lambda^+ \mid \omega < \cf(\beta) < \lambda\}$. Since
  shrinking the clubs in a $C$-sequence only makes it harder to witness
  the value of the $C$-sequence number, we can assume without
  loss of generality that, for all $\beta \in \Gamma$, it is forced that $\dot{D}_\beta\s C_\beta$.

  For all  $x \in [\lambda^+]^{<\lambda}$, let $F_x$ denote the set of $p \in \mathbb{P}$
  such that, for some $\beta < \lambda^+$, $p \Vdash_{\mathbb{P}} ``\check{x}\s\dot{D}_\beta"$.

  \begin{claim}
    There is  $\Delta \in[\lambda^+]^{\lambda^+}$ such that, for every 
    $x \in [\Delta]^{<\lambda}$, $F_x$ is dense in $\mathbb{P}$.
  \end{claim}
  \begin{cproof}
    Suppose not. Then, for every $\Delta \in[\lambda^+]^{\lambda^+}$, there is
    $x \in [\Delta]^{<\lambda}$ such that $F_x$ is not dense in $\mathbb{P}$.
    This is a $\Pi^1_1$ statement, which we will call $\Phi$, which is satisfied
    by the structure $(H(\lambda^+), \in, A)$,
    where $A \s H(\lambda^+)$ codes the pair $(\mathbb{P}, \dot{D})$ in a natural way.
    By the fact that $\lambda$ is $\lambda^+$-$\Pi^1_1$-subcompact, we can
    find $\rho < \lambda$, $B \s H(\rho^+)$, and an elementary embedding
    $j: (H(\rho^+), \in, B) \rightarrow (H(\lambda^+), \in, A)$ such that $\crit(j) = \rho$ and
    $(H(\rho^+), \in, B)$ satisfies $\Phi$. By elementarity, $B$ naturally codes
    a pair $(\bar{\mathbb{P}}, \dot{\bar{D}})$, where $\bar{\mathbb{P}}$ is a
    Prikry forcing at $\rho$ and $\dot{\bar{D}}$ is a $\bar{\mathbb{P}}$-name
    for a $C$-sequence over $\rho^+$.

    Let $\gamma := \sup(j``\rho^+)$, so that $\gamma \in E^{\lambda^+}_{\rho^+}\s \Gamma$.
    As $|C_\gamma| < \lambda$ and forcing with $\mathbb{P}$ adds no new bounded
    subsets of $\lambda$ and hence no new subsets of $C_\gamma$,
    we infer that $\dot{D}_\gamma$ is forced to be a member of $V$. Moreover, by the Prikry property and the fact
    that $|\mathcal{P}(C_\gamma)| < \lambda$, it follows that, for every $p \in \mathbb{P}$,
    there is $q \leq^*_{\mathbb{P}} p$ that decides the value of $\dot{D}_\gamma$.
    Thus, for every $s \in [\lambda]^{<\omega}$, there is $A_s\in U$ such that  $(s, A_s)$
    decides the value of $\dot{D}_\gamma$, say as $D^*_s$. Since $\cf(\gamma) = \rho^+$,
    $D^* := \bigcap_{s \in [\rho]^{<\omega}}D_s$ is a club in $\gamma$. Since
    $j``\rho^+$ is a $({<}\rho)$-club in $\gamma$,
    $\bar{\Delta} := \{\alpha < \rho^+\mid j(\alpha) \in D^*\}$ is a $({<}\rho)$-club in $\rho^+$.
    In particular, $\bar\Delta\in[\rho^+]^{\rho^+}$.

    Fix an arbitrary $x \in [\bar{\Delta}]^{<\rho}$, and let $\bar{F}_x$
    denote the set of $p \in \bar{\mathbb{P}}$ such that, for some
    $\beta < \rho^+$, $p \Vdash_{\bar{\mathbb{P}}}``\check{x} \s \dot{\bar{D}}_\beta"$.
    We will show that $\bar{F}_x$ is dense in $\bar{\mathbb{P}}$. To this end,
    fix $p_0 = (s, A)$ in $\bar{\mathbb{P}}$. In $\mathbb{P}$, $p^* := (s, j(A) \cap A_s)$
    extends both $j(p_0)$ and $(s, A_s)$. In particular, since $j(x)=j``x \s j``\bar\Delta$,
    it follows that $p^* \Vdash_{\mathbb{P}}``j(x) \s \dot{D}_\gamma"$. By elementarity,
    there is $p \leq_{\bar{\mathbb{P}}} p_0$ in $\bar{F}_x$. But this implies that
    $\bar{\Delta}$ witnesses the failure of $\Phi$ in $(H(\rho^+), \in, B)$,
    which is a contradiction and finishes the proof of the claim.
  \end{cproof}

  Fix $\Delta\in[\lambda^+]^{\lambda^+}$ as given by the claim.
  For each $\alpha < \lambda^+$, fix a surjection $\varphi_\alpha : \lambda \rightarrow \alpha$.
  Let $G$ be  $\mathbb{P}$-generic over $V$ and work in $V[G]$.
  Let $\{\lambda_n \mid n < \omega\}$ be some cofinal subset of $\lambda$, and
  let $\langle D_\beta \mid \beta < \lambda^+
  \rangle$ be the interpretation of $\dot{D}$.
  By the choice of $\Delta$,  for every $x \in [\Delta]^{<\lambda} \cap V$,
  there is $\beta < \lambda^+$ such that $x \s D_\beta$.
  In particular, for all $\alpha < \lambda^+$ and $n < \omega$,
  there is $\beta_{\alpha, n} < \lambda^+$ such that $\Delta \cap \varphi_\alpha``\lambda_n$ is covered by $D_{\beta,n}$.
  Define $b: \lambda^+ \rightarrow [\lambda^+]^{\leq \omega}$
  by letting $b(\alpha) := \{\beta_{\alpha, n} \mid n < \omega\}$ for all 
  $\alpha < \lambda^+$. Then $\Delta$ and $b$ witness this instance of $\chi(\lambda^+) = \aleph_0$.
\end{proof}

We next obtain a similar result about a non-trivial failure of $\U(\lambda^+, \ldots)$, starting from
a large cardinal notion weaker than strong compactness. It is analogous to Theorem~2.14 of \cite{paper34}.

\begin{thm}
  Suppose that $\lambda$ is a measurable cardinal and there is a $\lambda$-complete uniform ultrafilter over $\lambda^+$.
  Then, in the extension by the Prikry forcing defined from a normal measure on $\lambda$, $\U(\lambda^+, 2, \theta, \aleph_1)$ fails for
  all $\theta \in \reg(\lambda^+)\setminus\{\cf(\lambda)\}$.
\end{thm}
\begin{proof}
  Let $U$ be a normal measure on $\lambda$,
  let $\mathbb{P}$ be the Prikry forcing defined from $U$, and let $\theta \in \reg(\lambda) \setminus \{\aleph_0\}$.
  Set $\kappa:=\lambda^+$, and let $\dot{c}$ be a $\mathbb{P}$-name
  for a function from $[\kappa]^2$ to $\theta$.
  The Prikry property implies that, for every $p \in \mathbb{P}$ and every
  $(\alpha, \beta) \in [\kappa]^2$, there is $q \leq^*_{\mathbb{P}} p$ that
  decides the value of $\dot{c}(\check{\alpha}, \check{\beta})$. Therefore,
  for every stem $s \in [\lambda]^{<\omega}$ and every $(\alpha, \beta)  \in [\kappa]^2$, 
  we can fix a condition $p^s_{\alpha, \beta} = (s, A^s_{\alpha, \beta})  \in \mathbb{P}$ 
  and a color $i^s_{\alpha, \beta} < \theta$ such that
  $p^s_{\alpha, \beta} \Vdash_{\mathbb{P}} ``\dot{c}(\check{\alpha}, \check{\beta})  = \check{i}^s_{\alpha, \beta}"$.

  Let $W$ be a $\lambda$-complete
  uniform ultrafilter over $\kappa$.
  Since $W$ is, in particular, $\theta^+$-complete, we can fix,
  for each stem $s \in [\lambda]^{<\omega}$ and each $\alpha < \kappa$, 
  a color  $i^s_\alpha < \theta$ such that the set
  $X^s_\alpha := \{\beta \in \kappa \setminus (\alpha + 1) \mid i^s_{\alpha, \beta} = i^s_\alpha\}$ is in $W$.

  Next, given $\gamma<\kappa$ and $E\in[\gamma]^{<\lambda}$, let
  \[
    D(\gamma,E):=\left\{(s, A) \in \mathbb{P}\Mid \exists \beta \in \bigcap_{\alpha \in E} X^s_\alpha
    \setminus \gamma\left[ A \s \bigcap_{\alpha\in E} A^s_{\alpha, \beta}\right]\right\}.
  \]
  Clearly, $D(\gamma,E)$ is dense in $\mathbb P$.
  For each $\gamma<\kappa$, fix a surjection $\varphi_\gamma : \lambda \rightarrow \gamma$,
  and then let $D'(\gamma,\nu):=D(\gamma,\varphi_\gamma``\nu)$ for all $\nu<\lambda$.

  Let $G$ be $\mathbb{P}$-generic over $V$, let $\langle \lambda_n \mid n < \omega \rangle$
  be the Prikry sequence, and, for all $m < \omega$, let $s_m := \langle \lambda_n  \mid n < m \rangle$.
  We now recursively construct a sequence $\langle (a_\zeta,i_\zeta) \mid \zeta < \kappa\rangle$
  such that, for all $(\xi,\zeta)\in[\lambda]^2$, we have $a_\xi\in[\kappa]^{\le\aleph_0}$, $i_\xi<\theta$, and $a_\xi<a_\zeta$.

  Suppose that $\zeta < \kappa$ and that $\{ a_\xi \mid \xi < \zeta \}$ have been chosen. Let
  $\gamma_\zeta := \ssup(\bigcup_{\xi < \zeta} a_\xi)$, noting that $\ssup(\emptyset) = 0$.
  For each  $n < \omega$, since $D'(\gamma_\zeta,\lambda_n)$ is dense,
  fix $p_{\zeta, n} \in G \cap D'(\gamma_\zeta,\lambda_n)$, and let
  $\beta_{\zeta, n}$ be a witness to the fact that $p_{\zeta, n} \in D(\gamma_\zeta,\varphi_{\gamma_\zeta}``\lambda_n)$.
  Finally, let $a_\zeta := \{\beta_{\zeta, n} \mid n < \omega\}$
  and $i_\zeta := \ssup\{i^{s_m}_\alpha  \mid \alpha \in a_\zeta,~ m < \omega\}$.
  This completes the description of the construction.

  Fix $i<\theta$ and $Z\in[\kappa]^{\kappa}$
  such that $i_\zeta=i$ for all $\zeta\in Z$, and set $\mathcal A:=\{ a_\zeta\mid \zeta\in Z\}$.
  Let $c:[\kappa]^2 \rightarrow \theta$ be the interpretation
  of $\dot{c}$ in $V[G]$.  We shall show that $c$ fails to witness $\U(\kappa, 2, \theta, \aleph_1)$
  by proving that for all $(a,b) \in [\mathcal{A}]^2$, there are
  $\alpha \in a$ and $\beta \in b$ such that $c(\alpha, \beta) < i$.

  Fix $(a,b) \in [\mathcal A]^2$, along with $(\xi, \zeta) \in [Z]^2$ such
  that $a = a_\xi$ and $b = a_\zeta$. Let $\alpha \in a$ be arbitrary, and let
  $n < \omega$ be large enough so that $\alpha \in \varphi_{\gamma_\zeta}[\lambda_n]$.
  In our construction of $a_\zeta$, we found a condition $p_{\zeta, n} \in G\cap D'(\gamma_\zeta,\lambda_n)$ with a witness $\beta_{\zeta,n}$.
  As $p_{\zeta, n} \in G$, let $m < \omega$ and $A \in U$ be such that $p_{\zeta, n} = (s_m, A)$.
  As  $\beta := \beta_{\zeta, n}$ witnesses that $(s_m,A)$ is in $D(\gamma,\varphi_{\gamma_\zeta}``\lambda_n)$ and $\alpha\in \varphi_{\gamma_\zeta}``\lambda_n$,
  we have that $\beta\in X_\alpha^{s_m}$ and $A\s A_{\alpha,\beta}^{s_m}$.
  As $(s_m,A)\in G$ and $A\s A_{\alpha,\beta}^{s_m}$, we also have $p^{s_m}_{\alpha, \beta} \in G$,
  and thus $c(\alpha, \beta) = i^{s_m}_{\alpha, \beta}$. 
  Since $\beta \in X^{s_m}_{\alpha}$, it follows that 
  $i^{s_m}_{\alpha, \beta} = i^{s_m}_\alpha < i_\xi=i$. 
  We have thus found $\alpha \in a$ and $\beta \in b$ for which $c(\alpha, \beta) < i$, as desired.
\end{proof}

We now show that the preceding results can be brought down to $\aleph_{\omega + 1}$ by starting
with a supercompact cardinal and performing a Prikry forcing with interleaved collapses.

\begin{thm}\label{downtoalephw}
  Suppose that $\lambda$ is a supercompact cardinal. Then there
  is a forcing extension in which
  \begin{enumerate}
    \item $\lambda = \aleph_\omega$;
    \item $\chi(\aleph_{\omega + 1}) = \aleph_0$;
    \item $\U(\aleph_{\omega + 1}, 2, \aleph_k, \aleph_1)$ fails for all $1 \leq k < \omega$.
    \item There exists a non-reflecting stationary subset of $E^{\aleph_{\omega + 1}}_\omega$.
  \end{enumerate}
\end{thm}
\begin{proof}
  Let $V_0$ denote the universe in which we begin. We may assume without loss of
  generality that, in $V_0$, $2^\lambda = \lambda^+$. Let $U_0$ be a normal, fine
  ultrafilter over $\mathcal{P}_\lambda(\lambda^+)$, and let $j_0:V_0 \rightarrow M_0$
  be the corresponding ultrapower map. Let $\mathbb{P}_0$ be the Easton-support forcing iteration of length
  $\lambda + 1$ in which, for every inaccessible $\alpha \leq \lambda$, we add
  $\alpha^{++}$-many Cohen subsets to $\alpha$. Let $V$ be an extension of $V_0$
  by $\mathbb{P}_0$, and work now in $V$. By standard arguments (see, e.g., \cite{gitik_sharon}),
  we can find an extension of $j_0$ to an elementary embedding $j:V \rightarrow M$
  and a sequence of functions $\langle f_\beta \mid \beta < j(\lambda) \rangle$
  from $\lambda$ to $\lambda$ such that, for all $\beta < j(\lambda)$, we have
  $j(f_\beta)(\lambda) = \beta$.

  \begin{claim} \label{representation_claim}
    For every $c \in \mathrm{Coll}(\lambda^{++}, < j(\lambda))^M$, there is a
    function $f \in V$ such that
    \begin{itemize}
      \item $\dom(f) = \lambda$;
      \item for all $\alpha < \lambda$, $f(\alpha) \in \mathrm{Coll}(\alpha^{++}, < \lambda)$;
      \item $j(f)(\lambda) = c$.
    \end{itemize}
  \end{claim}
  \begin{cproof}
    Let $F$ be a function with domain $\lambda$ such that, for all $\alpha < \lambda$,
    $F(\alpha) = \langle c^\alpha_\xi \mid \xi < \lambda \rangle$ is an enumeration
    of $\mathrm{Coll}(\alpha^{++}, < \lambda)$. 
    Then $j(F)(\lambda) = \langle c^\lambda_\xi \mid \xi < j(\lambda) \rangle$ 
    is an enumeration of $\mathrm{Coll}(\lambda^{++}, < j(\lambda))^M$. 
    Now, given an arbitrary $c \in \mathrm{Coll}(\lambda^{++}, < j(\lambda))^M$,
    let $\xi < j(\lambda)$ be such that $c = c^\lambda_\xi$, and let $f$ be a function
    with domain $\lambda$ such that, for all $\alpha < \lambda$, 
    $f(\alpha) =  c^\alpha_{f_\xi(\alpha)}$. 
    Then $j(f)(\lambda) = c^\lambda_{j(f_\xi)(\lambda)} = c^\lambda_\xi = c$, as desired.
  \end{cproof}

  The poset $\mathrm{Coll}(\lambda^{++}, < j(\lambda))^M$ is $\lambda^{++}$-closed
  in $M$. Moreover, from the point of view of $M$, the poset has $j(\lambda)$-many
  maximal antichains. But since $|j(\lambda)| = \lambda^{++}$ and $M$ is closed under
  $\lambda^+$-sequences, we can build in $V$ an $M$-generic filter $H$ for
  $\mathrm{Coll}(\lambda^{++}, < j(\lambda))^M$.

  For later use, fix in $V$ a well-ordering $\lhd$ of $\mathcal{P}_\lambda(\lambda^+)$ and
  a sequence $\langle \varphi_\beta \mid \beta < \lambda^+ \rangle$
  such that, for all $\beta < \lambda^+$, $\varphi_\beta : \lambda \rightarrow \beta$ is a
  surjection.

  Let $U$ be the normal, fine ultrafilter over $\mathcal{P}_\lambda(\lambda^+)$ derived
  from $j$, and let $U^*$ be the normal measure over $\lambda$ obtained by projecting $U$.
  Note that $U$ concentrates on the set of $x \in \mathcal{P}_\lambda(\lambda^+)$ such
  that $\lambda_x := x \cap \lambda$ is an inaccessible cardinal, so we will implicitly
  assume that all elements of $\mathcal{P}_\lambda(\lambda^+)$ that we work with
  are of this form.
  We now let $\mathbb{P}$ be the Prikry forcing with interleaved collapses defined
  from $U^*$ and $H$ (see \cite[\S 4]{gitik_handbook} for further information about
  Prikry-type forcings with interleaved collapses).
   More precisely, conditions of $\mathbb{P}$ are all sequences
  \[
    p = \langle c_0, \alpha_0, c_1, \alpha_1, c_2, \ldots, \alpha_{n-1}, c_n, A, C \rangle
  \]
  satisfying the following conditions.
  \begin{enumerate}
    \item $n < \omega$.
    \item $\langle \alpha_i \mid i < n \rangle$ is an increasing sequence of
      inaccessible cardinals below $\lambda$. For ease of notation, set $\alpha_{-1} := \omega$.
    \item For all $i < n$, $c_i \in \mathrm{Coll}(\alpha^{++}_{i-1}, < \alpha_i)$.
    \item $c_n \in \mathrm{Coll}(\alpha^{++}_{n-1}, < \lambda)$.
    \item $A \in U^*$.
    \item $C$ is a function such that
      \begin{itemize}
        \item $\dom(C) = A$;
        \item for all $\alpha \in A$, $C(\alpha) \in \mathrm{Coll}(\alpha^{++}, < \lambda)$;
        \item $j(C)(\lambda) \in H$.
      \end{itemize}
  \end{enumerate}
  The number $n$ is referred to as the \emph{length} of $p$, or $\mathrm{lh}(p)$, and
  $\langle c_0, \alpha_0, c_1, \ldots, \alpha_{n-1}, c_n \rangle$ is the \emph{stem}
  of $p$, or $s(p)$. We will sometimes write $p$ as $s(p)\conc\langle A,C \rangle$.
  Since $\mathrm{lh}(p)$ depends only on the stem, we can also refer to it as $\mathrm{lh}(s(p))$.
  Given a condition $p$, its constituents will sometimes be referred to as
  $\langle \alpha^p_i \mid i < \mathrm{lh}(p) \rangle$, $\langle c^p_i \mid
  i \leq \mathrm{lh}(p) \rangle$, $A^p$, and $C^p$. The same will be done for stems.
  If $s$ is a stem for $\mathbb{P}$ and $i < \mathrm{lh}(s)$, then we will let
  $s \restriction i$ denote $\langle c^s_0, \alpha^s_0, \ldots, \alpha^s_{i-1},
  c^s_i \rangle$. If $s$ is a stem of length $n$ and $\alpha < \lambda$, then we will say that
  $s$ is \emph{below} $\alpha$ if $c^s_{n} \in \mathrm{Coll}(\alpha^s_{n-1}, < \alpha)$.

  To define the ordering on $\mathbb{P}$, let us first define an ordering on stems.
  If $s$ and $t$ are two stems, then we let $t \leq s$ if
  \begin{itemize}
    \item $\mathrm{lh}(t) \geq \mathrm{lh}(s)$;
    \item for all $i < \mathrm{lh}(s)$, $\alpha^t_i = \alpha^s_i$;
    \item for all $i \leq \mathrm{lh}(s)$, $c^t_i \leq c^s_i$ (note that, if
      $\mathrm{lh}(t) > \mathrm{lh}(s)$, then $c^s_{\mathrm{lh}(s)}$ is being seen here as an
      element of $\mathrm{Coll}((\alpha^t_{\mathrm{lh}(s)-1})^{++}, < \alpha^t_{\mathrm{lh}(s)})$).
  \end{itemize}
  Moreover, we define a notion of \emph{direct extension} by letting $t \leq^* s$
  if $t \leq s$ and $\mathrm{lh}(t) = \mathrm{lh}(s)$.

  Now, if $p,q \in \mathbb{P}$, then we let $q \leq p$ if
  \begin{itemize}
    \item $s(q) \leq s(p)$;
    \item for all $i \in [\mathrm{lh}(p), \mathrm{lh}(q))$, $\alpha^q_i \in A^p$;
    \item for all $i \in (\mathrm{lh}(p), \mathrm{lh}(q)]$, $c^q_i \leq C^p(\alpha^q_{i-1})$;
    \item $A^q \s A^p$;
    \item for all $\alpha \in A^q$, $C^q(\alpha) \leq C^p(\alpha)$.
  \end{itemize}
  Finally, $q \leq^* p$ if $q \leq p$ and $\mathrm{lh}(q) = \mathrm{lh}(p)$.

  Note that, if $p,q \in \mathbb{P}$ and $s(p) = s(q)$, then $p$ and $q$ are
  compatible in $\mathbb{P}$. In fact, if $\mathcal{D}$ is a collection of
  fewer than $\lambda$-many conditions in $\mathbb{P}$, each of which has the same
  stem, then $\mathcal{D}$ has a lower bound in $\mathbb{P}$. In particular, since there are only $\lambda$-many
  stems, $\mathbb{P}$ has the $\lambda^+$-cc. Also, $\mathbb{P}$ satisfies the
  \emph{Prikry property}: for every sentence $\phi$ in the forcing language
  and every condition $p \in \mathbb{P}$, there is $q \leq^* p$ such that
  $q$ decides the truth value of $\phi$. With these facts, standard arguments
  show that, in the extension by $\mathbb{P}$, the only cardinals collapsed
  are those explicitly in the scopes of the interleaved collapses, and hence
  $\lambda = \aleph_\omega$ and $(\lambda^+)^V = \aleph_{\omega + 1}$. In addition,
  if $p \in \mathbb{P}$, $i < \mathrm{lh}(p)$, and $S \s (\alpha^p_i)^+$ is
  stationary, then $p$ forces that $S$ remains stationary in the extension by
  $\mathbb{P}$.

  $V^\mathbb{P}$ will be our desired model. Let us first deal with
  Clause (4) of the theorem. 
  Let $S : =  (E^{\lambda^+}_\lambda)^V$. In $V$, $S$ is clearly a non-reflecting
  stationary subset of $\lambda^+$. Since $\mathbb{P}$ has the $\lambda^+$-cc
  and since being non-reflecting is upward absolute, $S$ remains a non-reflecting
  stationary subset of $\lambda^+$ in $V^{\mathbb{P}}$. Since $\lambda =
  \aleph_\omega$ in $V^{\mathbb{P}}$, it follows that $S$ is a non-reflecting
  stationary subset of $E^{\aleph_{\omega + 1}}_\omega$, as desired.

  We now show that $\chi(\aleph_{\omega + 1}) = \aleph_0$ in $V^{\mathbb{P}}$.
  To this end, let $\dot{C}$ be a $\mathbb{P}$-name for a
  $C$-sequence over $\lambda^+$ and, for all $\beta < \lambda^+$, let $\dot{C}_\beta$
  be a $\mathbb{P}$-name for the $\beta^{\mathrm{th}}$ entry of $\dot{C}$.
  For all $x \in [\lambda^+]^{<\lambda}$, let $F_x$ denote the set of 
  $p \in \mathbb{P}$ such that, for some $\gamma < \lambda^+$, $p \Vdash ``\check{x} \s \dot{C}_\gamma"$.

  \begin{claim} \label{delta_claim}
    There is $\Delta \in [\lambda^+]^{\lambda^+}$ such that $F_x$ is dense in $\mathbb{P}$ for every $x \in [\Delta]^{<\lambda}$.
  \end{claim}
  \begin{cproof}
    Let $\eta := \sup(j``\lambda^+) < j(\lambda^+)$. Note that, if $s$ is a stem
    for $\mathbb{P}$, then $j(s) = s$, and hence $s\conc \langle \lambda, \emptyset \rangle$
    is a valid stem for $j(\mathbb{P})$. Fix a stem $s$ for $\mathbb{P}$, let
    $n := \mathrm{lh}(s)$, and let $p \in j(\mathbb{P})$ be a condition of the form 
    $s\conc \langle \lambda, \emptyset, A, C \rangle$. Working in $M$, define a sequence
    $\langle p_\beta \mid \beta < \lambda^+ \rangle$ of conditions in $j(\mathbb{P})$
    such that, for all $\beta < \lambda^+$,
    \begin{enumerate}
      \item $p_\beta \leq^* p$ and $p_\beta$ decides the truth value of the statement $``j(\beta) \in j(\dot{C})_\eta"$;
      \item $c^{p_\beta}_{n+1} \in H$.
    \end{enumerate}
    This is possible due to the following facts.
    \begin{itemize}
      \item $j(\mathbb{P})$ satisfies the Prikry property in $M$.
      \item Because of the Prikry property, the set of $c$ for which there is
        $q \leq^* p$ such that
        \begin{itemize}
          \item $q$ decides the truth value of $``j(\beta) \in j(\dot{C})_\eta"$;
          \item $c^q_{n+1} = c$;
        \end{itemize}
        is dense in $\mathrm{Coll}(\lambda^{++}, < j(\lambda))^M$.
      \item $H$ is $M$-generic for $\mathrm{Coll}(\lambda^{++}, < j(\lambda))^M$ and hence meets this dense set.
    \end{itemize}

    At the end of the process, since each $c^{p_\beta}_{n+1}$ comes from $H$,
    we can find a lower bound $c_s^* \in H$ for 
    $\langle c^{p_\beta}_{n+1} \mid \beta < \lambda^+ \rangle$. In addition,
    for each $\beta < \lambda^+$, $s(p_\beta) \restriction n$ is a stem for
    $\mathbb{P}$. Since there are only $\lambda$-many stems for $\mathbb{P}$, we
    can find a stationary $S \s E^{\lambda^+}_{<\lambda}$ and a stem $t$ for $\mathbb{P}$
    such that $s(p_\beta) \restriction n = t$ for all $\beta \in S$. 
    For all $\beta \in S$, let $p^*_\beta := t \conc \langle \lambda, c_s^*, A^{p_\beta}, C^{p_\beta} \rangle$.

    Using the fact that each $p^*_\beta$ has the same stem
    and that any collection (in $M$) of fewer than $j(\lambda)$-many conditions in
    $j(\mathbb{P})$ with the same stem has a lower bound, we can find a single
    condition $p_s^*$ such that $p_s^* \leq p^*_\beta$ for all $\beta \in S$.
    We can also assume that $c^{p_s^*}_{n+1} = c_s^*$.
    It follows that, for all $\beta \in S$, $p_s^*$ decides the truth value of
    the statement $``j(\beta) \in j(\dot{C})_\eta"$ in the same way that $p_\beta$ does.
    Moreover, for every stationary subset $T \subseteq E^{\lambda^+}_{<\lambda}$,
    $p_s^*$ forces that $T$ remains stationary in $\lambda^+$ and hence, since
    $j$ is continuous at ordinals of $V$-cofinality less than $\lambda$, that
    $j``T$ remains stationary in $\eta$ in the extension by $j(\mathbb{P})$.

    Let $S' := \{\beta \in S \mid p_s^* \Vdash``j(\beta) \in j(\dot{C})_\eta"\}$.
    Clearly, $S'$ is stationary, as otherwise we would have
    $p_s^* \Vdash ``j(\dot{C})_\eta \cap j``S \text{ is nonstationary}"$, contradicting the
    fact that $p_s^*$ forces $j``S$ to remain stationary in $\eta$. In particular,
    $S'$ is unbounded in $\lambda^+$. Moreover, since $j(\dot{C})_\eta$ is forced to
    be a club in $\eta$, it follows that, letting $D_s$ denote the ordinal closure
    of $S'$ in $\eta$, we have $p_s^* \Vdash ``D_s \s j(\dot{C})_\eta"$.

    Next, let $D := \bigcap \{D_s \mid s \text{ is a stem for } \mathbb{P}\}$.
    Since each $D_s$ is club in $\eta$, $\cf(\eta) = \lambda^+$, and there are only
    $\lambda$-many stems for $\mathbb{P}$, it follows that $D$ is a club and that,
    for every stem $s$, $p_s^* \Vdash ``\check{D} \s j(\dot{C})_\eta"$. Let
    $\Delta := \{\beta < \lambda^+ \mid j(\beta) \in D\}$. Since $j$ is continuous
    at points of cofinality less than $\lambda$, $\Delta$ is $({<}\lambda)$-club
    in $\lambda^+$.

    We claim that $\Delta$ witnesses the conclusion of the claim. To this end, fix
    $x \in [\Delta]^{<\lambda}$ and $p \in \mathbb{P}$. We must find $q \leq p$
    with $q \in F_x$. By the definition of $\mathbb{P}$, we have
    \begin{itemize}
      \item $j(p) = s(p) \conc \langle j(A^p), j(C^p) \rangle$;
      \item $\lambda \in j(A^p)$;
      \item $j(C^p)(\lambda) \in H$.
    \end{itemize}
    Let $n := \mathrm{lh}(p)$. We can now find $\hat{p} \leq j(p)$ in $j(\mathbb{P})$ such that
    $\mathrm{lh}(\hat{p}) = n + 1$, $s(\hat{p}) \restriction n = s(p)$, $\alpha^{\hat{p}}_n = \lambda$,
    and $c^{\hat{p}}_n = j(C^p)(\lambda)$. It follows that $\hat{p}$ and $p^*_{s(p)}$
    are compatible in $j(\mathbb{P})$, so we can find a common extension, $q^*$.
    Note that $j(x) = j``x \s D$, so, since $q^*$ extends $p^*_{s(p)}$, we have
    $q^* \Vdash ``j(x) \s j(\dot{C})_\eta"$. In particular, $q^* \in j(F_x)$.
    By elementarity, there is $q \leq p$ in $F_x$.
  \end{cproof}

  Let $\Delta \in [\lambda^+]^{\lambda^+}$ be as given by the claim, let $G$ be a
  $\mathbb{P}$-generic filter over $V$, and let $\langle \alpha_n \mid n < \omega  \rangle$ 
  be the associated Prikry sequence. 
  Let $\langle C_\beta \mid \beta < \lambda^+ \rangle$ be the interpretation of $\dot{C}$ in $V[G]$. 
  By the claim,  we know that, for every $x \in [\Delta]^{<\lambda} \cap V$, 
  there is $\gamma < \lambda^+$ such that $x \s C_\gamma$. In $V[G]$, for all $\beta < \lambda^+$ and
  $n < \omega$, let $x_{\beta, n} := \Delta \cap \varphi_\beta``\alpha_n$. 
  Then  $x_{\beta, n} \in [\Delta]^{<\lambda} \cap V$, so there is 
  $\gamma_{\beta, n} < \lambda^+$ such that $x_{\beta, n} \s C_{\gamma_{\beta, n}}$. 
  Define  $b:\lambda^+ \rightarrow [\lambda^+]^{\leq \omega}$ by letting 
  $b(\beta) =  \{\gamma_{\beta, n} \mid n < \omega\}$. For every $\beta < \lambda^+$, we have
  $\Delta \cap \beta = \bigcup_{n < \omega} x_{\beta, n}$, so $\Delta$ and $b$ witness this
  instance of $\chi(\lambda^+) = \aleph_0$ in $V[G]$.

  \smallskip

  We finally show that $\U(\aleph_{\omega + 1}, 2, \aleph_k, \aleph_1)$ fails
  in $V^{\mathbb{P}}$ for all $1 \leq k < \omega$. To this end, fix such a $k$,
  and, in $V$, let $\dot{d}$ be a $\mathbb{P}$-name for a function from
  $[\lambda^+]^2$ to $\aleph_k^{V^{\mathbb{P}}}$. Work below a condition 
  $p_0 \in \mathbb{P}$ that has sufficient length so that there is a cardinal $\theta$ in
  $V$ that is forced by $p_0$ to be $\aleph_k^{V^{\mathbb{P}}}$. For ease of
  notation, we will take $k = 1$, so that we may let $p_0 = 1_{\mathbb{P}}$ and
  $\theta = \aleph_1^V$. The general case will follow from the same arguments, with
  appropriate bookkeeping.

  Fix a stem $s$ for $\mathbb{P}$, and let $n:= \mathrm{lh}(s)$. Let $p \in j(\mathbb{P})$ be a
  condition of the form $s \conc \langle \lambda, \emptyset, A, C \rangle$,
  and, for each $\beta < \lambda^+$, find a condition $p_\beta$ such that
  $p_\beta \leq^* p$, $p_\beta$ decides the value of $j(\dot{c})(j(\beta), \eta)$,
  say as $\iota_{s, \beta}$, and such that $c^{p_\beta}_{n+1} \in H$. As in the proof of
  Claim~\ref{delta_claim}, we can find $c_s \in H$ such that $c_s \leq c^{p_\beta}_{n+1}$
  for all $\beta < \lambda^+$. Let $C_s$ be a function as given by
  Claim~\ref{representation_claim} such that $j(C_s)(\lambda) = c_s$.

  For all $\beta < \lambda^+$, let $X_{s, \beta}$ be the set of $x \in \mathcal{P}_\lambda(\lambda^+)$
  for which $\beta \in x$ and there is a condition 
  $q = s' \conc \langle \lambda_x, C_s(\lambda_x), A, C \rangle$ such that
  \begin{itemize}
    \item $s' \leq^* s$;
    \item $q \Vdash ``\dot{d}(\beta, \sup(x)) = \iota_{s, \beta}"$.
  \end{itemize}
  By the discussion in the previous paragraph, we have $j``\lambda^+ \in j(X_{s, \beta})$,
  so $X_{s, \beta} \in U$. Note that, if $q$ is such a condition, then
  $s'$ is below $\lambda_x$. If $\lambda_x$ is inaccessible, then the number of
  stems below $\lambda_x$ is precisely $\lambda_x$. Therefore, since $U$ is a normal
  ideal, there is in fact a set $X'_{s, \beta} \in U$ and a single stem
  $t_{s, \beta} \leq^* s$ such that, for every $x \in X'_{s, \beta}$, there is a
  condition $q = t_{s, \beta} \conc \langle \lambda_x, C_s(\lambda_x), A, C \rangle$
  such that $q \Vdash ``\dot{d}(\beta, \sup(x)) = \iota_{s, \beta}"$.

  Let $G$ be $\mathbb{P}$-generic over $V$, let $\langle \alpha_n \mid n < \omega \rangle$ 
  be the associated Prikry sequence, and let $d$ be the realization
  of $\dot{d}$ in $V[G]$. We now recursively construct
  a family $\mathcal{A} := \{a_\zeta \mid \zeta < \lambda^+\}$, consisting of
  non-empty elements of $[\lambda^+]^{\leq \aleph_0}$, with the property that
  $a_\zeta < a_\xi$ for all $(\zeta, \xi) \in [\lambda^+]^2$. To aid us,
  let $\langle s_\alpha \mid \alpha < \lambda \rangle$ be an enumeration in $V$
  of all stems for $\mathbb{P}$, and recall that we fixed in $V$ a sequence of
  surjections $\langle \varphi_\beta : \lambda \rightarrow \beta \mid \beta < \lambda^+ \rangle$.

  Begin by letting $a_0 = \{0\}$. Suppose now that $\xi < \lambda^+$ and
  that $\{a_\zeta \mid \zeta < \xi\}$ has been defined. First, let
  $\gamma_\xi := \ssup(\bigcup_{\zeta < \xi} a_\zeta)$. For each $n < \omega$, let
  $E_{\xi, n} := \varphi_{\gamma_\xi}[\alpha_n]$, and,
  if there is $x \in \bigcap \{X'_{s_\alpha, \beta} \mid \alpha < \alpha_n, ~ \beta \in E_{\xi, n}\}$ 
  such that $\sup(x) > \gamma_\xi$ and $\lambda_x = \alpha_n$,
  then let $x_{\xi, n}$ be the $\lhd$-least such $x$. Otherwise, let $x_{\xi, n}$ be
  an arbitrary element $x$ of $\mathcal{P}_\lambda(\lambda^+)$ with $\sup(x) > \gamma_\xi$.
  Finally, let $a_\xi = \{\sup(x_{\xi, n}) \mid n < \omega \}$.

  We must thin out $\mathcal{A}$ to obtain a family witnessing this instance of
  the failure of $\U(\aleph_{\omega + 1}, 2, \aleph_k, \aleph_1)$.
  Let us say that two stems $s$ and $t$ of the same length are \emph{compatible}
  if there is a single stem that is a direct extension of both. Note that this amounts
  to saying that $\alpha^s_i = \alpha^t_i$ for all $i < \mathrm{lh}(s)$ and that
  $c^s_i$ and $c^t_i$ are compatible for all $i \leq \mathrm{lh}(s)$. Note also that
  any two conditions with compatible stems are themselves compatible in $\mathbb{P}$.
  Let us additionally say that a stem $s$ is \emph{compatible with $G$} if there is a condition in $G$
  whose stem is $s$. Note that, if two stems of the same length are both compatible
  with $G$, then they are compatible with one another.

  \begin{claim}
    Suppose that $\beta < \lambda^+$ and that $s_0$ and $s_1$ are stems for $\mathbb{P}$
    of the same length such that $t_{s_0, \beta}$ and $t_{s_1, \beta}$ are compatible.
    Then $\iota_{s_0, \beta} = \iota_{s_1, \beta}$.
  \end{claim}

  \begin{cproof}
    Recall that $j(C_{s_0})(\lambda)$ and $j(C_{s_1})(\lambda)$ are in $H$, so the set
    of $x \in \mathcal{P}_\lambda(\lambda^+)$ for which $C_{s_0}(\lambda_x)$ and
    $C_{s_1}(\lambda_x)$ are compatible is in $U$. We can therefore fix
    $x \in X'_{s_0, \beta} \cap X'_{s_1, \beta}$ such that $x$ is in this set.

    We now have conditions $p_\ell = t_{s_\ell, \beta} \conc \langle \lambda_x, C_{s_\ell}(\lambda_x), A_\ell, C_\ell \rangle$ for $\ell < 2$ such that
    $p_\ell \Vdash ``\dot{d}(\beta, \sup(x)) = \iota_{s_\ell, \beta}."$ 
    But $t_{s_0, \beta}$ and $t_{s_1, \beta}$ are compatible stems and $C_{s_0}(\lambda_x)$ and $C_{s_1}(\lambda_x)$ are compatible, so $p_0$ and
    $p_1$ are compatible in $\mathbb{P}$. It follows that $\iota_{s_0, \beta} = \iota_{s_1, \beta}$.
  \end{cproof}

  For each $\zeta < \lambda^+$, let
  \[
    \iota_\zeta := \sup\{\iota_{s, \beta} \mid \beta \in a_\zeta, ~ s \text{ is a stem, and }
    t_{s, \beta} \text{ is compatible with } G\}.
  \]
  By the claim, $\iota_\zeta$ is the supremum of a countable set and is thus below
  $\aleph_k$. Fix an $\iota < \aleph_k$ and an unbounded $B \s \lambda^+$ such
  that $\iota_\zeta = \iota$ for all $\zeta \in B$, and let $\mathcal{B} := \{a_\zeta \mid \zeta \in B\}$.
  We claim that $\mathcal{B}$ is as desired, as witnessed by $\iota$.

  To show this, fix $(\zeta, \xi) \in [B]^2$. We must find $(\beta_0, \beta_1) \in a_\zeta \times a_\xi$ 
  such that $d(\beta_0, \beta_1) \leq \iota$. Begin by letting $\beta_0 \in a_\zeta$ be arbitrary.

  \begin{claim}
    There are a natural number $n < \omega$ and a stem $s$ of length $n$ satisfying
    all of the following statements.
    \begin{enumerate}
      \item $\alpha_n$ is large enough so that $s \in \{s_\alpha \mid \alpha < \alpha_n\}$ and $\beta_0 \in E_{\xi, n}$.
      \item $t_{s, \beta_0}$ is compatible with $G$.
      \item There is $x \in \bigcap\{X'_{s_\alpha, \beta} \mid \alpha < \alpha_n, ~ \beta \in E_{\xi, n}\}$ 
            such that $\sup(x) > \gamma_\xi$ and $\lambda_x = \alpha_n$. Moreover, $d(\beta_0, \sup(x)) = \iota_{s, \beta_0}$ for the $\lhd$-least such $x$.
    \end{enumerate}
  \end{claim}
  \begin{cproof}
    Work in $V$, and let $p_0 = s \conc \langle A_0, C_0 \rangle$ be an arbitrary condition in
    $\mathbb{P}$. We will find $q \leq p_0$ forcing the claim to be true for $s$ and
    $n := \mathrm{lh}(s)$. Fix $\alpha^* < \lambda$ such that $s = s_{\alpha^*}$;
    we may assume that $\min(A_0) > \max\left\{\alpha^*, \min\left(\varphi_{\gamma_\xi}^{-1}[\{\beta_0\}]\right)\right\}$.

    For all $\alpha \in A_0$, let $X_\alpha := \bigcap\{X'_{s_{\alpha'}, \beta} \mid  \alpha' < \alpha, ~ \beta \in \varphi_{\gamma_\xi}[\alpha]\}$. 
    Note that  $X_\alpha \in U$. Let $X$ be the collection of $x \in \mathcal{P}_\lambda(\lambda^+)$
    such that $\sup(x) > \gamma_\xi$ and $x \in X_\alpha$ for all $\alpha < \lambda_x$. By the normality of
    $U$, we have $X \in U$ as well. Let $A^* := \{\lambda_x \mid x \in X\}$.
    We can now find a condition $p_1 = t_{s, \beta} \conc \langle A_1, C_1 \rangle$
    such that
    \begin{itemize}
      \item $p_1 \leq^* p_0$;
      \item $A_1 \s A^*$;
      \item for all $\alpha \in A_1$, $C_1(\alpha) \leq C_s(\alpha)$.
    \end{itemize}

    Next, choose $\alpha \in A_1$ and let $p_2 = t_{s, \beta} \conc \langle \alpha, C_1(\alpha), A_2, C_2 \rangle$ extend $p_1$. 
    By our choice of $X$,  there is $x$ such that $\lambda_x = \alpha$, $\sup(x) > \gamma_\xi$,
    and $x \in \bigcap_{\alpha' < \alpha} X_{\alpha'} = \bigcap \{X'_{s_{\alpha'}, \beta} \mid \alpha' < \alpha, ~ \beta \in \varphi_{\gamma_\xi}[\alpha]\}$. 
    Let $x^*$  be the $\lhd$-least such $x$. Since $x^* \in X'_{s, \beta_0}$, we can find $A_3$ and $C_3$ so that 
    $q := t_{s, \beta} \conc \langle \alpha,  C_1(\alpha), A_3, C_3 \rangle$ extends $p_2$ and
    $q \Vdash ``\dot{d}(\beta_0, \sup(x^*)) = \iota_{s, \beta_0}."$.

    Since $q$ forces that $\alpha_n = \alpha$, and $x^*$ is the $\lhd$-least $x$ such that 
    $x \in \bigcap_{\alpha' < \alpha} X_{\alpha'} = \bigcap \{X'_{s_{\alpha'}, \beta} \mid \alpha' < \alpha, ~ \beta \in \varphi_{\gamma_\xi[\alpha]}\}$, $\sup(x) > \gamma_\xi$,
    and $\lambda_x = \alpha$, the fact that $q$ forces the conclusion of the claim follows from the construction.
  \end{cproof}

  Let $n$ and $s$ be given by the claim, and let $\beta_1 := \sup(x_{\xi, n})$.
  By the claim and our construction of $\mathcal{A}$, $x_{\xi, n}$ is the
  $\lhd$-least $x$ such that $\sup(x) > \gamma_\xi$ and $\lambda_x = \alpha_n$,
  and $d(\beta_0, \beta_1) = \iota_{s, \beta_0}$. Moreover, since $t_{s, \beta_0}$
  is compatible with $G$, we have $\iota_{s, \beta_0} \leq \iota_\eta = \iota$, so
  $(\beta_0, \beta_1)$ is as desired, thus finishing the proof.
\end{proof}

\section{The $C$-sequence spectrum}

To obtain a finer understanding of the $C$-sequence number, we shall want to study the whole spectrum of $C$-sequence values.
As we shall see in the next section, this study will also allow us to prove new results about $\U(\ldots)$.
We begin this section by giving more general versions of the definitions from the Introduction of this paper.

\begin{defn}
  Let $\Gamma$ be a set of ordinals. A \emph{$C$-sequence} over $\Gamma$ is a sequence $\vec C=\langle C_\beta\mid\beta\in\Gamma\rangle$ such that,
  for all $\beta\in\Gamma$, $C_\beta$ is a closed subset of $\beta$ with $\sup(C_\beta)=\sup(\beta)$.
\end{defn}

\begin{defn}
  Given a $C$-sequence $\vec C$ over a stationary subset $\Gamma$ of $\kappa$,
  we let $\chi(\vec C)$ denote the least (finite or infinite) cardinal $\chi\le\kappa$
  such that there exist $\Delta\in[\kappa]^\kappa$ and $b:\kappa\rightarrow[\Gamma]^{\chi}$
  with $\Delta\cap\alpha\s\bigcup_{\beta\in b(\alpha)}C_\beta$
  for every $\alpha<\kappa$.
\end{defn}

Following \cite[Proposition~1.6]{paper29}, we say that $\vec C$ is \emph{amenable} if for every $\Delta\in[\kappa]^\kappa$,
the set $\{\beta\in\Gamma\mid \Delta\cap\beta\s C_\beta\}$ is nonstationary.
In particular, if $\chi(\vec C)>1$, then $\vec C$ is amenable.

\begin{defn} $\spec(\kappa):=\{ \chi(\vec C)\mid \vec C\text{ is a $C$-sequence over }\kappa\}\setminus\omega$.
\end{defn}

By Lemma~\ref{lemma53}(2), $\spec(\kappa)=\emptyset$ iff $\chi(\kappa)\in\{0,1\}$.
The first result of this section asserts that if $\chi(\kappa)>1$, then, in fact, $\chi(\kappa)=\max(\spec(\kappa))$.
Later on, we shall establish that if $\chi(\kappa)>1$, then $\min(\spec(\kappa))=\omega$.
We begin with two technical lemmas.

\begin{lemma} \label{high_cf_lemma}
  Suppose that $\Gamma \s \kappa$, $\vec C=\langle C_\beta\mid\beta\in\Gamma\rangle$ is a $C$-sequence, $D\s\kappa$ is a club, and $\nu$ is a cardinal for which $\Gamma\cap E^\kappa_{\ge\nu}$ is stationary.
  If both
  \begin{enumerate}
  \item  $\otp(C_\beta)<\nu$ for all $\beta\in\Gamma\setminus E^\kappa_{\ge\nu}$ \textbf{and}
  \item $\min(C_\beta)\ge\sup(D\cap\beta)$ for all nonzero $\beta\in\Gamma\setminus D$
  \end{enumerate}
hold, then $\chi(\vec C)=\chi(\vec C\restriction (D\cap E^\kappa_{\ge\nu}))$.
\end{lemma}
\begin{proof}
  Let $\chi:=\chi(\vec C)$. Clearly, $\chi\le \chi(\vec C\restriction (D\cap E^\kappa_{\ge\nu}))\le \sup(\reg(\kappa))$.
  Suppose now that $\chi<\sup(\reg(\kappa))$,
  and that Clauses (1) and (2) above are satisfied.
  We shall show that $\chi(\vec C\restriction (D\cap E^\kappa_{\ge\nu}))\le \chi$.

  As $\Sigma:=D\cap E^\kappa_{\ge\nu}\cap E^\kappa_{>\chi}$ is stationary,
  Lemma~\ref{lemma53} provides us with $\Delta\in[\kappa]^\kappa$ and $b:\kappa\rightarrow[\Gamma]^{\chi}$
  such that:
  \begin{itemize}
  \item $\Delta \cap \alpha \s \bigcup_{\beta \in b(\alpha)} C_\beta$ for   all $\alpha < \kappa$;
  \item  $A:=\{\alpha\in \Sigma\mid \forall \beta\in b(\alpha)[\sup(C_\beta\cap\alpha)=\alpha]\}$ is stationary.
  \end{itemize}

  Let $\alpha\in A$ and $\beta\in b(\alpha)$ be arbitrary.

  $\br$ As $\sup(C_\beta\cap\alpha)=\alpha$ and $\cf(\alpha)\ge\nu$, we have
  $\otp(C_\beta)\ge\nu$, so, by Clause~(1) above, $\cf(\beta)\ge\nu$.

  $\br$ As $\sup(C_\beta\cap\alpha)=\alpha$ and $\alpha>0$,
  we have $\sup(C_\beta\cap\alpha)>\min(C_\beta)$.
  Now, if $\beta\notin D$, then since $\alpha\in D$, we have $\sup(D\cap\beta)\ge\alpha$.
  Putting this together with Clause~(2) above, we infer that if $\beta\notin D$, then $\alpha=\sup(C_\beta\cap\alpha)>\min(C_\beta)\ge\sup(D\cap\beta)\ge\alpha$,
  which is impossible.

  Thus we have shown that, for all $\alpha\in\Sigma$, $b(\alpha)\s \Gamma\cap D\cap E^\kappa_{\ge\nu}$.
  Define $b':\kappa\rightarrow[\Gamma\cap D\cap E^\kappa_{\ge\nu}]^\chi$ via $b'(\alpha):=b(\min(A\setminus\alpha))$.
  Then $\Delta$ and $b'$ witness together that $\chi(\vec C\restriction (D\cap E^\kappa_{\ge\nu}))\le \chi$.
\end{proof}

\begin{lemma} Suppose that $\vec C=\langle C_\beta\mid\beta<\kappa\rangle$ is a $C$-sequence with $\chi(\vec C)=\chi(\kappa)$.
For every family $\mathcal S$ of stationary subsets of $\kappa$, if $\Sigma:=E^\kappa_{>\chi}\bigcap_{S\in\mathcal S}\Tr(S)$ is stationary,
then $\chi(\vec C)=\chi(\vec C\restriction \Sigma)$.

In particular, for every stationary $S\s E^\kappa_{>\chi(\kappa)}$, $\chi(\vec C)=\chi(\vec C\restriction \Tr(S))$.
\end{lemma}
\begin{proof}
Let $\chi := \chi(\kappa)$, and fix a family $\mathcal S\s\mathcal P(\kappa)$ for which $\Sigma:=E^\kappa_{>\chi}\bigcap_{S\in\mathcal S}\Tr(S)$ is stationary.
In particular, we assume that $\chi(\kappa)<\kappa$.
Pick a $C$-sequence $\vec D=\langle D_\beta\mid\beta<\kappa\rangle$ such that for all $\beta<\kappa$:
\begin{itemize}
\item $D_{\beta+1}=\{\beta\}$;
\item if $\beta\in\Tr(\Sigma)$, then $D_\beta=C_\beta$;
\item if $\beta\in\acc(\kappa)\setminus\Tr(\Sigma)$, then $D_\beta\s C_\beta$ and $D_\beta\cap\Sigma=\emptyset$.
\end{itemize}

As $D_\beta\s C_\beta$ for every $\beta<\kappa$, it follows that $\chi(\vec C\restriction T)\le \chi(\vec D\restriction T)$ for every $T\in[\kappa]^\kappa$.
In particular (using $T:=\kappa$), $\chi(\vec D)=\chi(\kappa)$.
Now, by Lemma~\ref{lemma53},
there exist $\Delta\in[\kappa]^\kappa$ and a function $b:\kappa\rightarrow[\kappa]^{\le\chi(\kappa)}$ such that:
\begin{enumerate}
\item $\Delta\cap\alpha \subseteq \bigcup_{\beta\in b(\alpha)}D_\beta$ for all $\alpha<\kappa$;
\item $A:=\{\alpha\in \Sigma\mid \forall \beta\in b(\alpha)[\sup(D_\beta\cap\alpha)=\alpha]\}$ is stationary.
\end{enumerate}

It follows from the definition of $\vec D$ together with Clause~(2) that, for all $\alpha\in A$ and $\beta\in b(\alpha)$, either $\beta=\alpha$ or $\beta\in\Tr(\Sigma)$.
It is not hard to see that $\Tr(\Sigma)\s\Sigma$, and hence $\Delta$ and the function $b':\kappa\rightarrow[\Sigma]^{\le\chi(\kappa)}$ defined via $b'(\alpha):=b(\min(A\setminus\alpha))$
witness together that $\chi(\vec D\restriction\Sigma)\le\chi(\kappa)$.
Altogether, $$\chi(\vec C)\le\chi(\vec C\restriction\Sigma)\le\chi(\vec D\restriction\Sigma)\le\chi(\kappa)=\chi(\vec C).$$

Now, the ``in particular'' part follows from Lemma~\ref{prop53}(4).
\end{proof}

\begin{cor} If $\chi(\kappa)=1$, then for every $C$-sequence $\vec C$ over $\kappa$
and every stationary $S\s\kappa$, $\chi(\vec C\restriction\Tr(S))=1$.\qed
\end{cor}

\begin{thm}\label{realized}
  Suppose that $\chi(\kappa)>1$. Then $\chi(\kappa)=\max(\spec(\kappa))$.
\end{thm}
\begin{proof}
  If $\chi(\kappa)=\omega$ or if $\chi(\kappa)$ is a successor cardinal, then the result follows easily.
  Also, if $\kappa$ is the successor of a regular cardinal, then the proposition follows
  immediately from the proof of Clause~(2) of Lemma~\ref{prop53} (cf. Lemma~\ref{successorlemma} below).
  We may therefore assume that $\chi(\kappa)$ is an uncountable limit cardinal and that $\reg(\kappa)$ has no maximal element.

  Let $\mu := \cf(\chi(\kappa))$, and let
  $\langle \chi_\eta \mid \eta < \mu \rangle$ be a strictly increasing sequence
  of infinite cardinals that converges to $\chi(\kappa)$. For each $\eta < \mu$, 
  let $\vec C^\eta=\langle C^\eta_\beta \mid \beta < \kappa \rangle$ be a $C$-sequence with $\chi(\vec C^\eta) > \chi_\eta$.

  $\br$ Suppose first that $\mu<\kappa$.
  Form a $C$-sequence $\vec C=\langle C_\beta \mid \beta < \kappa \rangle$
  such that $\otp(C_\beta)\le\mu$ for all $\beta\in E^\kappa_{\le\mu}$,
  and such that $C_\beta = \bigcap_{\eta < \mu} C^\eta_\beta$ for all $\beta \in E^\kappa_{>\mu}$.
  We claim that $\chi(\vec C)=\chi(\kappa)$.
  Trivially, $\chi(\vec C)\le\chi(\kappa)$.
  Also, by Lemma~\ref{high_cf_lemma} (using $\nu=\mu^+$ and $D:=\kappa$), we have $\chi(\vec C)=\chi(\vec C\restriction E^\kappa_{>\mu})$.
  Thus, it suffices to show that $\chi_\eta\le \chi(\vec C\restriction E^\kappa_{>\mu})$ for all $\eta<\mu$.
  But this is clear, as for all $\eta<\mu$ and $\beta\in E^\kappa_{>\mu}$, we have $C_\beta\s C_\beta^\eta$,
  so that $\chi_\eta\le\chi(\vec C^\eta)$ $\le \chi(\vec C^\eta\restriction E^\kappa_{>\mu})\le \chi(\vec C\restriction E^\kappa_{>\mu})$.

  $\br$ Suppose now that $\mu = \kappa$. It follows that $\kappa$ is (weakly) inaccessible.
  For each $\beta \in \acc(\kappa)$, let $\pi_\beta:\cf(\beta)\rightarrow\beta$ be a strictly increasing and continuous function whose image is a club in $\beta$.
  Form a $C$-sequence $\vec C=\langle C_\beta \mid \beta < \kappa \rangle$
  such that $\otp(C_\beta)=\omega$ for all $\beta\in E^\kappa_\omega$
  and such that $C_\beta=\pi_\beta[\diagonal_{\eta<\cf(\beta)}\pi_\beta^{-1}[C^\eta_\beta]]$ for all $\beta\in E^\kappa_{>\omega}$.
  Towards a contradiction, suppose that $\chi:=\chi(\vec C)$ is smaller than $\kappa$.

  As $\chi<\kappa=\sup(\reg(\kappa))$ and $\otp(C_\beta)=\cf(\beta)$ for all $\beta\in\acc(\kappa)$,
  Lemma~\ref{high_cf_lemma} (using $\nu=\chi^+$ and $D:=\kappa$) implies that $\chi(\vec C\restriction E^\kappa_{>\chi})=\chi(\vec C)$.
  Thus, let us fix  $\Delta \in [\kappa]^\kappa$ and $b:\kappa \rightarrow  [E^\kappa_{>\chi}]^\chi$
  such that $\Delta \cap \alpha \s \bigcup_{\beta \in b(\alpha)} C_\beta$ for all $\alpha<\kappa$.

  Let $\alpha \in E^\kappa_{>\chi}$ be arbitrary, and set
  $\epsilon_\alpha := \sup\{\sup(\pi_\beta[\chi+1]\cap\alpha)\mid \beta\in b(\alpha)\}$.
  As $$\Delta\cap\alpha\s\bigcup_{\beta\in b(\alpha)}C_\beta\cap\alpha\s \bigcup_{\beta\in b(\alpha)}(\pi_\beta[\chi+1]\cap\alpha)\cup(C_\beta\setminus \pi_\beta[\chi+1]),$$
  we have $\Delta\cap(\epsilon_\alpha,\alpha)\s \bigcup_{\beta\in b(\alpha)}(C_\beta\setminus \pi_\beta[\chi+1])$.

  Use Fodor's lemma to find $\epsilon < \kappa$
  and a stationary $A \s E^\kappa_{>\chi}$ such that $\epsilon_\alpha = \epsilon$
  for all $\alpha \in A$.
  Let $\Delta' := \Delta \setminus (\epsilon + 1)$.
  Then, for all $\alpha\in A$, $$\Delta'\cap\alpha\s\bigcup_{\beta\in b(\alpha)}(C_\beta\setminus \pi_\beta[\chi+1]).$$
  Finally, note that for all $\alpha\in A$ and $\beta\in b(\alpha)$, by definition of $C_\beta$, any $\gamma\in C_\beta\setminus\pi_\beta[\chi+1]$
  is of the form $\pi_\beta(\xi)$ for some $\xi\in(\chi,\cf(\beta))$ such that $\xi\in\pi_\beta^{-1}[C_\beta^\eta]$ for all $\eta<\xi$.
  In particular, $\gamma\in C^\chi_\beta$.
  It follows that for all $\alpha\in A$, we have $\Delta'\cap\alpha\s\bigcup_{\beta\in b(\alpha)} C^\chi_\beta$.
  Thus, by appealing with $b\restriction A$ and $\Delta'$ to Lemma~\ref{lemma53},
  we obtain $b^*:\kappa\rightarrow[\kappa]^{\le\chi}$ and $\Delta^*\in[\kappa]^\kappa$ such that $\Delta^*\cap\alpha\s\bigcup_{\beta\in b^*(\alpha)}C^\chi_\beta$
  for all $\alpha<\kappa$, contradicting the fact that $\chi(\vec C^\chi)> \chi_\chi\ge\chi$.
\end{proof}

\begin{cor}
If $\kappa$ is a Mahlo cardinal, then there exists a $C$-sequence $\vec C$ over $\kappa$
for which $\chi(\vec C\restriction \reg(\kappa))=\chi(\kappa)$.\qed
\end{cor}

\begin{cor}\label{chaincondition}
  Suppose $\mathbb{P}$ is a $\nu$-cc poset for an infinite regular cardinal $\nu$.
  Let $\chi:=\max\{1,\chi(\kappa)\}$.
  \begin{enumerate}
  \item If $\nu<\kappa$, then $V^{\mathbb{P}}\models \chi(\kappa) \leq \chi$;
  \item If $\nu\le\chi<\kappa$, then $V^{\mathbb{P}}\models \chi(\kappa) = \chi$.
  \end{enumerate}
\end{cor}
\begin{proof} (1)
  Assume that $\nu < \kappa$, and suppose that $\dot{\vec{C}} = \langle \dot{C}_\beta \mid \beta < \kappa \rangle$ is a $\mathbb{P}$-name for a $C$-sequence.
  Using the fact that $\mathbb{P}$ has the $\nu$-cc, we may
  find a $C$-sequence $\langle D_\beta \mid \beta < \kappa \rangle$ in $V$ such that
  \begin{itemize}
    \item for all $\beta \in E^\kappa_{<\nu}$, $\otp(D_\beta)<\nu$;
    \item for all $\beta \in E^\kappa_{\geq \nu}$, $\Vdash_{\mathbb{P}} ``\check{D}_\beta \subseteq \dot{C}_\beta"$.
  \end{itemize}

  By Lemma~\ref{high_cf_lemma}, we have $\chi(\vec D\restriction E^\kappa_{\ge\nu})=\chi(\vec D)\le\max\{1,\chi(\kappa)\}=\chi$,
  so we may fix $\Delta\in[\kappa]^\kappa$ and $b:\kappa\rightarrow[E^\kappa_{\ge\nu}]^{\chi}$
  with $\Delta\cap\alpha\s\bigcup_{\beta\in b(\alpha)}D_\beta$ for every $\alpha<\kappa$.
  But then, for every $\alpha < \kappa$, we have $\Vdash_{\mathbb{P}}``\check{\Delta} \cap \check{\alpha} \s\bigcup_{\beta\in \check b(\check \alpha)}\dot{C}_{\beta}"$.
  So, $\Vdash_{\mathbb{P}}``\chi(\dot{\vec{C}})\le\check\chi"$.

  (2) Assume $\nu\le\chi<\kappa$, and let us show that $V^{\mathbb{P}}\models \chi(\kappa) = \chi$.
  By Lemma~\ref{prop53}(2), we may assume that $\kappa$ is either an inaccessible or the successor of a singular cardinal.
  In particular, $\nu$ is smaller than $\lambda:=\sup(\reg(\kappa))$. As $\mathbb P$ has the $\nu$-cc, $V^{\mathbb P}\models \sup(\reg(\kappa))=\lambda$.
  In $V$, using Theorem~\ref{realized}, let us pick a $C$-sequence $\vec C=\langle C_\beta\mid\beta<\kappa\rangle$
  with $\chi(\vec C)=\chi$.
  Towards a contradiction, suppose that $V^{\mathbb{P}}\models \chi(\kappa) < \chi$, so that, in particular,
  $V^{\mathbb{P}}\models \chi(\kappa) < \sup(\reg(\kappa))$.
  Let $\chi':=(\chi(\kappa))^{V^{\mathbb P}}$.
  In $V^{\mathbb{P}}$, using Lemma~\ref{lemma53},
  let us fix $\Delta_1\in[\kappa]^\kappa$ and a function $b:\kappa \rightarrow [\kappa]^{\le\chi'}$ such that
      \begin{itemize}
      \item $\acc^+(\Delta_1)\cap E^\kappa_{>\chi'}\s\Delta_1$;
      \item $\Delta_1\cap \alpha \s \bigcup_{\beta \in b(\alpha)}\acc(C_\beta)$ for all $\alpha < \kappa$.
      \end{itemize}
  As $\mathbb P$ has the $\kappa$-cc, let us fix in $V$ a subclub $D$ of $\acc^+(\Delta_1)$.
  In $V$, we let $\Delta:=D\cap E^\kappa_{>\max\{\nu,\chi'\}}$.
  As $\max\{\nu,\chi'\}<\lambda=\sup(\reg(\kappa))$ and $\lambda$ is a limit cardinal,
  $\Delta$ forms a cofinal subset of $\Delta_1$.
  In $V^{\mathbb P}$, for every $\alpha<\kappa$, we have
  $$\Delta\cap \alpha \s \Delta_1\cap\alpha\s \bigcup_{\beta \in b(\alpha)}\acc(C_\beta)\s \bigcup_{\beta \in b(\alpha)} C_\beta.$$

  $\br$ If $\nu\le\chi'$, then since $\mathbb P$ has the $\nu$-cc, any set of ordinals of size $\chi'$ in $V^{\mathbb P}$ is covered by a $V$-set of size $\chi'$.
  This means that, in $V$, for every $\alpha<\kappa$, there exists $b_\alpha\in[\kappa]^{\le\chi'}$ with $\Delta\cap\alpha \subseteq \bigcup_{\beta \in b_\alpha} C_\beta$.
  This contradicts the choice of $\vec C$ and the fact that $\chi'<\chi$.

  $\br$ If $\chi'<\nu$, then since $\nu$ is regular, the union of $\chi'$ many sets, each of size $<\nu$, is of size smaller than $\nu$.
  This means that, in $V$, for every $\alpha<\kappa$, there exists $b_\alpha\in[\kappa]^{<\nu}$ with $\Delta\cap\alpha \subseteq \bigcup_{\beta \in b_\alpha} C_\beta$.
  This contradicts the choice of $\vec C$ and the fact that $\nu<\chi<\kappa$.
\end{proof}
\begin{remark} Theorem~\ref{prikryforcing} takes care of the case $\chi<\nu=\kappa$, witnessing that the preceding is optimal.
\end{remark}

\begin{lemma}\label{successorlemma}
  For every infinite cardinal $\lambda$, $\cf(\lambda)\in\spec(\lambda^+)$.
\end{lemma}
\begin{proof}
  If $\lambda$ is a regular cardinal, then the conclusion follows immediately
  from the proof of Lemma~\ref{prop53}(2). 
  Thus, suppose that $\lambda$ is a singular cardinal, and
  let $\{ \lambda_i\mid i<\cf(\lambda)\}$ be a cofinal subset of $\reg(\lambda)$.
  For each $\alpha<\lambda^+$, fix a surjection $\varphi_\alpha:\lambda\rightarrow\alpha$.
  Also fix an arbitrary progressive function $b:\lambda^+\rightarrow[\lambda^+]^{\cf(\lambda)}$
  such that $\otp(b(\alpha)) = \cf(\lambda)$ and $b(\alpha)\cap b(\alpha')=\emptyset$ for all $(\alpha,\alpha')\in[\lambda^+]^2$.
  For every $\beta<\lambda^+$,
  let $\alpha(\beta)$ denote the unique $\alpha$ such that $\beta\in b(\alpha)$, if such $\alpha$ exists; otherwise, $\alpha(\beta)$ is undefined.
  Let $\vec{C} = \langle C_\beta \mid \beta < \lambda^+ \rangle$ be a $C$-sequence satisfying the following conditions:
  \begin{itemize}
    \item for all $\beta<\lambda^+$, $\otp(C_\beta)<\lambda$;
    \item for all $\beta<\lambda^+$, if $\alpha(\beta)$ is defined, then $C_\beta\supseteq \varphi_{\alpha(\beta)}[\lambda_{\otp(b(\alpha(\beta))\cap\beta)}]$.
  \end{itemize}
  By the proof of Lemma~\ref{prop53}(2), $\chi(\vec C)\ge\cf(\lambda)$. To see that $\chi(\vec C)\le\cf(\lambda)$,
  fix an arbitrary $\alpha<\lambda^+$. Let $\{ \beta_i\mid i<\cf(\lambda)\}$ be the increasing enumeration of $b(\alpha)$.
  Then $\bigcup_{\beta\in b(\alpha)}C_\beta=\bigcup_{i<\cf(\lambda)}C_{\beta_i}\supseteq\bigcup_{i<\lambda}\varphi_\alpha[\lambda_i]=\alpha$,
  so $\lambda^+$ and $b$ witness that $\chi(\vec{C}) \leq \cf(\lambda)$.
\end{proof}

As one might expect, there are finer connections between the $C$-sequence
number and square principles. Let us note a few of them here.

\begin{lemma}\label{lemma47}
  Suppose that $\vec C=\langle C_\beta\mid\beta\in\Gamma\rangle$ is a transversal for $\square(\kappa,{<}\mu,\sq_\sigma)$,
  with $\sigma\in\reg(\kappa)$ and $\mu<\kappa$. Then, for every stationary $\Gamma'\s\Gamma$,
  \begin{itemize}
    \item $\chi(\vec C\restriction\Gamma')\ge\omega$;
    \item if $\chi(\vec C\restriction\Gamma')<\sup(\reg(\kappa))$, then $\chi(\vec C\restriction\Gamma')<\mu$.
  \end{itemize}
\end{lemma}

\begin{proof}
  Recalling the definitions from \cite[\S1]{paper29}, the hypothesis amounts to asserting the existence of a sequence $\langle\mathcal C_\beta\mid \beta<\kappa\rangle$ such that
  \begin{itemize}
    \item for every $\beta \in\acc(\kappa)$, $\mathcal C_\beta$ is a collection of fewer than $\mu$ clubs in $\beta$;
    \item for every  $\beta \in \acc(\kappa)$, every $C \in \mathcal C_\beta$, and every $\alpha\in \acc(C)$,
      either $C\cap\alpha\in\mathcal C_\alpha$ or ($\otp(C)< \sigma$ and $\nacc(C)$ consists only of successor ordinals);
    \item for every club $D\s\kappa$, there exists $\beta\in\acc(D)$ such that $(D\cap\beta)\notin \mathcal  C_\beta$;
    \item $\Gamma = \{ \beta \in \acc(\kappa) \mid \forall C \in \mathcal C_\beta \forall \alpha \in \acc(C) [C \cap \alpha \in \mathcal C_{\alpha}] \}$;
    \item for every $\beta\in\Gamma$, $C_\beta\in\mathcal C_\beta$.
  \end{itemize}

  Now, let $\Gamma'$ be an arbitrary stationary subset of $\Gamma$.
  \begin{claim}
    $\chi(\vec C\restriction\Gamma')\ge\omega$.
  \end{claim}
  \begin{cproof}
    Suppose not. By Lemma~\ref{lemma53}(2), then, we may fix $\Delta\in[\kappa]^\kappa$
    and a function $b:\kappa \rightarrow \Gamma'$
    such that $\Delta \cap \alpha\s C_{b(\alpha)}$ for all  $\alpha < \kappa$.
    Let $D := \acc^+(\Delta)$.
    For every $\alpha\in D$, we have $b(\alpha)\in\Gamma$ and $\sup(C_{b(\alpha)}\cap\alpha)=\alpha$,
    so $C_\alpha^\bullet:=C_{b(\alpha)}\cap\alpha$ is in $\mathcal C_\alpha$, and $\Delta\cap\alpha\s C_\alpha^\bullet$.
    So $\Delta$ witnesses that $\langle C_\alpha^\bullet\mid \alpha\in D\rangle$ is not amenable.
    As $C_\alpha^\bullet\in\mathcal C_\alpha$ for all $\alpha\in D$, we thus get a contradiction to \cite[Lemma~1.23]{paper29} (cf.~\cite[Corollary 2.6]{hayut_lh}).
  \end{cproof}
  Next, suppose that $\chi:=\chi(\vec C\restriction\Gamma')$ is smaller than $\sup(\reg(\kappa))$.
  In particular, $\Sigma:=E^\kappa_{>\chi}$ is stationary.
  \begin{claim}
    $\chi(\vec C\restriction\Gamma')<\mu$.
  \end{claim}
  \begin{proof} \renewcommand{\qedsymbol}{\ensuremath{\boxtimes \ \square}}
    By Lemma~\ref{lemma53}(1),    we can fix $\Delta\in[\kappa]^\kappa$
    and a function $b:\kappa \rightarrow [\Gamma']^\chi$
    such that
    \begin{itemize}
      \item $\Delta\cap \alpha \s \bigcup_{\beta \in b(\alpha)} C_\beta$ for all $\alpha < \kappa$;
      \item $A:=\{\alpha\in \Sigma\mid \forall \beta\in b(\alpha)[\sup(C_\beta\cap\alpha)=\alpha]\}$ is stationary.
    \end{itemize}
    For every $\alpha\in A$ and $\beta\in b(\alpha)$, we have $\beta\in\Gamma$, so $C_\beta\cap\alpha\in\mathcal C_\alpha$.
    It follows that we may pick $b':A\rightarrow[\Gamma']^{<\mu}$ such that, for all $\alpha\in A$:
    \begin{itemize}
    \item $b'(\alpha)\in[ b(\alpha)]^{<\mu}$, and
    \item $\{ C_\beta\cap\alpha\mid \beta\in b'(\alpha)\}=\{ C_\beta\cap\alpha\mid \beta\in b(\alpha)\}$.
    \end{itemize}
    Fix $A'\in[A]^\kappa$ on which $\alpha\mapsto|b'(\alpha)|$ is constant, with value, say, $\mu'$.
    Then $\Delta\cap \alpha \s \bigcup_{\beta \in b'(\alpha)} C_\beta$ for all $\alpha \in A'$,
    so $\chi(\vec C\restriction\Gamma')\le\mu'<\mu$.
\end{proof}
\let\qed\relax
\end{proof}

\begin{cor}\label{cor48}
  Suppose that $\square(\kappa,{<}\mu,{\sq_\sigma})$ holds, with $\sigma\in\reg(\kappa)$.
  \begin{enumerate}
    \item If $\mu<\kappa$, then $\chi(\kappa)\ge\omega$;
    \item If $\mu\le\omega$, then $\chi(\kappa)=\sup(\reg(\kappa))$.
  \end{enumerate}
\end{cor}
\begin{proof}
  Fix a transversal $\vec C=\langle C_\beta\mid\beta\in\Gamma\rangle$ for $\square(\kappa,{<}\mu,{\sq_\sigma})$.
  As $\kappa\setminus\Gamma\s E^\kappa_{<\sigma}$, let us fix an extension $\vec C^\bullet=\langle C_\beta\mid\beta<\kappa\rangle$ of $\vec C$
  such that $\otp(C_\beta)<\sigma$ for all $\beta\in\kappa\setminus\Gamma$.
  Let $\Gamma':=E^\kappa_{\ge\sigma}$. Then $\Gamma'\s\Gamma$, so
  by Lemma~\ref{high_cf_lemma}, $\chi(\vec C^\bullet\restriction \Gamma')=\chi(\vec C^\bullet)\le\chi(\vec C)\le\chi(\vec C\restriction\Gamma')=\chi(\vec C^\bullet\restriction\Gamma')$.
  That is, $\chi(\vec C^\bullet)=\chi(\vec C)$.

  (1) If $\mu<\kappa$, then by Lemma~\ref{lemma47}(1), we have $\chi(\kappa)\ge\chi(\vec C^\bullet)=\chi(\vec C)\ge\omega$.

  (2) Suppose that $\chi(\kappa)<\sup(\reg(\kappa))$. Then $\chi(\vec C)=\chi(\vec C^\bullet)<\sup(\reg(\kappa))$,
  so, by Lemma~\ref{lemma47}(2), $\chi(\vec C)<\mu$.
  It now follows from Lemma~\ref{lemma47}(1) that $\mu>\omega$.
\end{proof}

\section{The $C$-sequence spectrum and closed colorings} \label{c_u_section}

The upcoming subsections will uncover some connections between elements of $\spec(\kappa)$ and the
third and fourth parameters in closed instances of $\U(\kappa, \ldots)$. Before getting to those connections,
let us recall some relevant definitions and results from \cite{paper34}.

\begin{defn}[\cite{paper34}]
  For a subset $\Sigma\s\kappa$, we say that $c:[\kappa]^2\rightarrow\theta$ is
  \emph{$\Sigma$-closed} if, for all $\beta<\kappa$ and $i\le \theta$,
  the set $D^c_{\le i}(\beta):=\{\alpha<\beta\mid c(\alpha,\beta)\le i\}$ satisfies
  $\acc^+(D^c_{\le i}(\beta))\cap\Sigma\s D^c_{\le i}(\beta)$.
  We say that $c$ is \emph{somewhere-closed} if it is $\Sigma$-closed for some stationary $\Sigma\s\kappa$,
  that $c$ is \emph{tail-closed} if it is $E^\kappa_{\ge\sigma}$-closed for some $\sigma\in\reg(\kappa)$,
  and that $c$ is \emph{closed} if it is $\kappa$-closed.
\end{defn}

\begin{fact}[\cite{paper34}]\label{l23}\label{increasechi}\label{predecessor}
 Suppose that $\lambda$ is an uncountable cardinal and $\theta\in\reg(\lambda^+)$.
  \begin{enumerate}
    \item There exists a closed witness to  $\U(\lambda^+,\lambda^+,\lambda,\lambda)$;
    \item If $\lambda$ is regular, then there exists a closed witness to $\U(\lambda^+,\lambda^+,\theta,\lambda)$;
    \item If there exists a tail-closed witness to $\U(\lambda^+,2,\theta,2)$,
  then there exists a closed witness to $\U(\lambda^+,\lambda^+,\theta,\cf(\lambda))$.
  \end{enumerate}
\end{fact}

\begin{fact}[\cite{paper34}]\label{pumpclosed}
  Suppose that $c:[\kappa]^2\rightarrow\theta$ is a coloring and $\omega\le\chi<\kappa$.
  Then $(1)\implies(2)\implies(3)$:
  \begin{enumerate}
    \item For some stationary $\Sigma\s E^\kappa_{\ge\chi}$, $c$ is a $\Sigma$-closed
      witness to $\U(\kappa,2,\theta,\chi)$.
    \item For every family $\mathcal A\s[\kappa]^{<\chi}$ consisting of $\kappa$-many
      pairwise disjoint sets, for every club $D\s\kappa$, and for every $i<\theta$,
      there exist $\gamma\in D$, $a\in\mathcal A$, and $\epsilon < \gamma$ such that
      \begin{itemize}
        \item $\gamma < a$;
        \item for all $\alpha \in (\epsilon, \gamma)$ and all $\beta\in a$,
        we have $c(\alpha,\beta)>i$.
      \end{itemize}
      \item $c$ witnesses $\U(\kappa,\kappa,\theta,\chi)$.
  \end{enumerate}
\end{fact}

We next recall some of the basic definitions from the theory of \emph{walks on ordinals},
which is our primary technique for constructing witnesses to $\U(\kappa, \ldots)$.

\begin{defn}[\cite{MR908147}]\label{walks}
  Given a $C$-sequence $\langle C_\alpha\mid\alpha<\kappa\rangle$, we derive various functions as follows.
  For all $\alpha<\beta<\kappa$,
  \begin{itemize}
    \item $\Tr(\alpha,\beta)\in{}^\omega\kappa$ is defined recursively by letting, for all $n<\omega$,
      $$\Tr(\alpha,\beta)(n):=\begin{cases}
      \beta & \text{if } n=0;\\
      \min(C_{\Tr(\alpha,\beta)(n-1)}\setminus\alpha) & \text{if } n>0\ \&\ \Tr(\alpha,\beta)(n-1)>\alpha;\\
      \alpha & \text{otherwise};
      \end{cases}$$
    \item (Number of steps) $\rho_2(\alpha,\beta):=\min\{n<\omega\mid \Tr(\alpha,\beta)(n)=\alpha\}$;
    \item (Upper trace) $\tr(\alpha,\beta):=\Tr(\alpha,\beta)\restriction \rho_2(\alpha,\beta)$.
  \end{itemize}
\end{defn}
\begin{remark}
  To avoid notational confusion, note that there is no relationship between the two-place instance $\Tr(\alpha,\beta)$
  and the one-place instance $\Tr(S)$.
\end{remark}

\begin{defn}[\cite{paper18}]  For $\gamma<\beta<\kappa$, let
  $$\lambda_2(\gamma,\beta):=\sup(\gamma\cap\{ \sup(C_\tau\cap\gamma)
  \mid \tau \in \im(\tr(\gamma, \beta))\}).$$
\end{defn}

Note that $\lambda_2(\gamma, \beta) < \gamma$ whenever $0<\gamma < \beta < \kappa$.
To motivate the preceding definition, let us point out the following fact.

\begin{fact}[\cite{paper34}]\label{lambda2}\label{rho2_closed}
  Suppose that $\lambda_2(\gamma,\beta)<\alpha<\gamma<\beta<\kappa$.
  Then $\tr(\gamma,\beta)\sq \tr(\alpha,\beta)$ and one of the following cases holds:
  \begin{enumerate}
    \item $\gamma\in\im(\tr(\alpha,\beta))$; or
    \item $\gamma\in\acc(C_\delta)$ for $\delta:=\min(\im(\tr(\gamma,\beta)))$.
      In particular, $\gamma\in\acc(C_\delta)$ for some  $\delta\in\im(\tr(\alpha,\beta))$.
  \end{enumerate}

  As a consequence, $\rho_2:[\kappa]^2\rightarrow\omega$ is closed.
\end{fact}

\subsection{From $C$-sequences to closed colorings}

In this subsection, we derive closed witnesses to $\U(\kappa, \ldots)$ from the existence of
$C$-sequences witnessing that certain infinite cardinals are in $\spec(\kappa)$. As we shall see,
there is a relationship between elements of $\spec(\kappa)$ and the third and fourth parameters
of $\U(\kappa, \ldots)$.

\begin{lemma}\label{cnontrivial}
  There exists a closed witness to $\U(\kappa,\kappa,\omega,\chi(\kappa))$.
\end{lemma}

\begin{proof}
  To avoid trivialities, suppose that $\chi(\kappa)>1$, so that $\chi(\kappa)\ge\omega$.
  The proof strategy is identical to that of  \cite[Theorem~6.3.6]{todorcevic_book}.
  Using Theorem~\ref{realized}, let us fix a $C$-sequence $\vec C=\langle C_\beta\mid\beta<\kappa\rangle$ with $\chi(\vec C)=\chi(\kappa)$.
  Let $\rho_2:[\kappa]^2\rightarrow\omega$ denote the characteristic function derived from walking along $\vec C$.
  By Fact~\ref{rho2_closed}, $\rho_2$ is closed. We will show that $\rho_2$ witnesses $\U(\kappa, \kappa, \omega, \chi(\kappa))$.

  \begin{claim}\label{claim5172}
    For every $\chi<\chi(\kappa)$, every family $\mathcal A\s[\kappa]^{\chi}$ consisting of $\kappa$-many pairwise disjoint sets, and every club $D\s\kappa$,
    there exist $\gamma\in D\cap E^\kappa_{>\chi}$ and $a\in\mathcal A$ such that
    \begin{itemize}
      \item $\gamma < a$;
      \item $\gamma\notin\bigcup_{\beta\in a}C_\beta$.
    \end{itemize}
  \end{claim}

  \begin{cproof}
    Since $\chi<\chi(\kappa)$, and by the choice of $\vec C$, we may pick $\alpha<\kappa$ such that, for all $a \in \mathcal{A}$,
    $D\cap E^\alpha_{>\chi}\nsubseteq\bigcup_{\beta\in a} C_\beta$.
    Pick an arbitrary $a\in\mathcal A$ with $a>\alpha$,
    and then pick $\gamma\in D\cap E^\alpha_{>\chi}\setminus\bigcup_{\beta\in a} C_\beta$.
  \end{cproof}

  \begin{claim}
    For every $\chi<\chi(\kappa)$, every family $\mathcal A\s[\kappa]^{\chi}$ consisting of $\kappa$-many pairwise disjoint sets, and every $n<\omega$,
    there exist $\mathcal B\in[\mathcal A]^\kappa$ such that
    $\min(\rho_2[a\times b])\ge n$ for all $(a,b)\in[\mathcal A]^2$.
  \end{claim}

  \begin{proof} \renewcommand{\qedsymbol}{\ensuremath{\boxtimes \ \square}}
    We proceed by induction on $n$. Suppose that $n<\omega$ and that the claim holds for $n$.
    Fix a family $\mathcal A\s[\kappa]^{\chi}$ consisting of $\kappa$-many pairwise disjoint sets.
    By the preceding claim, we may find a stationary set $\Gamma\s E^\kappa_{>\chi}$ along with a sequence $\langle a_\gamma\mid \gamma\in\Gamma\rangle$
    such that for all $\gamma\in\Gamma$, we have $a_\gamma\in\mathcal A$, $\gamma < a_\gamma$, and $\gamma\notin\bigcup_{\beta\in a_\gamma}C_\beta$.
    For each $\gamma\in\Gamma$, let $\widetilde{a_\gamma}:=\bigcup\{ \im(\tr(\gamma,\beta))\mid \beta\in a_\gamma\}$.

    Define $f:\Gamma\rightarrow\kappa$, $g:\Gamma\rightarrow\kappa$ and $h:\Gamma\rightarrow\chi(\kappa)$ by setting, for all $\gamma \in \Gamma$,
    \begin{itemize}
      \item $f(\gamma):=\sup\{ \sup(C_\beta\cap\gamma)\mid \beta\in a_\gamma\}$;
      \item $g(\gamma):=\sup(a_\gamma)$;
      \item $h(\gamma):=|\widetilde{a_\gamma}|$.
    \end{itemize}

    Pick $\epsilon<\kappa$ and $\chi'<\chi(\kappa)$ for which
    \[
      A:=\{ \gamma\in\Gamma\mid f(\gamma)=\epsilon\ \&\ g[\gamma]\s\gamma\ \&\ h(\gamma)=\chi'\}
    \]
    is stationary. By the hypothesis on $n$, there exists $B\in[A]^\kappa$ such that for every $(\gamma,\gamma')\in[B]^2$,
    we have $\min(\rho_2[\widetilde{a_\gamma}\times \widetilde{a_{\gamma'}}])\ge n$.
    We claim that $\mathcal B:=\{ a_\gamma\mid \gamma\in B\}$ is as sought.
    To see this, fix arbitrary $(\gamma,\gamma')\in[B]^2$ and $(\alpha,\beta)\in a_\gamma\times a_{\gamma'}$.

    Let $\delta:=\min(C_\beta\setminus\gamma')$.
    Then $\epsilon<\gamma<\alpha<\gamma'<\delta<\beta$ and $\min(C_\beta\setminus\alpha)=\min(C_\beta\setminus\gamma')=\delta$,
    so $\tr(\alpha,\beta)=\tr(\delta,\beta)\conc\tr(\alpha,\delta)$.
    Since $(\alpha,\delta)\in \widetilde{a_\gamma}\times \widetilde{a_{\gamma'}}$, we have $\rho_2(\alpha,\delta)\ge n$,
    and hence $\rho_2(\alpha,\beta)\ge n+1$.
  \end{proof}
  \let\qed\relax
\end{proof}

\begin{remark}
  The preceding construction is not the only way of obtaining instances of $\U(\kappa,\kappa,\theta,\chi)$ with $\theta:=\omega$.
  By an analysis from the forthcoming Part~III of this series, in the model of Theorem~\ref{thm51},
  in which $\chi(\kappa)=\omega$, there is a closed witness to $\U(\kappa,\kappa,\omega,\kappa)$.
  Also, by Theorem~\ref{t55}, if $\lambda$ is a limit of an $\omega$-sequence of strongly compact cardinals, then $\chi(\lambda^+)=\omega$,
  but, by \cite[Corolllary~4.13]{paper34}, there is a closed witness to $\U(\lambda^+,\lambda^+,\omega,\lambda)$.
\end{remark}
\begin{cor} Suppose that $\kappa$ is strongly inaccessible and $\theta'\in E^\kappa_{>\omega}$ is a cardinal.
If every $\kappa$-Aronszajn tree admits a $\theta'$-ascent path, then $\chi(\kappa)\le\theta'$.
\end{cor}
\begin{proof} Suppose that $\chi:=\chi(\kappa)$ is greater than $\theta'$.
Then, by Lemma~\ref{cnontrivial}, $\U(\kappa,\kappa,\theta,\chi)$ holds for $\theta:=\omega$.
But, then, by Fact~\ref{lemma610}(2), the $\kappa$-Aronszajn tree $\mathcal T(c)$ cannot admit a $\theta'$-ascent path. This is a contradiction.
\end{proof}

\begin{thm}\label{thm411}
  If $\chi\in\spec(\kappa)$, then there exists a closed witness to $\U(\kappa,\kappa,\chi,\chi)$.
\end{thm}

\begin{proof}
  Fix $\chi \in \spec(\kappa)$. As $c:[\kappa]^2\rightarrow\kappa$ defined by letting
  $c(\alpha,\beta):=\min\{\alpha,\beta\}$ is a closed witness to $\U(\kappa,\kappa,\kappa,\kappa)$,
  we may assume that $\chi<\kappa$.
  By Fact~\ref{predecessor}, we may furthermore assume that $\chi^+<\kappa$.
  So, by Lemma~\ref{cnontrivial}, we may altogether assume that $\omega<\chi<\chi^+<\kappa$.

  Fix a $C$-sequence $\langle C_\beta\mid\beta<\kappa\rangle$ with $\chi(\vec C)=\chi$.
  Fix $\Delta\in[\kappa]^\kappa$ and $b:\kappa\rightarrow{}^{\chi}\kappa$ such that, for all $\alpha<\kappa$,
  $\alpha < \im(b(\alpha))$ and
  $\Delta\cap(\alpha+1)\s\bigcup_{\beta\in\im(b(\alpha))}C_\beta$.
  As $\chi<\sup(\reg(\kappa))$, by Lemma~\ref{lemma53}, we may assume that the stationary set $\Sigma:=\acc^+(\Delta)\cap E^\kappa_{>\chi}$ is a subset of $\Delta$.
  Define a coloring $c:[\kappa]^2\rightarrow\chi$ by stipulating that
  $$c(\gamma,\delta):=\min\{i<\chi\mid \min(\Sigma\setminus\gamma)\in C_{b(\min(\Sigma\setminus\delta))(i)}\}.$$

  \begin{claim}
    $c$ is $\Sigma$-closed.
  \end{claim}
  \begin{cproof}
    Suppose that $\delta<\kappa$, $i < \chi$, and $A\s D^c_{\leq i}(\delta)$,
    with $\gamma := \sup(A)$ in $\Sigma\cap\delta\setminus A$.
    We shall show that $c(\gamma,\delta)\le i$. We commence with a few simplifications.
    \begin{itemize}
      \item By the definition of $c$, we may assume that $\delta=\min(\Sigma\setminus\delta)$.
      \item We may assume that $\min(\Sigma\setminus\alpha)<\gamma$ for all $\alpha\in A$,
    as otherwise we would have $\min(\Sigma \setminus \alpha) = \min(\Sigma \setminus \gamma)$,
    and hence $c(\gamma, \delta) = c(\alpha, \delta) \leq i$.
    In particular, by replacing each element $\alpha \in A$ with $\min(\Sigma \setminus \alpha)$,
    we may assume that $A$ is a cofinal subset of $\Sigma\cap\gamma$.
      \item As $\cf(\gamma)>\chi$, we can thin out $A$ and assume that there is $i^* \leq i$ such that
    $c(\alpha,\delta)=i^*$ for all $\alpha\in A$.
    \end{itemize}

    It follows that $A\s C_{b(\delta)(i^*)}$. As the latter is closed, we conclude that $\gamma\in C_{b(\delta)(i)}$, so $c(\gamma,\delta)\le i$.
  \end{cproof}

  \begin{claim}
    Suppose that $\omega\le\chi'<\chi$, $\mathcal A\s[\kappa]^{\chi'}$ is a family consisting of
    $\kappa$-many pairwise disjoint sets,  $D\s\kappa$ is a club, and $i<\chi$ is a prescribed color.
    Then there exist $\gamma\in D\cap\Sigma$ and $a\in\mathcal A$ such that
    \begin{itemize}
      \item $\gamma < a$;
      \item for all $\beta\in a$, $c(\gamma,\beta)>i$.
    \end{itemize}
  \end{claim}

  \begin{cproof}
    Suppose not, and set $\nabla:=D\cap\Sigma$.
    We shall obtain a contradiction to the choice of $\vec C$ by showing that for every $\alpha<\kappa$, there exists some
    $b_\alpha\in[\kappa]^{< \chi}$ such that $\nabla\cap\alpha\s\bigcup_{\delta\in b_\alpha}C_\delta$.

    To this end, let $\alpha<\kappa$ be arbitrary.
    Pick $a\in\mathcal A$ with $a>\alpha$,
    and set
    \[
      b_\alpha:=\{ b(\min(\Sigma\setminus\beta))(j)\mid \beta\in a, ~ j\le i\},
    \]
    so $|b_\alpha|<\chi$. Now, let $\gamma\in\nabla\cap\alpha$ be arbitrary.
    As $a>\alpha>\gamma$, we may find some $\beta\in a$ such that $c(\gamma,\beta)\le i$.
    As $\min(\Sigma\setminus\gamma)=\gamma$, this means that $\gamma\in C_{b(\min(\Sigma\setminus\beta))(j)}$ for some $j\le i$. That is, $\gamma\in C_\delta$ for some $\delta\in b_\alpha$, as sought.
  \end{cproof}

  As $c$ is $\Sigma$-closed and $\Sigma\s E^\kappa_{\ge\chi}$, the preceding claim, together with the implication $(2)\implies(3)$
  of Fact~\ref{pumpclosed}, implies that $c$ witnesses $\U(\kappa,\kappa,\chi,\chi)$.
  By Lemma~\ref{lemma310} below, then, there exists a closed witness to $\U(\kappa,\kappa,\chi,\chi)$.
\end{proof}

\begin{cor} Suppose that $\kappa$ is strongly inaccessible,
and every $\kappa$-Aronszajn tree admits an $\omega$-ascent path.
Then $\spec(\kappa)\cap\kappa\s E^\kappa_\omega$.
\end{cor}
\begin{proof}
Let $\chi\in\spec(\kappa)\cap\kappa$ be arbitrary.
Set $\theta:=\cf(\chi)$, so that $\theta\in\reg(\kappa)$.
By Theorem~\ref{thm411}, $\U(\kappa,\kappa,\theta,\chi)$ holds.

Let $\theta':=\omega$.
Now, if $\cf(\chi)\neq\omega$, then $\chi>\theta'$ and $\cf(\theta')\neq\theta$, so that, by Fact~\ref{lemma610}(2), the $\kappa$-Aronszajn tree $\mathcal T(c)$ cannot admit a $\theta'$-ascent path. This is a contradiction.
\end{proof}

\subsection{From closed colorings to $C$-sequences}

\begin{defn}
  A coloring $c:[\kappa]^2\rightarrow\theta$ is said to have the \emph{covering property}
  if, for every $\alpha<\kappa$, there is an injection $f_\alpha:\theta\rightarrow \acc(\kappa\setminus\alpha)\cap E^\kappa_{\neq\theta}$
  such that $\alpha\setminus\bigcup_{i<\theta}D^{c}_{\le i}(f_\alpha(i))$ is bounded below $\alpha$.
\end{defn}

In the next subsection, we shall see how, in certain circumstances, we can
derive colorings with the covering property that witness $\U(\ldots)$. For now,
we show that the existence of such colorings provides information about the
$C$-sequence spectrum.

\begin{lemma}\label{lemma511}
  Suppose that $\omega\le\chi\le \theta = \cf(\theta) <\kappa$,
  $\Sigma\s E^\kappa_{\ge\chi}$ is stationary,
  and $c:[\kappa]^2\rightarrow\theta$ is a coloring.
  Then there exists a corresponding $C$-sequence $\vec C$ over $\kappa$ satisfying the following conditions.
  \begin{enumerate}
    \item If $c$ is a $\Sigma$-closed witness to $\U(\kappa,2,\theta,\chi)$, then $\chi(\vec C)\ge\chi$.
    \item If $c$ has the covering property, then $\chi(\vec C)\le\theta$.
  \end{enumerate}
\end{lemma}

\begin{proof}
  Define $f:\acc(\kappa)\rightarrow\theta+1$ by letting, for all $\beta\in\acc(\kappa)$,
  $$
    f(\beta):=\min\{ i\le \theta\mid \sup(D^c_{\le i}(\beta))=\beta\}.
  $$
  Let $\Gamma:=\{\beta\in\acc(\kappa)\mid f(\beta)<\theta\}$. Note that $\acc(\kappa)\setminus\Gamma\s E^\kappa_{\theta}$.
  At this stage, for each $\beta\in\Gamma$, we fix $i(\beta)\in[f(\beta),\theta)$ arbitrarily.
  Later on, in our handling of Clause~(2), we shall make a more educated choice of $i(\beta)$.
  Now, pick a $C$-sequence $\vec C=\langle C_\beta\mid\beta<\kappa\rangle$ such that, for all
  $\beta\in\Gamma$, $C_\beta=\cl(D^c_{\le i(\beta)}(\beta))$,
  and for each $\beta\in\acc(\kappa)\setminus\Gamma$, $\otp(C_\beta)=\theta$.

  (1) Suppose that $c$ is a $\Sigma$-closed witness to $\U(\kappa,2,\theta,\chi)$.
  Towards a contradiction, suppose that $\chi':=\chi(\vec C)$ is smaller than $\chi$.
  In particular, $\Sigma$ is a stationary subset of $E^\kappa_{>\chi'}$.
  Using Lemma~\ref{lemma53}, fix $\Delta \in [\kappa]^\kappa$ and a progressive $b:\kappa \rightarrow [\kappa]^{\chi'}$
  such that $\acc^+(\Delta)\cap E^\kappa_{>\chi'}\s\Delta$ and
  $\Delta \cap \alpha \s \bigcup_{\beta \in b(\alpha)} C_\beta$ for all
  $\alpha < \kappa$. Note that for every $\alpha<\kappa$,
  $b_\alpha:=\{\alpha\}\cup b(\alpha)$ has cardinality less than $\chi$ and hence less than the regular cardinal $\theta$.
  In particular, we may find an $i<\theta$ for which the following set is stationary:
  $$
    S:=\{\alpha\in\acc^+(\Delta)\cap\Sigma\mid \sup\{i(\beta)\mid \beta\in b_\alpha\cap\Gamma\}=i\}.
  $$
  Pick a sparse enough $A\in[S]^\kappa$ such that, for every $(\alpha,\alpha')\in[A]^2$, we have
  $b_\alpha < \alpha'$. Then $\mathcal A:=\{b_\alpha \mid \alpha\in A\}$
  consists of $\kappa$-many pairwise disjoint sets.

  As, by the implication $(1)\implies(3)$ of Fact~\ref{pumpclosed},
  $c$ witnesses $\U(\kappa,\kappa,\theta,\chi)$,
  we can find $B\in[A]^\kappa$ such that
  $\min(c[b_\alpha\times b_{\alpha'}])>i$ for every $(\alpha,\alpha')\in[B]^2$.

  Let $\alpha'\in B$ be such that $\otp(B\cap\alpha') = \theta^2$.
  Since
  \[
    B\s A\s S\s \acc^+(\Delta)\cap\Sigma\s \acc^+(\Delta) \cap E^\kappa_{>\chi'}\s\Delta,
  \]
  we have $B\cap\alpha'\s\bigcup_{\beta\in b_{\alpha'}}C_\beta$.
  As $|b_{\alpha'}|<\theta = \cf(\theta)$ and $\otp(B\cap\alpha') = \theta^2$,
  we may find some $\beta\in b_{\alpha'}$ such that $\otp(B\cap\alpha'\cap C_\beta) = \theta^2$.
  In particular, $\otp(C_\beta)>\theta$,
  so $\beta\in\Gamma$ and hence $i(\beta)\le i$. Pick $\alpha\in B\cap\alpha'\cap C_\beta$.
  Then $\alpha\in \Sigma\cap C_\beta\s D^c_{\le i}(\beta)$,
  so $c(\alpha,\beta)\le i$, contradicting the fact that $(\alpha,\alpha')\in[A]^2$
  and $(\alpha,\beta)\in b_\alpha\times b_{\alpha'}$.

  (2) Suppose that $c$ has the covering property, as witnessed by a sequence of functions $\langle f_\alpha\mid\alpha<\kappa\rangle$.
  Fix $\epsilon<\kappa$ and a stationary $S\s\kappa$ such that, for all $\alpha\in S$,
  $\sup(\alpha\setminus\bigcup_{i<\theta}D^{c}_{\le i}(f_\alpha(i)))=\epsilon$.
  As $\min(\im(f_\alpha))>\alpha$ for all $\alpha<\kappa$, let us pick a sparse enough $A\in[S]^\kappa$ such that,
  for all $(\alpha,\alpha')\in[A]^2$, we have $\im(f_\alpha)\cap\im(f_{\alpha'})=\emptyset$.
  In particular, for every $\beta<\kappa$, there exists at most a single
  pair $(\alpha,j)\in A\times\theta$ satisfying $f_\alpha(j)=\beta$.
  Now, let us revisit our definition of $i(\beta)$ from the beginning of our proof, requiring,
  for all $\beta\in\Gamma$, not only that $i(\beta)\ge f(\beta)$, but also that if $\alpha \in A$ and
  $f_\alpha(j)=\beta$, then $i(\beta)\ge j$.

  Let $\Delta:=\kappa\setminus(\epsilon+1)$, and define $b:\kappa\rightarrow[\Gamma]^\theta$
  by letting $b(\alpha):=\im(f_{\min(A\setminus\alpha)})$.
  The function $b$ is well-defined, since, for all $\alpha<\kappa$, $\im(f_\alpha)\s \acc(\kappa)\setminus E^\kappa_{\theta}\s \Gamma$.
  We claim that $\Delta$ and $b$ witness that $\chi(\vec C)\le\theta$.
  To see this, let $\alpha$ be arbitrary, and write $\alpha':=\min(A\setminus\alpha)$.
  By the choice of $A$ and $\epsilon$, we have
  \[
    \Delta\cap\alpha\s\Delta\cap\alpha'\s \bigcup_{j<\theta}D^{c}_{\le j}(f_{\alpha'}(j)))\s
    \bigcup_{\beta\in b(\alpha)} D^c_{\le i(\beta)}(\beta)\s\bigcup_{\beta\in b(\alpha)}C_\beta,
  \]
  as sought.
\end{proof}

\begin{cor}\label{lemma33}
  Suppose that $\omega\le\chi\le \theta = \cf(\theta) <\kappa$ and there exists a
  $\Sigma$-closed witness to $\U(\kappa,2,\theta,\chi)$ for some stationary
  $\Sigma\s E^\kappa_{\ge\chi}$. Then $\chi(\kappa)\ge\chi$.\qed
\end{cor}

\begin{remark}
   The hypothesis ``$\chi\le\cf(\theta)$'' in the preceding corollary cannot be waived.
   Indeed, by a result from Part~III of this series, $\square^{\ind}(\kappa,\theta)$ implies the existence of a closed witness to $\U(\kappa,\kappa,\theta,\sup(\reg(\kappa))$,
   but in the model of Theorem \ref{thm51}, we have $\chi(\kappa)<\sup(\reg(\kappa))$.
\end{remark}

\subsection{From colorings to closed colorings}

Fact~\ref{increasechi} shows that, at the level of successor cardinals,
the existence of a tail-closed witness to $\U(\ldots)$ entails the existence of fully closed witnesses.
The next lemma provides certain circumstances in which we can improve certain somewhere-closed
witnesses to $\U(\ldots)$ to fully closed witnesses, while also gaining the covering property.

\begin{lemma}\label{lemma310}
  For every coloring $c:[\kappa]^2\rightarrow\theta$,   there exists a corresponding
  coloring $d:[\kappa]^2\rightarrow\theta$ satisfying the following conditions.
  \begin{enumerate}
    \item $d$ is closed;
    \item If $\kappa\ge\aleph_2$, then $d$ has the covering property;
    \item For every infinite cardinal $\chi\le\chi(\kappa)$,
      if $c$ is a $\Sigma$-closed witness to $\U(\kappa,2,\theta,\chi)$
      for some stationary $\Sigma\s E^\kappa_{\ge\chi}$,
      then $d$ witnesses $\U(\kappa,\kappa,\theta,\chi)$.
  \end{enumerate}
\end{lemma}
\begin{proof}
  To avoid trivialities, suppose that $\chi(\kappa)$ is an infinite cardinal.
  By Fact~\ref{l23}(2), we may also assume that $\kappa\ge\aleph_2$.
  As the function $d:[\kappa]^2\rightarrow\kappa$ defined by letting $d(\alpha,\beta):=\min\{\alpha,\beta\}$
  is a closed witness to $\U(\kappa,\kappa,\kappa,\kappa)$ that has the covering property,
  we may also assume that $\theta\in\reg(\kappa)$.

  \begin{claim}\label{realized2}
    There exists a $C$-sequence $\vec C=\langle C_\beta \mid \beta < \kappa \rangle$ such that $\chi(\vec C)=\chi(\kappa)$,
    and, for all $\alpha\in E^\kappa_\theta$, there are stationarily many $\beta\in E^\kappa_{\neq\theta}$ such that $C_\alpha=C_\beta\cap\alpha$.
  \end{claim}

  \begin{cproof}
    Let $\langle S_\alpha\mid\alpha<\kappa\rangle$ be a partition of $E^\kappa_{\neq\theta}$ into stationary sets.
    For every $\beta\in E^\kappa_{\neq\theta}$, let $\alpha(\beta)$ denote the unique ordinal $\alpha<\kappa$ such that $\beta\in S_{\alpha}$.
    Using Theorem~\ref{realized}, fix a $C$-sequence $\vec C=\langle C_\beta \mid \beta < \kappa \rangle$ such that $\chi(\vec C)=\chi(\kappa)$.
    Now, define a $C$-sequence $\vec C^\bullet=\langle C^\bullet_\beta\mid\beta<\kappa\rangle$ by setting, for all $\beta<\kappa$,
    \[
      C^\bullet_\beta:=\begin{cases}
      C_{\alpha(\beta)}\cup\{\alpha(\beta)\}\cup(C_\beta\setminus\alpha(\beta)) &\text{if }\beta\in\acc(\kappa)\cap E^\kappa_{\neq\theta}\text{ and }\alpha(\beta)\in E^\beta_\theta;\\
      C_\beta &\text{otherwise}.
      \end{cases}
    \]
    Clearly, for every $\alpha\in E^\kappa_\theta$, for a tail of $\beta\in S_\alpha$, we have $\alpha(\beta)=\alpha<\beta$,
    so $C^\bullet_\beta\cap\alpha=(C_{\alpha}\cup\{\alpha\}\cup(C_\beta\setminus\alpha))\cap\alpha=C_\alpha=C^\bullet_\alpha$.
    In addition, every club in $\vec C^\bullet$ is covered by at most three clubs from $\vec C$,
    so $\chi(\vec C^\bullet)=\chi(\vec C)=\chi(\kappa)$.
  \end{cproof}
  Let $\vec C$ be given by the claim.
  Walk along $\vec C$ and derive $d:[\kappa]^2\rightarrow\theta$ by setting,
  for all $\alpha < \beta < \kappa$,
  \[
    d(\alpha,\beta):=\begin{cases}\max\{ c(\tr(\alpha,\beta)(n),\tr(\alpha,\beta)(1))\mid 1<n<\rho_2(\alpha,\beta)\} &\text{if }\rho_2(\alpha,\beta)>2;\\
    0 &\text{otherwise}.\end{cases}
  \]

  \begin{claim}
    $d$ is closed.
  \end{claim}
  \begin{cproof}
    Suppose that $\beta<\kappa$, $i < \theta$, and $A\s D^d_{\leq i}(\beta)$,
    with $\gamma := \sup(A)$ a limit ordinal less than $\beta$. To see that $\gamma \in D^d_{\leq i}(\beta)$,
    fix $\alpha\in A$ above $\lambda_2(\gamma,\beta)$. By Fact~\ref{lambda2},
    $\tr(\gamma,\beta)\sq \tr(\alpha,\beta)$, and hence, by definition of $d$,
    we have $d(\gamma,\beta)\le d(\alpha,\beta)\le i$.
  \end{cproof}

  \begin{claim}
    $d$ has the covering property.
  \end{claim}
  \begin{cproof}
    Let $\alpha<\kappa$ be arbitrary, and put $\alpha':=\min(E^\kappa_\theta\setminus\alpha)$.
    By the choice of $\vec C$, we can fix an injection $f_\alpha:\theta\rightarrow\acc(\kappa\setminus\alpha)\cap E^\kappa_{\neq\theta}$
    with the property that for all $\beta\in\im(f_\alpha)$, we have $C_\beta\cap\alpha'=C_{\alpha'}$.
    To see that $\alpha=\bigcup_{i<\theta}D^{d}_{\le i}(f_\alpha(i))$, let $\eta<\alpha$ be arbitrary
    and set $i:=d(\eta,\alpha')$ and $\beta:=f_\alpha(i)$.
    As $\eta<\alpha'$ and $C_\beta\cap\alpha'=C_{\alpha'}$, we have $\Tr(\eta,\beta)(n)=\Tr(\eta,\alpha')(n)$ for all $n\ge1$,
    so $d(\eta, \beta) = d(\eta, \alpha')$ and hence $\eta\in D^{d}_{\le i}(f_\alpha(i))$.
  \end{cproof}

  Suppose now that  $\chi \le \chi(\kappa)$ is an infinite cardinal, that
  $\Sigma\s E^\kappa_{\ge\chi}$ is stationary (in particular, $\chi<\kappa$),
  and that $c$ is a $\Sigma$-closed witness to $\U(\kappa,2,\theta,\chi)$.
  By the implication $(1)\implies(3)$ of Fact~\ref{pumpclosed}, to show that $d$ witnesses $\U(\kappa,\kappa,\theta,\chi)$,
  it suffices to verify that $d$ witnesses $\U(\kappa,2,\theta,\chi)$.
  To this end,  suppose that $i<\theta$ is a prescribed color, $\chi' < \chi$,
  and $\mathcal A\s\mathcal [\kappa]^{\chi'}$ is a family of $\kappa$-many pairwise disjoint sets.

  By Claim~\ref{claim5172}, we may fix a sequence $\langle a_\iota\mid \iota\in H\rangle$
  such that $H$ is a stationary subset of $E^\kappa_{\ge\chi}$ and, for all $\iota\in H$, we have
  $a_\iota\in\mathcal A$, $a_\iota>\iota$ and $\iota\notin\bigcup_{\beta\in a_\iota}C_\beta$.
  Let $\mathcal B:=\{ \{\iota\}\cup\{\min(C_\beta\setminus\iota)\mid \beta\in a_\iota\} \mid \iota\in H\}$.
  By shrinking $H$, we may assume that
  \begin{itemize}
  \item $\mathcal B$ consists of pairwise disjoint sets, and
  \item there exists $\varepsilon<\kappa$ such that, for all $\iota\in H$,
  $\sup\{\sup(C_\beta\cap\iota)\mid \beta\in a_\iota\}=\varepsilon$.
  \end{itemize}

  \begin{claim} There is a stationary set $\Gamma\s\Sigma$ such that for all $\gamma\in\Gamma$,
  there are $b_\gamma\in\mathcal B$ and $\epsilon_\gamma < \kappa$ with $b_\gamma >\gamma>\epsilon_\gamma$
  such that for all $\eta \in (\epsilon_\gamma, \gamma]$ and all $\beta'\in b_\gamma$,
  we have $c(\eta,\beta')>i$.
  \end{claim}
  \begin{cproof}
  Let $D$ be an arbitrary club in $\kappa$.
  Find $\mathcal X\s\mathcal P(\kappa)$ consisting of $\kappa$-many pairwise disjoint sets
  such that every $x\in\mathcal X$ is of the form $\{\gamma\} \cup b$ for some
  $\gamma \in \Sigma \cap D$ and $b \in \mathcal B$. As $c$ witnesses $\U(\kappa,2,\theta,\chi)$,
  we may pick $(x,y)\in[\mathcal X]^2$ such that $\min(c[x\times y])>i$.
  Fix $\gamma\in(\Sigma\cap D)\cap x$ and $b\in\mathcal B\cap\mathcal P(y)$.
  Clearly, $\gamma < b$ and $|b|<\chi\le\cf(\gamma)$.

  Now, let $\beta'\in b$ be arbitrary.
  Since $(\gamma,\beta')\in x\times y$, we have $c(\gamma,\beta')>i$.
  Then, as $\gamma\in\Sigma$, there must exist $\epsilon(\gamma,\beta')<\gamma$ such that,
  for all $\eta\in(\epsilon(\gamma,\beta'),\gamma)$, $c(\eta,\beta')>i$.
  Since $\cf(\gamma) \geq \chi > |b|$, we know that $\epsilon := \sup\{\epsilon(\gamma, \beta') \mid \beta' \in b\}$ is less than $\gamma$.
  So $\gamma$ is in $\Sigma\cap D$, and $b_\gamma:=b$ and $\epsilon_\gamma:=\epsilon$ are as sought.
  \end{cproof}

  Fix $\langle (b_\gamma,\epsilon_\gamma)\mid \gamma\in\Gamma\rangle$  as in the preceding claim.
  For all $\gamma\in\Gamma$, set $a^\gamma:=a_{\min(b_\gamma)}$.
  Define $f, g:\Gamma\rightarrow\kappa$ by setting, for all $\gamma < \kappa$,
  \begin{itemize}
    \item $f(\gamma)=\sup\{\epsilon_\gamma, ~ \lambda_2(\gamma,\beta)\mid \beta\in a^\gamma\}$;
    \item $g(\gamma):=\sup(a^\gamma)$.
  \end{itemize}
  Pick $\epsilon<\gamma$ for which
  $$
    S:=\{ \gamma\in\Gamma\mid f(\gamma)=\epsilon\ \&\ g[\gamma]\s\gamma\}
  $$
  is stationary, and consider the stationary set $\Delta:=\acc^+(S\setminus\varepsilon)\cap E^\kappa_{\ge\chi}$.

  \begin{claim}
    There exist $\gamma\in S$ and $\tau\in\Delta\cap\gamma$ such that
    $\min(d[\{\tau\}\times a^\gamma])> i$.
  \end{claim}

  \begin{cproof}
    Suppose not, and
    define $b:\kappa\rightarrow[\kappa]^{\le\chi'}$ as follows. For each $\alpha<\kappa$,
    set $\gamma_\alpha:=\min(S\setminus\alpha)$, and then let 
    $b(\alpha):=\{\min(\im(\tr(\gamma_\alpha,\beta)))\mid \beta\in a^{\gamma_\alpha}\}$.
    We claim that $\Delta\cap\alpha\s\bigcup_{\delta\in b(\alpha)}C_\delta$ for all $\alpha<\kappa$,
    contradicting the fact that $\chi'<\chi(\vec C)$.

    Let  $\alpha<\kappa$ and $\tau\in\Delta\cap\alpha$ be arbitrary.
    We shall find $\delta\in b(\alpha)$ with $\tau\in C_\delta$.
    For notational simplicity, set $\gamma:=\gamma_\alpha$ and $\iota:=\min(b_\gamma)$.
    Note that $$b_\gamma=\{\iota\}\cup\{\min(C_\beta\setminus\iota)\mid \beta\in a_\iota\}.$$

    As $\gamma\in S$ and $\tau\in\Delta\cap\gamma$, our indirect hypothesis
    provides us with some  $\beta\in a^\gamma$ such that $d(\tau,\beta)\le i$.
    Let $\beta':=\min(C_\beta\setminus\iota)$. Clearly, $\beta'\in b_\gamma$.
    Also, as $$\sup(C_\beta\cap\iota)\le\varepsilon<\min(\Delta)\le\tau<\gamma<\iota<\beta'<\beta,$$
    we have $\tr(\tau,\beta)(1)=\min(C_\beta\setminus\tau)=\beta'=\min(C_\beta\setminus\gamma)=\tr(\gamma,\beta)(1)$.
    In particular, $n:=\rho_2(\gamma,\beta)$ is greater than $1$.
    Set $\delta:=\tr(\gamma,\beta)(n-1)$, so that $\delta\in b(\alpha)$.
    We have
    \[
      \lambda_2(\gamma,\beta)\le\epsilon<\min(S)<\min(\Delta)\le\tau<\gamma<\beta,
    \]
    so, by Fact~\ref{lambda2}, $\tr(\gamma,\beta)\sq \tr(\tau,\beta)$ and
    $\gamma\in \im(\tr(\tau,\beta)) \cup \acc(C_\delta)$.

    $\br$ If $\gamma\in\im(\tr(\tau,\beta))$, then $\gamma=\tr(\tau,\beta)(n)$
    and $1<n<\rho_2(\tau,\beta)$,
    meaning that  $d(\tau,\beta)\ge c(\gamma,\beta')>i$,    which is not the case.

    $\br$ If $\gamma\in\acc(C_\delta)$, then we claim that $\tau\in C_\delta$, as desired.
    Indeed, otherwise, for $\eta:=\tr(\tau,\beta)(n)$, we would have $\epsilon_\gamma<\tau<\eta<\gamma$,
    so $1<n<\rho_2(\tau,\beta)$ and $d(\tau,\beta)\ge c(\eta,\beta')>i$, which is not the case.
  \end{cproof}

  Let $\gamma\in S$ and $\tau\in\Delta\cap\gamma$ be given by the preceding claim.
  Since $d$ is closed, for each $\beta\in a^\gamma$, there exists
  $\varepsilon(\beta)<\tau$ such that, for all $\alpha\in(\varepsilon(\beta),\tau)$, $d(\alpha,\beta)>i$.
  Since $\cf(\tau)\ge\chi>|a^\gamma|$, we know that $\sup_{\beta\in a^\gamma}\varepsilon(\beta)<\tau$.
  As $\tau\in\Delta\s\acc^+(S)$, we may then find some $\gamma'\in S$ with
  $\sup_{\beta\in a^\gamma}\varepsilon(\beta)<\gamma'<\tau$. As $\tau\in\acc^+(S)$
  and every element of $S$ is a closure point of $g$, we infer that
  $g[\tau]\s\tau$, and hence $a^{\gamma'}\s(\gamma',\tau)$. It follows that, for all
  $\alpha\in a^{\gamma'}$ and $\beta\in a^\gamma$, we have 
  $\varepsilon(\beta)< \gamma'<\alpha<\tau<\beta$, so $d(\alpha,\beta)>i$.
\end{proof}

The next corollary shows that, contrary to initial appearances, there is
indeed some monotonicity in the third parameter of $\U(\ldots)$. Namely,
If $\U(\kappa,\kappa,\theta,\omega)$ holds, then so does $\U(\kappa,\kappa,\omega,\omega)$.

\begin{cor}
  The following are equivalent:
  \begin{enumerate}
    \item $\chi(\kappa)\ge\omega$;
    \item There is a closed witness to $\U(\kappa,\kappa,\omega,\omega)$;
    \item There is a somewhere-closed witness to $\U(\kappa,2,\theta,\omega)$ for some infinite $\theta<\kappa$.
  \end{enumerate}
\end{cor}
\begin{proof}
$(1)\implies(2)$ follows from Lemma~\ref{cnontrivial}, $(2)\implies(3)$ is trivial, and
$(3)\implies(1)$ follows from Lemma~\ref{lemma33}.
\end{proof}

\begin{cor}
  If there exists a tail-closed witness to $\U(\kappa,2,\theta,\chi)$
  with an infinite $\chi\le\chi(\kappa)$,
  then there exists a closed witness to $\U(\kappa,\kappa,\theta,\chi)$.
\end{cor}
\begin{proof} This follows from Lemma~\ref{lemma310}.
\end{proof}

\subsection{The structure of the $C$-sequence spectrum}

\begin{cor}\label{connection}
  For every $\theta\in\reg(\kappa)$, the following are equivalent:
  \begin{enumerate}
    \item $\theta\in\spec(\kappa)$;
    \item There is a closed witness to $\U(\kappa,\kappa,\theta,\theta)$;
    \item There is a $\Sigma$-closed witness to $\U(\kappa,2,\theta,\theta)$ for some stationary $\Sigma\s E^\kappa_{\ge\theta}$.
  \end{enumerate}
\end{cor}
\begin{proof}
  $(1)\implies(2)$ follows from Theorem~\ref{thm411}, and $(2)\implies(3)$ is trivial.
  It remains to prove $(3)\implies(1)$. If $\kappa=\theta^+$, then by Lemma~\ref{successorlemma}, we have $\theta\in\spec(\kappa)$.
  Thus, suppose that $\theta^+<\kappa$ and that
  there is a $\Sigma$-closed witness to $\U(\kappa,2,\theta,\theta)$ for some stationary $\Sigma\s E^\kappa_{\ge\theta}$.
  By Corollary~\ref{lemma33} (using $\chi=\theta$), we have $\chi(\kappa)\ge\theta$.
  So, by Lemma~\ref{lemma310},
  there exists a closed witness to $\U(\kappa,\kappa,\theta,\theta)$
  that moreover has the covering property.
  It now immediately follows from Lemma~\ref{lemma511} that $\theta\in\spec(\kappa)$.
\end{proof}

\begin{remark}
  The hypothesis that $\theta$ is a regular cardinal cannot be removed from the implication $(3)\implies(1)$.
  For instance, if $\lambda$ is a singular limit of strongly compact cardinals,
  then there is a closed a witness to $\U(\lambda^+,\lambda^+,\lambda,\lambda)$,
  but $\lambda\notin\spec(\lambda^+)$.
\end{remark}

\begin{cor}\label{cor519} If $\spec(\kappa)\neq\emptyset$, then $\min(\spec(\kappa))=\omega$.
\end{cor}
\begin{proof}
  This follows from Lemma~\ref{cnontrivial} and Corollary~\ref{connection}.
\end{proof}

\begin{cor} \label{non_reflecting_cor}
  If $\chi\in\spec(\kappa)$, then $\cf(\chi)\in\spec(\kappa)$.
\end{cor}
\begin{proof}
  There is an elementary proof for this, but let us derive it from previous results.
  Suppose that $\chi\in\spec(\kappa)$.
  By Theorem~\ref{thm411}, there exists a closed witness to $\U(\kappa,\kappa,\chi,\chi)$.
  It follows easily that there exists a closed witness to $\U(\kappa,\kappa,\cf(\chi),\cf(\chi))$.
  So, by $(2)\implies(1)$ of Corollary~\ref{connection}, $\cf(\chi)\in\spec(\kappa)$.
\end{proof}

\begin{cor}\label{cor515} Suppose that $\chi\in\reg(\kappa)$ and there exists a non-reflecting stationary subset of $E^\kappa_{\ge\chi}$.
Then $\reg(\chi^+)\s \spec(\kappa)$.
\end{cor}
\begin{proof} By \cite[Corollary~4.10]{paper34},
the hypothesis implies that, for every $\theta\in\reg(\kappa)$,
there exists a closed witness to $\U(\kappa,\kappa,\theta,\chi)$.
Now, appeal to $(2)\implies(1)$ of Corollary~\ref{connection}.
\end{proof}

\begin{cor}\label{cor521} Any of the following implies that $\reg(\kappa)\s\spec(\kappa)$:
\begin{enumerate}
\item $\kappa$ is a successor of a regular cardinal;
\item $\kappa$ is an inaccessible cardinal which is not Mahlo;
\item $\square(\kappa,{<}\omega)$ holds.
\end{enumerate}
\end{cor}
\begin{proof}
\begin{enumerate}
\item This follows from Corollary~\ref{cor515}.

\item This follows from Corollary~\ref{connection}, using \cite[Theorem~4.23]{paper34}.

\item  Let $\theta\in\reg(\kappa)$ be arbitrary.
By the upcoming Theorem~\ref{thm518}, using $\Gamma:=E^\kappa_\theta$,
we may find a $\square(\kappa,{<}\omega)$-sequence $\langle\mathcal C_\alpha\mid\alpha<\kappa\rangle$ such that,
for every $i<\theta$, $H_i:=\{ \alpha\in E^\kappa_\theta\mid \forall C\in\mathcal C_\alpha[\min(C)=i]\}$ is stationary.
Fix a $C$-sequence $\langle C_\alpha\mid\alpha<\kappa\rangle$ with $C_\alpha\in\mathcal C_\alpha$ for all $\alpha<\kappa$.
For every $\alpha\in\acc(\kappa)$, we have that $\{i<\theta\mid\acc(C_\alpha)\cap H_i\neq\emptyset\}$ is a subset of the singleton $\{\min(C_\alpha)\}$.
So, by \cite[Theorem~4.11]{paper34}, there exists a closed witness to $\U(\kappa,\kappa,\theta,\theta)$.
Then, by Corollary~\ref{connection}, $\theta\in\spec(\kappa)$.\qedhere
\end{enumerate}
\end{proof}

The next theorem is motivated by Clause~(3) of the preceding. It simultaneously improves \cite[Lemma~3.2]{paper18} and \cite[Theorem~2.8]{hayut_lh} by blending the originals proofs together with some arguments from \cite[\S1]{paper29}.

\begin{thm}\label{thm518} Suppose that $\langle\mathcal C_\alpha\mid\alpha<\kappa\rangle$ is a $\square(\kappa,{<}\omega)$-sequence.
For every stationary $\Gamma\s\kappa$, there exists a $\square(\kappa,{<}\omega)$-sequence  $\langle\mathcal C_\alpha^\bullet\mid\alpha<\kappa\rangle$
such that:
\begin{enumerate}
\item for all $\alpha<\kappa$, $|\mathcal C^\bullet_\alpha|\le|\mathcal C_\alpha|$;
\item for all $i<\kappa$, $\{ \alpha\in\Gamma\mid \forall C\in\mathcal C^\bullet_\alpha[\min(C)=i]\}$ is stationary.
\end{enumerate}
\end{thm}
\begin{proof} Let $\Gamma$ be an arbitrary stationary subset of $\kappa$,
  and, for all $\epsilon, \delta < \kappa$, set $\Gamma_{\epsilon,\delta}:=\{ \alpha\in \Gamma\mid \forall C\in\mathcal C_\alpha[\min(C\setminus\epsilon)\ge\delta]\}$.

\begin{claim} There exists an ordinal $\epsilon<\kappa$ such that, for all $\delta<\kappa$,  $\Gamma_{\epsilon,\delta}$ is stationary.
\end{claim}
\begin{cproof}
Suppose not. Then, for every $\epsilon<\kappa$, we fix $\delta(\epsilon)<\kappa$ and a club $D_\epsilon$
such that $\Gamma_{\epsilon,\delta(\epsilon)}\cap D_\epsilon=\emptyset$.
Consider the club $D:=\left\{\alpha\in\diagonal_{\epsilon<\kappa}D_\epsilon ~ \middle| ~ \forall \epsilon<\alpha[\delta(\epsilon)<\alpha]\right\}$.
\begin{subclaim} There exist $\alpha\in D\cap\Gamma$ and $\beta\in D\cap\alpha$ such that $\beta\notin\bigcup\mathcal C_\alpha$.
\end{subclaim}
\begin{scproof} Suppose not. Then, for every $\alpha\in D\cap\Gamma$, we have $D\cap\alpha\s\bigcup\mathcal C_\alpha$.
In particular, there exists a least positive integer $n_\alpha$ along with $\mathcal C'_\alpha\in[\mathcal C_\alpha]^{n_\alpha}$
such that $D\cap\alpha\setminus\bigcup\mathcal C'_\alpha$ is bounded below $\alpha$.
Fix a stationary $S\s\Gamma\cap D$ and some $n$ such that $n_\alpha=n$ for all $\alpha\in S$.
By possibly shrinking $S$, we may also assume the existence of some $\varepsilon<\kappa$
such that, for all $\alpha\in S$, we have $D\cap\alpha\setminus\bigcup\mathcal C'_\alpha\s\varepsilon$.

Consider the stationary sets $B:=\acc^+(S\setminus\varepsilon)\cap S$ and $A:=\acc^+(B)\cap S$.
Let $\alpha\in A$ and  $\beta\in B\cap\alpha$ be arbitrary.
As $\alpha\in S$, we have $D\cap\alpha\setminus\bigcup\mathcal C'_\alpha\s\varepsilon$.
In particular, $D\cap\beta\setminus\bigcup\{ C\cap\beta\mid C\in \mathcal C'_\alpha\}\s\varepsilon$.
But $\beta>\varepsilon$ and $\mathcal C_\alpha'$ is finite, and hence also
\[
  D\cap\beta\setminus\bigcup\left\{ C\cap\beta\mid C\in \mathcal C'_\alpha, ~
  \sup(C\cap\beta)=\beta\right\}
\]
  is bounded below $\beta$.
Let $\mathcal D_\beta:=\{ C\cap\beta\mid C\in\mathcal C_\alpha', \sup(C\cap\beta)=\beta\}$.
By coherence, $\mathcal D_\beta\s\mathcal C_\beta$, and as $D\cap\beta\setminus\bigcup\mathcal D_\beta$ is bounded below $\beta$,
the fact that $\beta\in S$ implies that $|\mathcal D_\beta|=n_\beta=n=n_\alpha=|\mathcal C_\alpha'|$.
It follows that, for all $C\in\mathcal C_\alpha'$, we have $\beta\in\acc(C)$.

Thus, we have established that for every $C$-sequence $\vec C=\langle C_\alpha\mid\alpha\in\acc(\kappa)\rangle$ such that:
\begin{itemize}
\item for all $\alpha\in\acc(\kappa)$, $C_\alpha\in\mathcal C_\alpha$;
\item for all $\alpha\in S$, $C_\alpha\in\mathcal C_\alpha'$,
\end{itemize}
we have that $\{ \alpha\in\acc(\kappa)\mid B\cap\alpha\s C_\alpha\}$ covers the stationary set $A$,
so, $\vec C$ is not amenable.
However, this contradicts \cite[Lemma~1.23]{paper29}.
\end{scproof}
Fix $\alpha\in D\cap\Gamma$ and $\beta\in D\cap\alpha$ such that $\beta\notin\bigcup\mathcal C_\alpha$.
As $\beta<\alpha$ and $\alpha\in D$, we know that $\alpha\in D_\epsilon$ for all $\epsilon<\beta$.
So, for all $\epsilon<\beta$, there is $C^\epsilon\in\mathcal C_\alpha$ with $\min(C^\epsilon\setminus\epsilon)<\delta(\epsilon)<\beta$, since $\beta\in D$.
As $\beta$ is a limit ordinal and $\mathcal C_\alpha$ is finite, we may find $C^*\in\mathcal C_\alpha$ such that $C^\epsilon=C^*$ for cofinally many $\epsilon<\beta$.
So, for cofinally many $\epsilon<\beta$, we have $\min(C^*\setminus\epsilon)<\beta$,
and hence $\beta$ is an accumulation point of the club $C^*$ from $\mathcal C_\alpha$, contradicting the fact that $\beta\notin\bigcup\mathcal C_\alpha$.
\end{cproof}
Fix an ordinal $\epsilon<\kappa$ as in the claim. By Fodor's lemma, for every $\delta<\kappa$, we may fix some $x_\delta\in[\kappa\setminus\delta]^{<\omega}$
for which
\[
  \Gamma^\delta:=\left\{\alpha\in\Gamma_{\epsilon,\delta}\setminus(\epsilon+1)\mid
   \{\min(C\setminus\epsilon)\mid C\in\mathcal C_\alpha\}=x_\delta\right\}
\]
is stationary.
Consider the clubs $E:=\{\gamma<\kappa\mid  \forall\delta<\gamma[x_\delta\s\gamma]\}$ and $E':=\acc(E)$.
For every $\alpha<\kappa$, let $$\mathcal C_\alpha^\bullet:=\begin{cases}
\{\emptyset\}&\text{if }\alpha=0;\\
\{\beta\}&\text{if }\alpha=\beta+1;\\
\mathcal C_\alpha&\text{if }\alpha\in\acc(\epsilon+1);\\
\{ (C\setminus\epsilon)\cup\{\sup(\otp(\min(C\setminus\epsilon)\cap E'))\}\mid C\in\mathcal C_\alpha\}&\text{otherwise.}\end{cases}.$$
It is clear that $|\mathcal C_\alpha^\bullet|\le|\mathcal C_\alpha|$ and that each $C^\bullet\in\mathcal C_\alpha^\bullet$ is closed and satisfies $\sup(C^\bullet)=\sup(\alpha)$.
\begin{claim} Suppose that $\alpha<\kappa$, $C^\bullet\in\mathcal C_\alpha^\bullet$ and $\bar\alpha\in\acc(C^\bullet)$. Then $C^\bullet\cap\bar\alpha\in\mathcal C_{\bar\alpha}^\bullet$.
\end{claim}
\begin{cproof} To avoid trivialities, suppose that $\alpha\in\acc(\kappa\setminus\epsilon)$. Pick $C\in\mathcal C_\alpha$
such that $C^\bullet=(C\setminus\epsilon)\cup\{\sup(\otp(\min(C\setminus\epsilon)\cap E'))\}$.
Clearly, $\bar\alpha\in\acc(C\setminus\epsilon)$, so $\bar C:=C\cap\bar\alpha$ is in $\mathcal C_{\bar\alpha}$,
and $C^\bullet\cap\bar\alpha=(\bar C\setminus\epsilon)\cup\{\sup(\otp(\min(\bar C\setminus\epsilon)\cap E'))\}$
is in $\mathcal C^\bullet_{\bar\alpha}$.
\end{cproof}
\begin{claim}
Let $i<\kappa$. Then $\{ \alpha\in\Gamma\mid \forall C^\bullet\in\mathcal C^\bullet_\alpha[\min(C^\bullet)=i]\}$ is stationary.
\end{claim}
\begin{proof} \renewcommand{\qedsymbol}{\ensuremath{\boxtimes \ \square}}
Let $\delta_0$ be the unique element of $E'$ such that $\otp(E'\cap\delta_0)=i$.
Let $\delta_1:=\min(E\setminus(\delta_0+1))$ and $\delta_2:=\min(E'\setminus(\delta_0+1))$,
and note that $\delta_0<x_{\delta_1}<\delta_2$.
Let $\alpha\in\Gamma^{\delta_1}$ be arbitrary.
For each $C\in\mathcal C_\alpha$, we have $\min(C\setminus\epsilon)\in x_{\delta_1}$,
so $$\sup(\otp(\min(C\setminus\epsilon)\cap E'))=\sup(\otp((\delta_0+1)\cap E'))=\sup(i+1)=i.$$
Consequently,  $\{ \alpha\in\Gamma\mid \forall C^\bullet\in\mathcal C^\bullet_\alpha[\min(C^\bullet)=i]\}$ covers the stationary set $\Gamma^{\delta_1}$.
\end{proof}
\let\qed\relax
\end{proof}

\begin{cor} Suppose that $\square(\kappa,{<}\omega)$ holds. Then
for every stationary $\Gamma\s\kappa$, there exists a partition $\langle \Gamma_i\mid i<\kappa\rangle$ of $\Gamma$
into stationary sets such that $\Tr(\Gamma_i)\cap\Tr(\Gamma_j)=\emptyset$ for all $i<j<\kappa$.
\end{cor}
\begin{proof} Let $\Gamma$ be an arbitrary stationary subset of $\kappa$. By Theorem~\ref{thm518},
we may fix a $\square(\kappa,{<}\omega)$-sequence $\langle\mathcal C_\alpha\mid\alpha<\kappa\rangle$
such that, for every $i<\kappa$, $H_i:=\{ \alpha\in\Gamma\mid \forall C\in\mathcal C_\alpha[\min(C)=i]\}$ is stationary.
 As $G:=\Gamma\setminus\bigcup_{i<\kappa} H_i$ has cardinality at most $\kappa$,
it is easy to find a partition $\langle \Gamma_i\mid i<\kappa\rangle$ of $\Gamma$ such that, for all $i<\kappa$,
 $H_i\s \Gamma_i$ and $|H_i\cap G|\le 1$.
To see that  $\langle \Gamma_i\mid i<\kappa\rangle$ is as sought, we are left with verifying the following.

\begin{claim} Let $i<j<\kappa$. Then $\Tr(\Gamma_i)\cap\Tr(\Gamma_j)=\emptyset$.
\end{claim}
  \begin{proof} \renewcommand{\qedsymbol}{\ensuremath{\boxtimes \ \square}}
Suppose not. Fix $\beta\in\Tr(\Gamma_i)\cap\Tr(\Gamma_j)$. Fix a club $C\in\mathcal C_\beta$.
Pick $\alpha_i\in\acc(C)\cap H_i$ and $\alpha_j\in\acc(C)\cap H_j$.
By coherence, we have $C\cap\alpha_i\in\mathcal C_{\alpha_i}$ and $C\cap\alpha_j\in\mathcal C_{\alpha_j}$,
so $i=\min(C\cap\alpha_i)=\min(C)=\min(C\cap\alpha_j)=j$. This is a contradiction.
\end{proof}
\let\qed\relax
\end{proof}

By Lemma~\ref{successorlemma}, $\cf(\lambda)\in\spec(\lambda^+)$ for every infinite cardinal $\lambda$.
In addition, by Corollary~\ref{cor519}(1), $\reg(\lambda)\s\spec(\lambda^+)$ for every infinite regular cardinal $\lambda$.
This, together with Theorem~\ref{t55}, suggests that $\reg(\cf(\lambda))\s\spec(\lambda^+)$ for every singular cardinal $\lambda$.
The following theorem is a step in the right direction.

\begin{thm}\label{thm524} Suppose that $\lambda$ is a singular cardinal.
\begin{enumerate}
\item If $2^{\cf(\lambda)}<\lambda$ or $2^\lambda=\lambda^+$, then $\reg(\cf(\lambda))\s\spec(\lambda^+)$;
\item If $\cf(\lambda)=\mu^+$ is a successor cardinal, then $\reg(\mu)\s\spec(\lambda^+)$.
If, in addition, $2^{<\cf(\lambda)}\le\lambda$, then $\reg(\cf(\lambda))\s\spec(\lambda^+)$.
\end{enumerate}
\end{thm}
\begin{proof}
Let $\theta\in\reg(\cf(\lambda))$ be arbitrary.
By Corollary~\ref{connection}, to show that $\theta\in\spec(\lambda^+)$,
it suffices to prove that there exists a closed witness to $\U(\lambda^+,\lambda^+,\theta,\theta)$.

(1) If $2^{\cf(\lambda)}<\lambda$ or $2^\lambda=\lambda^+$ then by \cite[Theorem~B]{paper34},
there indeed exists a closed witness to $\U(\lambda^+,\lambda^+,\theta,\theta)$.

(2) Suppose that $\nu:=\cf(\lambda)$ is a successor cardinal, say $\nu=\mu^+$.
Recalling the proof of \cite[Theorem~4.21]{paper34},
to prove that there exists a closed witness to $\U(\lambda^+,\lambda^+,\theta,\theta)$,
it suffices to prove that the ideal $\mathcal I$ defined there in ``Case 1: Uncountable cofinality''
is not weakly $\theta$-saturated. For this, let us recall the definition of the ideal $\mathcal I$.
We first fix a stationary subset $\Delta\s E^{\lambda^+}_{\cf(\lambda)}$
  and a sequence $\vec e=\langle e_\delta\mid \delta\in \Delta\rangle$ such that
  \begin{itemize}
    \item for every $\delta\in \Delta$, $e_\delta$ is a club in $\delta$ of order type $\cf(\lambda)$;
    \item for every $\delta\in \Delta$, $\langle \cf(\gamma)\mid \gamma\in \nacc(e_\delta)\rangle$
      is strictly increasing and converging to $\lambda$;
    \item for every club $D$ in $\lambda^+$, there exists $\delta\in \Delta$ such that $e_\delta\s D$.
  \end{itemize}
  Then, the ideal $\mathcal I$ consists of all $\Gamma\s\lambda^+$ for which
  there exists a club $D\s\lambda^+$ such that $\sup(\nacc(e_\delta)\cap D\cap \Gamma)<\delta$
  for every $\delta\in \Delta\cap D$.

The upcoming analysis of the saturation degree of $\mathcal I$ is inspired by the work in \cite[\S3]{paper38}.
Let $\langle A_{i,j}\mid i<\mu, j<\nu\rangle$ be an Ulam matrix over $\nu$, that is:
\begin{itemize}
\item for all $i<\mu$ and $j<\nu$, $A_{i,j}\s\nu$;
\item for all $j<\nu$, $|\nu\setminus\bigcup_{i<\mu}A_{i,j}|\le\mu$;
\item for all $i<\mu$ and $j<j'<\nu$, $A_{i,j}\cap A_{i,j'}=\emptyset$.
\end{itemize}
Fix a club $\Lambda$ in $\lambda$ of order-type $\nu$.
For any subset $A\s\nu$, let $$(A)_\delta:=\{ \gamma\in\nacc(e_\delta)\mid\otp(\Lambda\cap\cf(\gamma))\in A\}.$$

\begin{claim}\label{5221}
There exists $i<\mu$ satisfying the following:
For every club $D\s\lambda^+$, there exist $\delta\in\Delta$ with $e_\delta\s D$ such that $\sup((A_{i,j})_\delta)=\delta$ for cofinally many $j<\nu$.
\end{claim}
\begin{cproof} Suppose not. Then, for every $i<\mu$, fix a club $D_i\s\lambda^+$ with the property that,
for every $\delta\in\Delta$, either $e_\delta\nsubseteq D_i$ or $\sup\{ j<\mu\mid \sup((A_{i,j})_\delta)=\delta\}<\mu$.
Let $D:=\bigcap_{i<\mu}D_i$. Pick $\delta\in\Delta$ such that $e_\delta\s D$.
It follows that, for all $i<\mu$, there exists $j_i<\nu$, such that, for all $j\in(j_i,\nu)$, $\sup((A_{i,j})_\delta)<\delta$.
Let $j:=(\sup_{i<\mu}j_i)+1$. Then, $j<\nu$ and for every $i<\mu$, $\sup((A_{i,j})_\delta)<\delta$.
As $\cf(\delta)=\nu>\mu$, it follows that $\eta:=\sup_{i<\mu}\sup((A_{i,j})_\delta)$ is below $\delta$.
Consequently,
\begin{align*}\sup\{ \gamma\in\nacc(e_\delta)\mid\otp(\Lambda\cap\cf(\gamma))\in \bigcup_{i<\mu}A_{i,j}\}=\\
\sup_{i<\mu}\sup\{ \gamma\in\nacc(e_\delta)\mid\otp(\Lambda\cap\cf(\gamma))\in A_{i,j}\}=\\
\sup_{i<\mu}\sup((A_{i,j})_\delta)=\eta.
\end{align*}
Fix $\varepsilon<\nu$ such that $\varepsilon\cup\bigcup_{i<\mu}A_{i,j}=\nu$.
Pick $\epsilon\in\Lambda$ with $\otp(\Lambda\cap\epsilon)>\varepsilon$.
Finally, pick $\gamma\in\nacc(e_\delta)$ with $\gamma>\eta$ and $\cf(\gamma)>\epsilon$.
As $\otp(\Lambda\cap\cf(\gamma))>\varepsilon$, we have $\otp(\Lambda\cap\cf(\gamma))\in\bigcup_{i<\mu}A_{i,j}$.
Pick $i<\mu$ such that $\otp(\Lambda\cap\cf(\gamma))\in A_{i,j}$.
Then $\gamma\in (A_{i,j})_\delta$, contradicting the fact that $\gamma>\eta$.
\end{cproof}

Let $\mathcal C(\nu,\theta)$ denote the least size of a subfamily $\mathcal C\s[\nu]^\theta$
with the property that for every club $b$ in $\nu$, there is $c\in\mathcal C$ with $c\s b$.
By the third bullet of \cite[Lemma~3.1]{paper38}, if $\theta\in\reg(\mu)$, then $\mathcal C(\nu,\theta)=\nu$.
In addition, it is clear that $\mathcal C(\nu,\theta)\le\nu^\theta$, so that in case $\theta=\mu$,
$\nu^\theta=(\mu^+)^\mu=2^\mu=2^{<\cf(\lambda)}$. Thus, in any case, we may
let $\mathcal C$ be a $\lambda$-sized subfamily of $\{ c\s\nu\mid \otp(c)=\theta\}$ with the property that for every club $b$ in $\nu$, there is $c\in\mathcal C$ with $c\s b$.

Let $i^*$ be given by Claim~\ref{5221}.
For every $c\in\mathcal C$, define a function $h_c:\lambda^+\rightarrow\theta$, as follows.
Given $\gamma<\lambda^+$, if there exists $j<\sup(c)$ such that $\otp(\Lambda\cap\cf(\gamma))\in A_{i^*,j}$,
then $j$ is unique, and we let $h_c(\gamma):=\sup(\otp(c\cap j))$. Otherwise, let $h_c(\gamma):=0$.

For every $c\in\mathcal C$ and every $\tau<\theta$, let $\Gamma_c^\tau:=\{\gamma<\lambda^+\mid h_c(\gamma)=\tau\}$.

\begin{claim} There exists $c\in\mathcal C$ such that, for all $\tau<\theta$, $\Gamma_c^\tau$ is $\mathcal I$-positive.
\end{claim}
\begin{proof} Suppose not. Then, for every $c\in\mathcal C$, we may find $\tau_c<\theta$ and a club $D_c\s\lambda^+$
such that $\sup(\nacc(e_\delta)\cap D_c\cap\Gamma_c^{\tau_c})<\delta$ for every $\delta\in\Delta\cap D_c$.
Let $D:=\bigcap_{c\in\mathcal C}D_c$.
By the choice of $i^*$, let us pick $\delta\in\Delta$ with $e_\delta\s D$ for which $J:=\{j<\nu\mid \sup((A_{i^*,j})_\delta)=\delta\}$ is cofinal in $\nu$.
In particular, $\delta\in\acc(D)\s D$.

As $\acc^+(J)$ is a club in $\nu$, we may find $c\in\mathcal C$ with $c\s\acc^+(J)$.
As $\otp(c)=\theta>\tau_c$, we may let $j'$ denote the unique element of $c$ to satisfy $\otp(c\cap j')=\tau_c$.
Now, let $j:=\min(J\setminus(j'+1))$.
As $c\s\acc^+(J)$, we know that $[j',j)\cap c=\{j'\}$, so that $\otp(c\cap j)=\otp(c\cap(j'+1))=\tau_c+1$.
As $c\s\acc^+(J)$, we also know that $j<\sup(c)$.

As $\delta\in\Delta\cap D\s\Delta\cap D_c$, we infer that $\sup(\nacc(e_\delta)\cap D_c\cap \Gamma^{\tau_c}_c)<\delta=\sup((A_{i^*,j})_\delta)$.
so we may pick $\gamma\in (A_{i^*,j})_\delta$ above $\sup(\nacc(e_\delta)\cap D_c\cap \Gamma^{\tau_c}_c)$.
Recall that the former means that $\gamma\in\nacc(e_\delta)$ and $\otp(\Lambda\cap\cf(\gamma))\in A_{i^*,j}$.
In effect, $h_c(\gamma)=\sup(\otp(c\cap j))=\sup(\tau_c+1)=\tau_c$. So, $\gamma\in\nacc(e_\delta)\cap \Gamma^{\tau_c}_c$.
Recalling that $e_\delta\s D\s D_c$, we infer that $\gamma\in\nacc(e_\delta)\cap D_c \cap\Gamma^{\tau_c}_c$,
contradicting the choice of $\gamma$ to be above $\sup(\nacc(e_\delta)\cap D_c\cap \Gamma^{\tau_c}_c)$.
\end{proof}

Let $c$ be given the preceding claim. Then $\langle \Gamma^\tau_c\mid \tau<\theta\rangle$ is a partition of $\lambda^+$
into $\theta$ many $\mathcal I$-positive sets, witnessing that $\mathcal I$ is indeed not weakly $\theta$-saturated.
\end{proof}

\section{Concluding remarks}

We end with some some questions that remain open,
followed by a couple of brief remarks connecting the topics of this paper with previous works.
First, we present a conjecture of a connection between the $C$-sequence number and
the infinite productivity of the $\kappa$-Knaster property.

\begin{conj} For any regular uncountable cardinal $\kappa$, the following are equivalent:
\begin{itemize}
\item $\chi(\kappa) \leq 1$;
\item for every $\kappa$-Knaster poset $\mathbb P$, $\mathbb P^\omega$ is $\kappa$-Knaster, as well.
\end{itemize}
\end{conj}

A number of fundamental questions about the $C$-sequence number and
$C$-sequence spectrum remain open. In addition to the above conjecture, here
are a few that we find especially interesting. The first set of questions concerns
the structure of $\spec(\kappa)$.

\begin{q}
  Must $\spec(\kappa)$ be an interval? Must $\spec(\kappa)$ be closed?
  If $\theta \in \reg(\kappa)$ and $\theta^+ \in \spec(\kappa)$, must we have
  $\theta \in \spec(\kappa)$?
  If $\theta\in\spec(\kappa)$ is an uncountable limit cardinal, must it be an
  accumulation point of $\spec(\kappa)$?
\end{q}

Our next question deal with the connections between the $C$-sequence spectrum
and $\U(\ldots)$.

\begin{q}
  Must it be the case that $\U(\kappa, \kappa, \chi(\kappa), \sup(\reg(\kappa)))$
  holds?
  If $\theta, \chi \in \spec(\kappa)$, must it be the case that
  $\U(\kappa, \kappa, \theta, \chi)$ holds?
\end{q}

The next question deal with the connections between the $C$-sequence numbers and $\kappa$-Aronszajn trees.

\begin{q}
  Suppose $\kappa$ is strongly inaccessible.
  If $\chi(\kappa) = 1$, must there be a coherent $\kappa$-Aronszajn tree?
  If $1 < \chi(\kappa) < \kappa$, must there
  be a $\kappa$-Aronszajn tree with a $\chi(\kappa)$-ascent path?
\end{q}

Another question has to do with a singular value for the $C$-sequence number.
\begin{q} Suppose that $\chi(\kappa)$ is singular. Must it be the case that $\cf(\chi(\kappa))=\cf(\sup(\reg(\kappa)))$?
\end{q}
The main result of \cite{paper15} states that if $\theta$ and $\kappa$ are regular cardinals, $\kappa>\theta^+$, and $E^\kappa_{\ge\theta}$ admits a non-reflecting stationary set,
then $\pr_1(\kappa,\kappa,\kappa,\theta)$ holds.
Theorem~\ref{thm51} above shows that this result is optimal.

\begin{cor}
  Suppose that $\kappa$ is a weakly compact cardinal, and $\theta\in\reg(\kappa)$.
  Then there is a cofinality-preserving forcing extension in which:
  \begin{enumerate}
    \item $\kappa$ is strongly inaccessible;
    \item $E^\kappa_\theta$ admits a non-reflecting stationary set, so $\pr_1(\kappa,\kappa,\kappa,\theta)$ holds;
    \item $\pr_1(\kappa, \kappa, \theta^+, \theta^+)$ fails.
  \end{enumerate}
\end{cor}
\begin{proof} Work in the model of Theorem~\ref{thm51}. By its Clause~(6), $\pr_1(\kappa,\kappa,\theta^+,\theta^+)$ fails.
\end{proof}

Finally, we note that the combination of Fact~\ref{rho2_closed}, Lemma~\ref{lemma33} and Theorem~\ref{t55} implies that,
in the statement of Theorem~6.3.6 of \cite{todorcevic_book}, ``of size $<\kappa$'' should have been ``of size $<\cf(\kappa)$''.

\section*{Acknowledgments}
The results of this paper were presented by the first author at the CUNY Set Theory Seminar in March 2019 and
by the second author at the \emph{The 15th International Luminy Workshop in Set Theory}, held in September 2019.
We thank the organizers for the warm hospitality.

\end{document}